\documentclass[12pt]{amsart}

\usepackage{amsmath}
\usepackage{amsthm}
\usepackage{caption}
\usepackage[curve]{xypic}
\usepackage{cite}
\usepackage{enumitem}
\usepackage{color}
\usepackage{url}
\usepackage[usenames,dvipsnames]{xcolor}
\usepackage{graphicx}
\usepackage{tikz}
\usepackage{tikz-cd}
\usepackage{hyperref} 
\usepackage{verbatim}

\setlist{itemsep=4pt, topsep=4pt}

\makeatletter

\def\chaptermark#1{}

\def\chapter{%
  \if@openright\cleardoublepage\else\clearpage\fi
  \thispagestyle{plain}\global\@topnum\z@
  \@afterindenttrue \secdef\@chapter\@schapter}

\def\@chapter[#1]#2{\refstepcounter{chapter}%
  \ifnum\c@secnumdepth<\z@ \let\@secnumber\@empty
  \else \let\@secnumber\thechapter \fi
  \typeout{\chaptername\space\@secnumber}%
  \def\@toclevel{0}%
  \ifx\chaptername\appendixname \@tocwriteb\tocappendix{chapter}{#2}%
  \else \@tocwriteb\tocchapter{chapter}{#2}\fi
  \chaptermark{#1}%
  \addtocontents{lof}{\protect\addvspace{10\p@}}%
  \addtocontents{lot}{\protect\addvspace{10\p@}}%
  \@makechapterhead{#2}\@afterheading}
\def\@schapter#1{\typeout{#1}%
  \let\@secnumber\@empty
  \def\@toclevel{0}%
  \ifx\chaptername\appendixname \@tocwriteb\tocappendix{chapter}{#1}%
  \else \@tocwriteb\tocchapter{chapter}{#1}\fi
  \chaptermark{#1}%
  \addtocontents{lof}{\protect\addvspace{10\p@}}%
  \addtocontents{lot}{\protect\addvspace{10\p@}}%
  \@makeschapterhead{#1}\@afterheading}
\newcommand\chaptername{Chapter}

\def\@makechapterhead#1{\global\topskip 7.5pc\relax
  \begingroup
  \fontsize{\@xivpt}{18}\bfseries\centering
    \ifnum\c@secnumdepth>\m@ne
      \leavevmode \hskip-\leftskip
      \rlap{\vbox to\z@{\vss
          \centerline{\normalsize\mdseries
              \uppercase\@xp{\chaptername}\enspace\thechapter}
          \vskip 3pc}}\hskip\leftskip\fi
     #1\par \endgroup
  \skip@34\p@ \advance\skip@-\normalbaselineskip
  \vskip\skip@ }
\def\@makeschapterhead#1{\global\topskip 7.5pc\relax
  \begingroup
  \fontsize{\@xivpt}{18}\bfseries\centering
  #1\par \endgroup
  \skip@34\p@ \advance\skip@-\normalbaselineskip
  \vskip\skip@ }
\def\appendix{\par
  \c@chapter\z@ \c@section\z@
  \let\chaptername\appendixname
  \def\thechapter{\@Alph\c@chapter}}

\newcounter{chapter}

\newif\if@openright

\makeatother

\makeatletter 
\def\@cite#1#2{{\m@th\upshape\bfseries%
[{#1\if@tempswa{\m@th\upshape\mdseries, #2}\fi}]}}
\makeatother 

\theoremstyle{plain}
\newtheorem{thm}{Theorem}[section]
\newtheorem{cor}[thm]{Corollary}
\newtheorem{ass}[thm]{Assumption}
\newtheorem{prop}[thm]{Proposition}
\newtheorem{lem}[thm]{Lemma}
\newtheorem{sublem}[thm]{Sublemma}
\theoremstyle{definition}
\newtheorem{defn}[thm]{Definition}
\newtheorem{war}[thm]{Warning}
\newtheorem{ex}[thm]{Example}

\newtheorem{prob}[thm]{Problem}
\newtheorem{conj}[thm]{Conjecture}
\theoremstyle{remark}
\newtheorem{rem}[thm]{Remark}

\numberwithin{equation}{subsection}
\renewcommand{\bold}[1]{\medskip \noindent {\bf #1 }\nopagebreak}

\newcommand{\nc}{\newcommand}
\newcommand{\rnc}{\renewcommand}

\nc\bA{\mathbb{A}}
\nc\bB{\mathbb{B}}
\nc\bC{\mathbb{C}}
\nc\bD{\mathbb{D}}
\nc\bE{\mathbb{E}}
\nc\bF{\mathbb{F}}
\nc\bG{\mathbb{G}}
\nc\bH{\mathbb{H}}
\nc\bI{\mathbb{I}}
\nc{\bJ}{\mathbb{J}} 
\nc\bK{\mathbb{K}}
\nc\bL{\mathbb{L}}
\nc\bM{\mathbb{M}}
\nc\bN{\mathbb{N}}
\nc\bO{\mathbb{O}}
\nc\bP{\mathbb{P}}
\nc\bQ{\mathbb{Q}}
\nc\bR{\mathbb{R}}
\nc\bS{\mathbb{S}}
\nc\bT{\mathbb{T}}
\nc\bU{\mathbb{U}}
\nc\bV{\mathbb{V}}
\nc\bW{\mathbb{W}}
\nc\bY{\mathbb{Y}}
\nc\bX{\mathbb{X}}
\nc\bZ{\mathbb{Z}}
\nc\cA{\mathcal{A}}
\nc\cB{\mathcal{B}}
\nc\cC{\mathcal{C}}
\rnc\cD{\mathcal{D}}
\nc\cE{\mathcal{E}}
\nc\cF{\mathcal{F}}
\nc\cG{\mathcal{G}}
\rnc\cH{\mathcal{H}}
\nc\cI{\mathcal{I}}
\nc{\cJ}{\mathcal{J}} 
\nc\cK{\mathcal{K}}
\rnc\cL{\mathcal{L}}
\nc\cM{\mathcal{M}}
\nc\cN{\mathcal{N}}
\nc\cO{\mathcal{O}}
\nc\cP{\mathcal{P}}
\nc\cQ{\mathcal{Q}}
\rnc\cR{\mathcal{R}}
\nc\cS{\mathcal{S}}
\nc\cT{\mathcal{T}}
\nc\cU{\mathcal{U}}
\nc\cV{\mathcal{V}}
\nc\cW{\mathcal{W}}
\nc\cY{\mathcal{Y}}
\nc\cX{\mathcal{X}}
\nc\cZ{\mathcal{Z}}
\nc\bfA{\mathbf{A}}
\nc\bfB{\mathbf{B}}
\nc\bfC{\mathbf{C}}
\nc\bfD{\mathbf{D}}
\nc\bfE{\mathbf{E}}
\nc\bfF{\mathbf{F}}
\nc\bfG{\mathbf{G}}
\nc\bfH{\mathbf{H}}
\nc\bfI{\mathbf{I}}
\nc{\bfJ}{\mathbf{J}} 
\nc\bfK{\mathbf{K}}
\nc\bfL{\mathbf{L}}
\nc\bfM{\mathbf{M}}
\nc\bfN{\mathbf{N}}
\nc\bfO{\mathbf{O}}
\nc\bfP{\mathbf{P}}
\nc\bfQ{\mathbf{Q}}
\nc\bfR{\mathbf{R}}
\nc\bfS{\mathbf{S}}
\nc\bfT{\mathbf{T}}
\nc\bfU{\mathbf{U}}
\nc\bfV{\mathbf{V}}
\nc\bfW{\mathbf{W}}
\nc\bfY{\mathbf{Y}}
\nc\bfX{\mathbf{X}}
\nc\bfZ{\mathbf{Z}}

\newcommand{\bk}{{\mathbf{k}}}

\nc{\dmo}{\DeclareMathOperator}
\nc{\wt}{\widetilde}
\rnc{\Re}{\operatorname{Re}}
\rnc{\Im}{\operatorname{Im}}
\rnc{\span}{\operatorname{span}}
\dmo{\rank}{rank}
\dmo{\End}{End}
\dmo{\Hom}{Hom}
\dmo{\Jac}{Jac}
\dmo{\Id}{Id}
\dmo{\Ann}{Ann}
\dmo{\Area}{Area}
\dmo{\CP}{\bC P^1}
\dmo{\rk}{rk}
\dmo{\rel}{rel}
\dmo{\ra}{\rightarrow}

\rnc{\Col}{\operatorname{Col}}

\nc{\ColOne}{\Col_{\bfC_1}}
\nc{\ColOneX}{\ColOne(X,\omega)}

\nc{\ColTwo}{\Col_{\bfC_2}}
\nc{\ColTwoX}{\ColTwo(X,\omega)}

\nc{\ColThree}{\Col_{\bfC_3}}
\nc{\ColThreeX}{\ColThree(X,\omega)}

\nc{\ColOneTwo}{\Col_{\bfC_1, \bfC_2}}
\nc{\ColOneTwoX}{\ColOneTwo(X,\omega)}

\nc{\ColOneThree}{\Col_{\bfC_1, \bfC_3}}
\nc{\ColOneThreeX}{\ColOneThree(X,\omega)}

\nc{\MOne}{\cM_{\bfC_1}}
\nc{\MTwo}{\cM_{\bfC_2}}
\nc{\MOneTwo}{\cM_{\bfC_1, \bfC_2}}

\nc{\MThree}{\cM_{\bfC_3}}

\nc{\MOneThree}{\cM_{\bfC_1, \bfC_3}}

\dmo{\For}{\cF}

\nc{\GL}{\mathrm{GL}^+(2, \bR)}

\renewcommand{\color}[1]{\unskip}

\title{Generalizations of the Eierlegende-Wollmilchsau}
\author[Apisa]{Paul~Apisa}
\author[Wright]{Alex~Wright}

\begin{document}
\maketitle
\thispagestyle{empty}


\begin{abstract}
We  classify a natural collection of $GL(2,\bR)$-invariant subvarieties, which includes loci of double covers, the orbits of the Eierlegende-Wollmilchsau, Ornithorynque, and Matheus-Yoccoz surfaces, and loci appearing naturally in the study of the complex geometry of Teichm\"uller space. 
This classification  is the key input in subsequent work of the authors that classifies ``high rank" invariant subvarieties, and in subsequent work of the first author that classifies certain invariant subvarieties with ``Lyapunov spectrum as degenerate as possible". 
We also derive applications to the complex geometry of Teichm\"uller space and construct new examples, which negatively resolve two questions of Mirzakhani and Wright and illustrate previously unobserved phenomena for the finite blocking problem. 
\end{abstract}


\setcounter{tocdepth}{1} 
\tableofcontents
\newpage
 \section{Introduction}\label{S:intro}

\subsection{New exceptional surfaces}
The Eierlegende-Wollmilchsau\footnote{German for egg-laying wool-milk-sow. Colloquially, an all-in-one device that can do the work of several specialized tools.}  square-tiled surface, independently studied by Forni \cite{Fsur} and Herrlich-Schmith\"{u}sen \cite{HS}, is a perpetual counterexample in the study of translation surfaces and the $GL(2,\bR)$--action. On this surface each cylinder is parallel to exactly one other cylinder, and these pairs of cylinders are homologous and isometric.  (Two cylinders are called homologous if their core curves are homologous.) Its relatives, namely the Ornithorynque\footnote{French for Platypus (the egg laying mammal).} studied by Forni-Matheus \cite{FM}, and the infinite sequence of surfaces studied by Matheus-Yoccoz \cite{MY}, share these properties, and are also continuing sources of insight. 

In this paper we consider generalizations of these surfaces that have higher dimensional $GL(2,\bR)$--orbit closures.  We give examples of orbit closures such that on almost every surface in the orbit closure, each cylinder is parallel to exactly one other cylinder, and these pairs of cylinders are homologous and isometric. As is the case for the Ornithorynque, in many of our examples there is moreover an involution exchanging the cylinders and negating the Abelian differential. Quotienting by this involution, we arrive at quadratic differentials with at most one cylinder in each direction, but whose orbit closure $\cM$ is not a connected component of a stratum.  

Our examples, which arise as loci of cyclic covers, negatively resolve two questions of Mirzakhani and Wright \cite[Questions 1.7 and 1.8]{MirWri2} (Section \ref{SS:Fake}), clarify the complex geometry of Teichm\"uller space (Section \ref{SS:Retracts}), and illustrate new behavior relevant to the finite blocking problem (Section \ref{SS:Blocking}). 

We now turn to our main result, whose investigation led directly to the examples above.  

\subsection{Main result} Our main result classifies orbit closures in which cylinders always occur individually or in isometric pairs. More precisely, let $\cM$ be a $GL(2,\bR)$-invariant subvariety of a stratum of translation surfaces. We say $\cM$ is \emph{geminal}, from the Latin for twins, if, for any cylinder $C$ on any $(X,\omega)\in \cM$, either 
\begin{itemize}
\item any cylinder deformation of $C$ remains in $\cM$, or 
\item there is a cylinder $C'$ such that $C$ and $C'$ are parallel and have the same height and circumference on $(X,\omega)$ as well as on all small deformations of $(X,\omega)$ in $\cM$, and any cylinder deformation that deforms $C$ and $C'$ equally remains in $\cM$. 
\end{itemize}

In the first case we say that $C$ is free, and in the second case we say $C$ and $C'$ are twins.  Section \ref{S:Pre} contains details on cylinder deformations, but we note for the moment that a cylinder deformation that deforms $C$ and $C'$ equally can be defined concretely as a deformation $(X', \omega')$ for which there exists a PL-homeomorphism $(X,\omega)\to (X', \omega')$ whose derivative is the identity off $C\cup C'$ and whose derivative is constant on $C \cup C'$. Our generalizations of the Eierlegende-Wollmilchsau give examples of geminal orbit closures (see Definition \ref{D:EW} and the subsequent examples). 

There are a few obvious examples of geminal orbit closures arising from loci of double covers, which we introduce after a few preliminary remarks. To start, it is important to note that we allow our surfaces to have marked points, which we treat as ``zeros of order zero".

For any cover, and any small deformation of the base, one obtains a deformation of the cover. Let $\cM$ and $\cN$ be invariant subvarieties of Abelian or quadratic differentials. We say that $\cM$ is a \emph{full} locus of covers of $\cN$ if every surface in $\cM$ is a  cover of a surface in $\cN$ in such a way that all deformations  in $\cN$  of the codomain  translation surface give rise to covers in $\cM$. We recall the definition of (branched) covering maps between Abelian and quadratic differentials in Definition \ref{D:Covering}.

We define an \emph{Abelian double} to be a full locus of covers of a component of a stratum of Abelian differentials such that the covering maps have degree two, and all preimages of marked points are either singularities or marked points. 

We define a \emph{quadratic double} to be a full locus of covers of a component of a stratum of quadratic differentials such that the covering maps are the holonomy double cover and all preimages of marked points are marked points. The preimage of a pole may be marked or unmarked. 

In both Abelian and quadratic doubles, we assume that the double covers are connected and do not have any marked points that map to unmarked points. Note that there may be more than one double associated to a component of stratum, since in the Abelian case there may be more than one choice of degree two cover, and in the quadratic case there may be more than one choice of which preimages of poles to mark.  

Abelian and quadratic doubles are both examples of geminal orbit closures. Our main result almost completely classifies geminal orbit closures.

\begin{thm}\label{T:geminal}
Any geminal orbit closure that doesn't consist of branched covers of tori is one of the following. 
\begin{enumerate}
\item A  component of a stratum of Abelian differentials. 
\item An Abelian or quadratic double.
\item A full locus of covers of a quadratic double of a genus zero stratum.
\end{enumerate}
\end{thm}

See Theorem \ref{T:geminal2} for a more detailed statement. Our generalizations of the Eierlegende-Wollmilchsau give many examples of the final type. However, we do not classify which loci of covers of quadratic doubles of genus zero strata  give geminal orbits closures. That remaining problem seems largely group theoretic (see the conjecture and open problem in Section \ref{SS:Open}) and is not required for our applications.

We also give a similar result for loci of branched covers of tori, where there are additional possibilities associated to loci of surfaces with $\bZ/2 \times \bZ/2$ symmetry. 

Our proof of Theorem \ref{T:geminal} is inductive. Absolutely key to the induction is that we prove a rather involved stronger statement, namely Theorem \ref{T:geminal2}, so that the inductive hypothesis is more powerful. The main tool in the proof is the theory of ``diamonds" developed in \cite{ApisaWrightDiamonds}, which allow us to prove that $\cM$ is a locus of covers by simultaneously using two degenerations. The proof also makes use of restrictions on the structure of the boundary developed in \cite{ChenWright}. This is the first paper to apply these results from \cite{ApisaWrightDiamonds, ChenWright}. 

\subsection{Main applications}
Theorem \ref{T:geminal} plays a major role in the classification of high rank invariant subvarieties \cite{ApisaWrightHighRank}.  There we desire to conclude that a  high rank invariant subvariety is a double or a component of a stratum, and Theorem \ref{T:geminal} resolves the difficulty that other invariant subvarieties might share many properties with a double from the point of view of cylinder deformations. 

Theorem \ref{T:geminal} also plays a major role in the first author's study of invariant subvarieties whose Lyapunov spectrum contains as many zero exponents as possible \cite{Apisa-MHD}. A key element of that study is the following result, which provides further evidence for the surprisingly broad applicability of geminal subvarieties. 

\begin{thm}[Apisa]
Suppose $\cM$ is an invariant subvariety that does not consist entirely of branched covers of tori. If, on generic surfaces in $\cM$,  all parallel cylinders are homologous, then $\cM$ can be obtained from a geminal subvariety by forgetting some marked points. 
\end{thm}

This result combines \cite[Theorems 1.5 and 1.7]{Apisa-MHD}, and relates to zero exponents via results of Forni \cite{F, Fcriterion}.

\subsection{Fake strata}\label{SS:Fake}
An invariant subvariety $\cM$ is called \emph{free} if every cylinder on every surface can be deformed without leaving $\cM$. In other words, $\cM$ is free if it is geminal and no cylinder  has a twin. The only obvious examples are components of strata, and any other such $\cM$ might be called a ``fake stratum", in that the flexibility to deform cylinders in $\cM$ matches that of a stratum. 

Mirzakhani and the second author previously proved that every free invariant subvariety of translation surfaces is a connected component of a stratum, and used this in their classification of invariant subvarieties of maximal rank \cite{MirWri2}. 

However, our examples show that there are many invariant subvarieties of quadratic differentials that are not a connected component of a stratum and in which every cylinder is free. This resolves \cite[Question 1.8]{MirWri2}, which  asked if such an invariant subvariety with rank bigger than 1 exists.  

\begin{figure}[h]
\includegraphics[width=\linewidth]{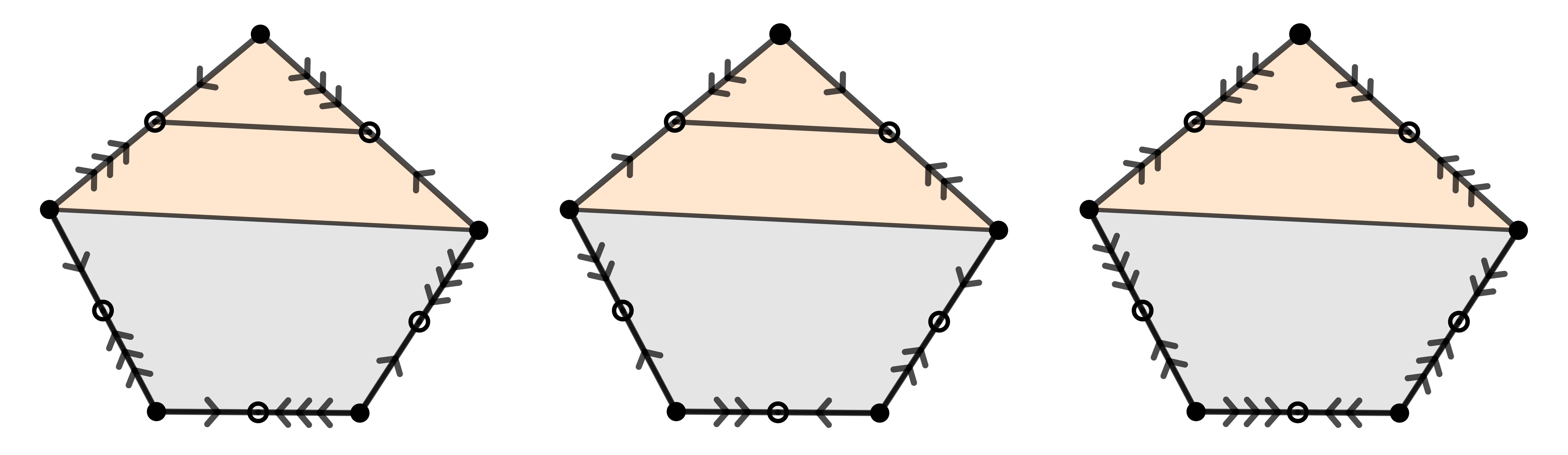}
\caption{A surface in a 4-dimensional invariant subvariety of genus 4 quadratic differentials in which every cylinder is free, contained in the 12-dimensional stratum $\cQ(7, 1^5)$.  The region shaded orange is a single cylinder.  }
\label{F:Deg3Rk2EW}
\end{figure}

An example of a non-trivial invariant subvariety of quadratic differentials of rank 2 in which every cylinder is free is obtained as the loci of three-fold covers of surfaces in $\cQ(1,-1^5)$ illustrated in Figure \ref{F:Deg3Rk2EW}.  The corresponding locus of holonomy double covers of these quadratic differentials is one of the smallest examples of a rank two generalization of the Eierlegende-Wollmilchsau constructed in this paper. (Using the notation of Section \ref{S:NewEW}, it is $\cE(\kappa, a)$ with $\kappa=(1, -1^5)$ and $a=(1^6)$.)

\subsection{The complex geometry of Teichm\"uller space}\label{SS:Retracts}  Markovic \cite{Markovic18} recently proved a longstanding conjecture of Siu \cite[Conjecture 3.7]{Siu} by showing that the Teichm\"uller space of a closed surface of genus $g\geq2$ is not biholomorphic to a bounded convex subset of  $\bC^{3g-3}$. This was accomplished by producing Teichm\"uller disks that are not holomorphic retracts of Teichm\"uller space.\footnote{Markovic showed that the  Caratheodory and  Kobayashi metrics are not equal on Teichm\"uller space. As Markovic has pointed out to us, these metrics are known to be equal on any convex domain in $\bC^{3g-3}$, even unbounded ones \cite[Lemma 3.1]{BracciSaracco}, so the boundedness assumption is not required in Siu's conjecture.}

Given a component $\cQ$ of a stratum of quadratic differentials, let $\cM_{\cQ}$ denote the collection of quadratic differentials $(X, q) \in \cQ$ so that the Teichm\"uller disk generated by $(X, q)$ is a holomorphic retract of Teichm\"uller space. By \cite[Lemma 2.2, Section 2.4]{GekhtmanMarkovic}, $\cM_{\cQ}$ is a union of invariant subvarieties. Moreover, Gekhtman and Markovic show that $\cM_{\cQ}$ is free \cite[Corollary 4.4]{GekhtmanMarkovic}\footnote{The exact statement of \cite[Corollary 4.4]{GekhtmanMarkovic} is that every horizontal cylinder is free on every horizontally periodic surface in $\cM_{\cQ}$. Given a cylinder on a surface in $\cM_{\cQ}$, one can rotate so that it is horizontal, and use the work of Minsky, Smillie, and Weiss \cite{MinW, SW2} discussed in \cite[Section 2.4]{GekhtmanMarkovic} to pass to a horizontally periodic surface. Using \cite[Corollary 4.4]{GekhtmanMarkovic}, this implies the original cylinder was free, as in \cite[Proof of Theorem 5.1]{Wcyl}.}; so the associated locus of holonomy double covers are geminal. The following is therefore a direct consequence of Theorem \ref{T:geminal}.

\begin{cor}\label{C:ComplexGeometry}
$\cM_{\cQ}$ is one of the following: empty, $\cQ$, or a union of full loci of covers of strata of genus zero quadratic differentials. 
\end{cor}

When all of the zeros on quadratic differentials in $\cQ$ have even order, $\cM_{\cQ} = \cQ$ by work of Kra \cite{KraCaratheodory}, see also McMullen \cite[Theorem 4.1]{Mc}. Gekhtman and Markovic conjecture \cite[Conjecture 1.1]{GekhtmanMarkovic} that $\cM_{\cQ}$ is empty when the quadratic differentials in $\cQ$ have an odd order zero. Corollary \ref{C:ComplexGeometry} reduces this conjecture to checking strata and loci of covers of genus zero quadratic differentials.

\bold{Acknowledgments.}  We are grateful to David Aulicino, Johannes Flake, and Andrea Thevis for helpful conversations relating to Section \ref{SS:Open},  Dmitri Gekhtman  and Vlad Markovic for helpful conversations relating to Section \ref{SS:Retracts},  and the referee for detailed and helpful comments.  During the preparation of this paper, the first author was partially supported by NSF Postdoctoral Fellowship DMS 1803625, and the second author was partially supported by a Clay Research Fellowship,  NSF Grant DMS 1856155, and a Sloan Research Fellowship.

\section{New Eierlegende-Wollmilchsaus}\label{S:NewEW}

\subsection{Exceptional loci of cyclic covers}\label{SS:Cyclic} For $\kappa=(\kappa_1, \ldots, \kappa_s)$, we denote by $\cQ_P(\kappa)$ the finite cover of the stratum $\cQ(\kappa)$ where the zeros are labeled $1, \ldots, s$ in such a way that smaller order zeros get smaller labels. (The ``P" stands for ``Pure", in analogy to the pure braid group.) 

Let $\cQ_P(\kappa)$ be a genus $0$ stratum, and let $k>1$. For any $$a=(a_1, \ldots, a_s)\in (\bZ/k)^s$$ such that $\sum a_i=0$ and such that the $a_i$ generate $\bZ/k$, define the locus of cyclic covers $\cE(\kappa,a)$ to be the set of $(X,q)$ arising as $(\bZ/k)$ covers of surfaces in $\cQ_P(\kappa)$, branched over zeros and poles, with local monodromy given by the $a_i$. 

These covers can be described as follows. First, puncture the genus 0 surface $\bP^1$ at the set $\Sigma$ of zeros and poles of the quadratic differential. Then, define a map $\pi_1(\bP^1\setminus \Sigma)\to \bZ/k$ by mapping the loop around the $i$-th puncture to the corresponding $a_i$. This map defines an unbranched cover $X'\to \bP^1\setminus \Sigma$. Filling in punctures gives a branched cover $X\to \bP^1$, and $\cE(\kappa,a)$ is defined to be the locus of $X$ arising in this way endowed with the pull back quadratic differentials. 

\begin{lem}\label{L:Hol}
Let $k=2\ell$ for some integer $\ell$, and assume $\kappa_i$ and $a_i$ have the same parity for each $i$. Then all the quadratic differentials in $\cE(\kappa,a)$ are squares of Abelian differentials. 
\end{lem}

Thus, in this situation we can consider $\cE(\kappa,a)$ as a locus of Abelian differentials.

\begin{proof}
The holonomy representation of a genus zero quadratic differential is the unique map from $\pi_1$ of the punctured sphere to $\bZ/2$ such that the loop around each odd order zero maps to $1$ and the loop around each even order zero maps to $0$. The given conditions thus imply that the natural map 
$$\pi_1(\bP^1\setminus \Sigma)\to \bZ/(2\ell)\to \bZ/2$$
is the holonomy representation, so the cyclic covers factor through the holonomy double cover. 
\end{proof}

\begin{rem}
$\cE(\kappa,a)$ is $GL(2,\bR)$ invariant since $\cQ_P(\kappa)$ is. The rank and dimension are the same as $\cQ_P(\kappa)$; the rank is $\frac12 m_{odd}-1$, where $m_{odd}$ is the number of $i$ for which $\kappa_i$ is odd; and the dimension is $s-2$. See, for example, \cite[Lemma 4.2]{ApisaWright} or \cite[Lemma 4.4]{ApisaWrightDiamonds} for the formulas for the rank of a stratum, and see Section \ref{S:Pre} for the definition of rank. 
\end{rem}

\begin{defn}\label{D:EW} A generalized Eierlegende-Wollmilchsau is a locus $\cE(\kappa,a)$, where 
\begin{itemize}
\item $k=2\ell$ for some $\ell>1$, 
\item $\kappa_i$ and $a_i$ have the same parity for each $i$,  and
\item for each  $I\subset\{1, \ldots, s\}$ such that $\sum_{i\in I} \kappa_i=-2$, the sum $\sum_{i\in I} a_i$ is a generator for $\bZ/\ell\subset \bZ/(2\ell)$.
\end{itemize}
\end{defn}

 Keeping in mind Lemma \ref{L:Hol}, each generalized Eierlegende-Wollmilchsau is viewed as a locus of Abelian differentials.  

\begin{rem}
Since $\sum \kappa_i$ is even and $\sum a_i$ is even, it suffices to check the parity condition for all but one $i$. 
\end{rem}

%

\begin{ex}
Assume $\kappa_i=-1$ for $i>1$. (Up to permuting the $\kappa_i$, this is equivalent to a quadratic double of $\cQ(\kappa)$ being a hyperelliptic connected component of a stratum of Abelian differentials, so this can be called the hyperelliptic case.) 
In this case, the second condition is that $a_i+a_j$ is a generator for $\bZ/\ell\subset \bZ/(2\ell)$ for each $i>j>1$. We can take $a_i=1$ for $i>1$ and $a_1=2\ell-(s-1)$. 
\end{ex}

\begin{ex}
For any $\kappa$, let $a_i$ be equal to $1$ if $i>1$ and $\kappa_i$ is odd, and $0$ if $i>1$ and $\kappa_i$ is even. Let $\ell>s$ be prime, and let $a_1=2\ell-(\sum_{i>1}a_i)$. Note that since $\sum \kappa_i = -4$, there must be at least one $\kappa_i$ that is $-1$, so in particular there must be at least one $\kappa_i$ that is odd. 
\end{ex}

\begin{ex}
Using $\kappa=(-1,-1,-1,-1)$ gives (orbits of) square-tiled surfaces. 
\begin{itemize}
\item The original Eierlegende-Wollmilchsau corresponds to $\ell=2$ and $a=(1,1,1,1)$ \cite[Equation (59)]{Fsur}, \cite[Proposition 1.5]{HS}, \cite[Equation (1.2)]{MY}. 
\item The Ornithorynque corresponds to $\ell=3$ and $a=(3,1,1,1)$ \cite[Proof of Theorem 1.3]{FM}, \cite[Equation (1.3)]{MY}. 
\item The Matheus-Yoccoz examples correspond to $\ell\geq 3$ odd, and $a=(\ell-2,1,1,\ell)$ \cite[Remark 3.1]{MY}. 
\end{itemize}
Our result gives many more (orbits of) square-tiled surface that haven't been previously studied. 
\end{ex}

The main result of this section is the following. 

\begin{thm}\label{T:NewExamples}
For each cylinder $C$ on each $(X,\omega)$ in a generalized Eierlegende-Wollmilchsau $\cE(\kappa,a)$, there is a cylinder $C'$  isometric to $C$ such that all cylinder deformations that equally deform $C$ and $C'$ remain in $\cE(\kappa,a)$. Furthermore, each pair $C$ and $C'$ are homologous.

Additionally, we have the following phenomena in special cases.
\begin{itemize}
\item If $\ell$ is odd, there is an involution of $X$ negating $\omega$ and interchanging each pair $C$ and $C'$. In the locus of quadratic differentials arising as quotients of these involutions, every cylinder is free. 
\item In the hyperelliptic case ($\kappa_i=-1$ for $i>1$), for every pair $C, C'$, one boundary circle of $C$ is glued to a boundary circle of $C'$ along a collection of saddle connections.
\end{itemize}
\end{thm}
\begin{proof}
To prove the main claim, it suffices to show that each cylinder on each surface in $\cQ_P(\kappa)$ lifts to exactly 2 cylinders on the cover in $\cE(\kappa,a)$. Since each cylinder can be deformed in $\cQ_P(\kappa)$, the lifts of the cylinder can all be deformed equally in $\cE(\kappa,a)$. 

Recall that for any cylinder on a genus zero quadratic differentials, the core curve is separating, and the sum of the orders of the zeros on each side is $-2$. This can be proven by cutting the cylinder, and gluing each boundary circle to itself to create two poles, and using that the sum of the orders of zeros is $-4$ on each of the two resulting genus zero quadratic differentials.

\begin{rem}\label{R:subtle}
Consider the $i$-th singularity of the base quadratic differential, which has order $\kappa_i$. Since the cover is regular, all preimages of this point have the same cone angle. If $\kappa_i>0$, then all the preimages are singular points, since they have cone angle greater than $2\pi$. Similarly, if $\kappa_i=-1$ and $a_i$ has order greater than 2 in $\bZ/(2\ell)$, then all  lifts will again be singular points. However, if $\kappa_i=-1$ and $a_i$ has order 2, then all lifts will be non-singular points. For a generalized Eierlegende-Wollmilchsau this happens if and only if $\ell$ is odd and $a_i=\ell$. As illustrated in \cite[Figure 4.1]{ApisaWrightDiamonds}, it is possible for  a cylinder on the base quadratic differential to have only poles and no zeros on one side of its boundary. In this case this boundary has two poles and we call the cylinder an envelope. If the two poles correspond to $\kappa_i=-1$ and $\kappa_j=-1$ note that because $a_i+a_j$ must be a generator for $\bZ/\ell\subset \bZ/(2\ell)$ it cannot be the case that both of these poles lift to non-singular points. The point of this somewhat subtle remark is that each lift of an open cylinder on the base gives an open cylinder on the cover with a zero on each boundary component, ruling out the possibility that the lift might be a strict subset of a cylinder because it didn't have singularities on one boundary component. 
\end{rem}

To see that each cylinder has  two lifts, by basic covering space theory it suffices to show that the monodromy representation maps the core curve of each cylinder to a generator of $\bZ/\ell\subset\bZ/2\ell$. Fix one side of the cylinder in question, and say the zeros on that side correspond to $I\subset \{1,\ldots, n\}$. We thus get $\sum_{i\in I} \kappa_i=-2$. The core curve of the cylinder has monodromy $\sum_{i\in I} a_i$, and the definition of a generalized Eierlegende-Wollmilchsau gives that this is a generator for  $\bZ/\ell\subset \bZ/(2\ell)$. Thus, the main claim is proven. 

To see that each pair $C, C'$ are homologous, note that cutting the core curves of these cylinders disconnects the surface into two halves, each of which covers one half of the genus zero surface minus the core curve.

We now check the first additional claim. If $\ell$ is odd, then $\bZ/(2\ell)=\bZ/2 \times \bZ/\ell$. As in the proof of Lemma \ref{L:Hol}, we see that the holonomy representation is given by the map to $\bZ/2$. So the generator of $\bZ/2$ negates the Abelian differentials in $\cE(\kappa,a)$. This involution must interchange each pair $C, C'$, since the stabilizer of $C$ is  $\bZ/\ell\subset \bZ/(2\ell)$.

We now check the final claim. This follows because, in the hyperelliptic case, the condition $\sum_{i\in I} \kappa_i=-2$ implies that  that either $I$ or its complement has size two and $\kappa_i=-1$ for both $i$ in whichever has size two. In other words,  each cylinder goes around a pair of poles.  (In the language of \cite[Section 4]{ApisaWrightDiamonds}, we have proven that all cylinders on genus zero quadratic differentials with only one zero are envelopes.) A saddle connection joining those two poles lifts to a collection of saddle connections joining the two lifts of the cylinder. 
\end{proof}

We close this subsection with some (incomplete) remarks on the history and context. 

\begin{rem}
The families of algebraic curves underlying the orbits of the Eierlegende-Wollmilchsau and Ornithorynque have a long history preceding their appearance in dynamics, and are related to Schwarz's 1873 list of hypergeometric functions that can be expressed algebraically \cite{Schwarz}. See, for example, \cite[Sections 5,6]{W1} for an exposition of how  hypergeometric differential equations arise in the study of cyclic (or Abelian) covers. A characterization of when the integrals solving the differential equations are expressible as algebraic functions appears in \cite[Section 9]{McM:braid}; note that the Eierlegende-Wollmilchsau corresponds to the  $(n,d)=(3,4)$ and $(n,d)=(4,4)$ cases in \cite[Table 10]{McM:braid}. 

These families also arose in the study of Shimura varieties; see for example \cite{Moonen}, where the Eierlegende-Wollmilchsau and Ornithorynque families correspond to items $(7)$ and $(12)$ in Table 1, and see also the references therein. The Ornithorynque family corresponds to one of a family of counterexamples to Coleman's Conjecture, which asserted that, for each $g\geq 4$, there are only finitely many points of $\cM_g$ whose Jacobian has complex multiplication. (The Eierlegende-Wollmilchsau family also fits seamlessly into discussions of Coleman's Conjecture, but does not correspond to a counterexample  since its genus is $3$.) 

These families also appear in \cite[Section 6.3]{Rohde} (see the table at the bottom of page 136), and elsewhere.
\end{rem}

\begin{rem}\label{R:Lyap}
Much of the study of the Eierlegende-Wollmilchsau has been driven by the study of Lyapunov exponents and Shimura curves \cite{Fsur, M5}. As noted in the introduction, geminal invariant subvarieties are also important objects for the study of Lyapunov exponents \cite{Apisa-MHD}. 


See \cite{EKZsmall, FMZ} for the Lyapunov spectrum of cyclic covers of the pillowcase (corresponding to $\kappa=(-1,-1,-1,-1)$ above) and \cite{W1} for a generalization to Abelian covers. More generally, Filip showed that the number of zero exponents is determined by the Zariski closure of monodromy \cite{FiZero}, which is well studied for families of cyclic covers \cite[Theorems 3.3.4, 5.1.1, Proposition 5.5.1]{Rohde}, see also \cite[Section 4]{Loo}, \cite[Corollary 5.3]{McM:braid}. 

If the restriction of the symplectic form to an eigenbundle has signature $(p,q)$, one expects this bundle to contribute $|p-q|$ zero exponents. The signature is well known and easy to compute, see for example \cite[Lemma 6.1]{MirWri}. 
\end{rem}

\begin{rem}
It would be interesting to compare our generalizations of the Eierlegende-Wollmilchsau to work in a more topological setting on ``simple closed curve homology", see for example \cite{MalesteinPutman} and the references therein. Cyclic covers and the study of lifts of simple closed curves are relevant in many contexts, and there may also be connections to work such as  \cite{PutmanWieland,AvilaMatheusYoccoz}.
\end{rem}

\subsection{The finite blocking problem}\label{SS:Blocking}

\begin{defn}
Given a half-translation surface $(X, q)$, two points $p$ and $p'$ are said to be \emph{finitely blocked} if there is a finite collection $B \subseteq (X, q) - \{p, p'\}$ such that all line segments from $p$ to $p'$ pass through a point of $B$. If $p=p'$, we say $p$ is finitely blocked from itself. 
\end{defn}

\begin{lem}\label{L:blocking}
Consider a generalized Eierlegende-Wollmilchsau $\cE(\kappa, a)$. Let $(X,\omega)\in \cE(\kappa, a)$, and let $\pi$ be the quotient map by the cyclic deck group. Then every regular point $p$ is finitely blocked from itself by $B=\pi^{-1}(\pi(p)) - \{ p \}$.
\end{lem}

For surfaces that are not covers of tori, Lemma \ref{L:blocking} gives the first examples where every regular point is finitely blocked from itself; in the same context it also gives the first example of finitely blocked points where the blocking set does not include periodic points (defined in \cite{ApisaWright}). 

\begin{proof}
Let $L$ be a line segment from $p$ to itself; so $L$ is the core curve of a cylinder $C$. Since $C$ maps to $\pi(C)$ via an $\ell$-to-1 map, and since $\ell>1$, $L$ must contain at least one other point of $\pi^{-1}(\pi(p))$.

A subtle detail is that $\pi(L)$ cannot be a saddle connection joining two poles; see  Remark \ref{R:subtle}. 
\end{proof}


\begin{rem}
If $\ell$ is odd and $\bZ/\ell$ has a non-trivial subgroup $\bZ/\ell'$, then any two points in the same $\bZ/\ell'$ orbit are finitely blocked from each other by the other points in the same $\bZ/\ell$ orbit. This gives an example where the finitely blocked points (and the blocking set) have slope 1 in the sense of \cite{ApisaWright}. 
\end{rem}

\section{Preliminaries on orbit closures}\label{S:Pre}

We will briefly review some facts about invariant subvarieties. 

\subsection{Rank, rel, and covers}\label{SS:RankRelAndCovers}
The tangent space $T_{(X, \omega)} \cM$ of $\cM$ at a point $(X, \omega) \in \cM$ is naturally identified with a subspace of $H^1(X, \Sigma; \mathbb{C})$, where $\Sigma$ denotes the set of zeros and marked points of $\omega$ on $X$. 

Let $p: H^1(X, \Sigma; \mathbb{C}) \ra H^1(X; \mathbb{C})$ denote the projection from relative to absolute cohomology. The \emph{rank} of $\cM$ is defined to be half the complex-dimension of $p\left( T_{(X, \omega)} \cM \right)$, which is independent of the choice of $(X, \omega) \in \cM$. Rank is an integer by \cite{AEM}.

A tangent direction to $\cM$ is called rel if it is in $\ker(p)$, and we define the rel of $\cM$ to be the complex dimension of $\ker(p) \cap T_{(X, \omega)} (\cM)$. We note the following immediate consequence of the definitions. 

\begin{lem}\label{L:R1Deformations}
If $\cM$ has rank 1, then every element of $T_{(X, \omega)} \cM$ can be written as a element of $\ker(p) \cap T_{(X, \omega)} (\cM)$ plus a tangent vector to the  $\mathrm{GL}(2, \bR)$ orbit of $(X,\omega)$.
\end{lem}

The following definition is central to our analysis.

\begin{defn}\label{D:Covering}
A \emph{translation covering} from $(X, \omega)$ to $(Y, \eta)$ is defined to be a holomorphic map $f: X \ra Y$ branched only over singularities and marked points such that:
\begin{itemize}
\item $f^* \eta = \omega$, 
\item all marked points on $(X,\omega)$ map to marked points on $(Y,\eta)$, 
\item and each marked point on $(Y,\eta)$ has at least one pre-image on $(X,\omega)$ that is a singular or marked point. 
\end{itemize}
A \emph{half-translation surface covering}, from a translation surface or half-translation surface to a half-translation surface, is defined similarly, with the additional stipulations that:
\begin{itemize} 
\item a marked point may map to a simple pole, 
\item but poles need not have  pre-images that are singular or marked. 
\end{itemize}
\end{defn}

The \emph{field of definition} $\bk(\cM)$ of an invariant subvariety $\cM$ is the smallest subfield of $\mathbb{R}$ so that $\cM$ can be defined by equations in $\bk(\cM)$ in any local period coordinate chart. 

A proof of the following well known fact is sketched in \cite[Lemma 3.4]{ApisaWrightDiamonds}.

\begin{lem}\label{L:R1Arithmetic}
If $\cM$ is a rank one invariant subvariety with $\bk(\cM) = \mathbb{Q}$, then $\cM$ is a locus of torus covers. 
\end{lem}

Recall that a translation surface $(X,\omega)$ is a translation cover of a torus if and only if its $\bZ$-module $\Lambda_{abs}\subset \bC$ of absolute periods is a lattice  (rank two $\bZ$-submodule of $\bC$), in which case every translation map to a torus is given by $$p\mapsto \int_{\gamma_p} \omega + \Lambda\in \bC/\Lambda,$$ where $\Lambda$ is a lattice  containing $\Lambda_{abs}$, and $\gamma_p$ is a path from a fixed basepoint to $p$. In particular, if $\pi_{abs}$ is the map defined by $\Lambda=\Lambda_{abs}$, then $\pi_{abs}$ is a map to a torus, and all other maps to a torus have $\pi_{abs}$ as a factor.

We now consider covers of higher genus surfaces. The following result is a combination of a result of M\"oller \cite[Theorem 2.6]{M2} and a slight extension found in \cite[Lemma 3.3]{ApisaWright}.   

\begin{thm}\label{T:MinimalCover}
Suppose that $(X, \omega)$ is not a torus cover. There is a unique translation surface $(X_{min}, \omega_{min})$ and a translation covering $$\pi_{X_{min}}: (X, \omega) \rightarrow (X_{min}, \omega_{min})$$ such that any translation cover from $(X, \omega)$ to translation surface is a factor of $\pi_{X_{min}}$. 

Additionally, there is a quadratic differential $(Q_{min}, q_{min})$ with a degree
1 or 2 map $(X_{min}, \omega_{min}) \rightarrow (Q_{min}, q_{min})$ such that any map from $(X, \omega)$ to a quadratic differential is a factor of the composite map $\pi_{Q_{min}} :
(X, \omega) \rightarrow (Q_{min}, q_{min})$.
\end{thm}

We will use Theorem \ref{T:MinimalCover} in conjunction with the following result \cite[Lemma 4.5]{ApisaWright}. 

\begin{defn}\label{D:F}
Suppose that $(X, \omega)$ is an Abelian or quadratic differential with marked points that belongs to an invariant subvariety $\cM$. Then $\For(X, \omega)$ will denote $(X, \omega)$ once marked points are forgotten. Similarly, we will define $\For(\cM)$ to be the invariant subvariety that is the closure of $\{ \For(Y, \eta) : (Y, \eta) \in \cM \}$.
\end{defn}

\begin{lem}\label{L:InvolutionImpliesHyp-background}
The generic element of a component $\cS$ of a stratum of
Abelian or quadratic differentials admits a non-bijective half-translation
cover to another translation or half-translation surface if and only if $\cF(\cS)$ is a hyperelliptic component.  For hyperelliptic strata of rank at least 2, the only such cover is the quotient by the hyperelliptic involution. 
\end{lem}

Combining the preceding two results gives the following.

\begin{cor}\label{C:MinimalCover}
Suppose that $(X, \omega)$ is a generic surface in either a component of a stratum of Abelian differentials that does not consist of flat tori or a quadratic double of a non-hyperelliptic component of a stratum of quadratic differentials. Then $\pi_{X_{min}}$ is the identity.
\end{cor}
\begin{proof}
This is immediate from Theorem \ref{T:MinimalCover} and Lemma \ref{L:InvolutionImpliesHyp-background}.
\end{proof}
\subsection{Cylinder deformations}\label{SS:CDT}

In this section we recall some definitions and results from \cite{Wcyl}.

\begin{defn}\label{D:CylAndBoundary}
A \emph{cylinder} on a  translation or half-translation surface is an isometric embedding of $\bR/(c \bZ) \times (0,h)$ into the surface, which is not the restriction of an isometric embedding of a larger cylinder. The circumference of the cylinder is defined to be $c$, and its height is defined to be $h$. 
\end{defn}


Two parallel cylinders on  $(X, \omega)\in\cM$ are called \emph{$\cM$-parallel} or \emph{$\cM$-equivalent} if they remain parallel on all nearby surfaces in $\cM$. A maximal collection of $\cM$-parallel cylinders on $(X, \omega)$ is called an \emph{$\cM$-equivalence class}.

If $\bfC = \{C_1, \hdots, C_n\}$ is a collection of parallel cylinders on a flat surface $(X, \omega)$ with core curves $\{ \gamma_1, \hdots, \gamma_n\}$, then we will say that the core curves are consistently oriented if their holonomy vectors are positive real multiples of each other.  

If $\bfC = \{C_1, \hdots, C_n\}$ is an $\cM$-equivalence class of cylinders on $(X, \omega) \in \cM$ with consistently oriented core curves $\{\gamma_1, \hdots, \gamma_n\}$, then the \emph{standard shear} in $\bfC$ is defined to be  the element of $H^1(X, \Sigma; \bC)$ given by 
$$\sigma_{\bfC} := \sum_{i=1}^n h_i \gamma_i^*$$ 
where $h_i$ denotes the height of cylinder $C_i$ and $\gamma_i^*$ is the intersection number with $\gamma_i$.  The following is the main theorem of \cite{Wcyl}, restated in a form closer to \cite[Theorem 4.1]{MirWri}.

\begin{thm}[The Cylinder Deformation Theorem]\label{T:CDT}
If $\bfC$ is an $\cM$-equivalence on $(X,\omega)$, then $\sigma_{\bfC}\in T_{(X,\omega)}\cM$.
\end{thm}

 Here we keep in mind that $T_{(X, \omega)} \cM$ can be identified with a subspace of $H^1(X, \Sigma; \bC)$. 

Notice that when $\bfC$ is a collection of horizontal cylinders, the straight line path in $\cM$ determined by the tangent direction $\sigma_\bfC$ at $(X, \omega)$ determines a family of translation surfaces formed from $(X, \omega)$ by applying $$\begin{pmatrix} 1 & t \\ 0 & 1 \end{pmatrix}$$ for $t \in \bR$ to the cylinders in $\bfC$ on $(X, \omega)$ while fixing the rest of the surface. 
Similarly, if $\bfC$ is a collection of horizontal cylinders the straight line path in $\cM$ determined by the tangent direction $i \sigma_\bfC$ at $(X, \omega)$ determines a family of translation surfaces formed from $(X, \omega)$ by applying $$\begin{pmatrix} 1 & 0 \\ 0 & e^t \end{pmatrix}$$ for $t \in \bR$ to the cylinders in $\bfC$ on $(X, \omega)$ while fixing the rest of the surface. 


\begin{defn}\label{D:SE} 
In \cite{ApisaWrightDiamonds}, we defined a collection of cylinders $\bfC = \{C_1, \hdots, C_n\}$ on $(X, \omega)$ to be an $\cM$-\emph{subequivalence class} if all of the cylinders in $\bfC$ are $\cM$-parallel and if $\bfC$ is a minimal collection of cylinders such that the standard shear $\sigma_{\bfC}$ belongs to $T_{(X, \omega)} \cM$. So every equivalence class is a disjoint union of subequivalence classes. 

For a geminal subvariety, the subequivalence classes of cylinders are either singletons or contain two isometric cylinders (twins).

We will say that a cylinder $C$ on a surface in $\cM$ is $\cM$-generic, or simply generic if $\cM$ is clear from context, if all saddle connections on the boundary of $C$ are $\cM$-parallel to $C$. Thus a subequivalence class of generic cylinders will refer to a subequivalence class all of whose cylinders are generic.
\end{defn}

See the comments after the statement of \cite[Lemma 3.27]{ApisaWrightDiamonds} for some warnings concerning the definition of ``subequivalence class". (The subtleties remarked on there do not occur for geminal orbit closures.)

\subsection{The boundary of an invariant subvariety}
In this paper, we use the What You See Is What You Get (WYSIWYG) partial compactification of invariant subvarieties, as introduced in \cite{MirWri} and further developed in \cite{ChenWright}; see \cite[Section 3.5]{ApisaWrightDiamonds} for a concise and self-contained summary that will suffice for our purposes. 

We use the same notation as \cite{ApisaWrightDiamonds}. In particular, if $\bfC$ is a collection of parallel cylinders on $(X,\omega)$  whose closure $\overline\bfC$ when viewed as a subset of  $(X,\omega)$ is not all of $(X,\omega)$, we define the \emph{cylinder collapse}
$$\Col_{\bfC}(X,\omega) := \lim_{t\to -\infty} a_t^\bfC(X,\omega),$$
where the \emph{standard dilation} $a_t^\bfC(X,\omega)$ is defined to be the result of rotating the surface so the cylinders are horizontal, applying 
$$a_t=\begin{pmatrix} 1 & 0 \\ 0 & e^t \end{pmatrix}$$
only to the cylinders in $\bfC$, and then applying the inverse rotation. 

Thus $\Col_{\bfC}(X,\omega)$ is the result of collapsing the cylinders in $\bfC$ in the direction perpendicular to their core curves, while keeping their circumferences constant and leaving the rest of the surface otherwise unchanged. We will be almost exclusively interested in the case when the collapse causes the surface to degenerate.

 Collapsing $\bfC$ gives rise to a  collection $\Col_{\bfC}(\bfC)$ of parallel saddle connections on $\Col_{\bfC}(X, \omega)$, which is discussed in detail in \cite[Remark 2.5]{ApisaWrightDiamonds} and the beginning of the proof of \cite[Sublemma 2.6]{ApisaWrightDiamonds}. 

If $(X,\omega)$ is contained in an invariant subvariety $\cM$, and $\bfC$ is a subequivalence class of cylinders, then $\Col_{\bfC}(X,\omega)$ is contained in a component of the boundary of $\cM$ that will be denoted by $\cM_{\bfC}$.

We will frequently make use of the following result from \cite[Lemma 2.2]{ApisaWrightDiamonds}.

\begin{lem}\label{L:ColExists}
Suppose that $f: (X, \omega) \ra (Y, \eta)$ is a half-translation covering. Let $\bfC\subset (X,\omega)$ be a collection of cylinders such that $$f^{-1}(f(\overline\bfC)) = \overline\bfC$$  and $\overline\bfC\neq (X,\omega)$. Then there is a half-translation surface covering map $$\Col_{\bfC}(f): \Col_{\bfC}(X, \omega) \ra \Col_{f(\bfC)}(Y, \eta)$$ of the same degree.
\end{lem}

\section{Properties of geminal orbit closures}

In this section we establish some general facts about geminal orbit closures which can be proved directly without inductive arguments. 

\subsection{Foundational results} We start with some basic observations and move on to some quite non-trivial statements. 

\begin{lem}\label{L:GeminalField}
If $\cM$ is geminal then $\bold{k}(\cM) = \mathbb{Q}$. 
\end{lem}
\begin{proof}
This follows from the proof of \cite[Theorem 1.9]{Wcyl}  as follows. Consider a surface $(X, \omega)\in\cM$, and a collection $\bfC$ of cylinders on $(X,\omega)$ that is either a single free cylinder or a pair of twins.  In the free case, the standard deformation $\sigma_{\bfC}\in H^1(X,\Sigma, \bC)$ is, up to scale, the dual of the core curve; and otherwise it is, up to scale, the sum of the duals of the two core curves. In particular, $\sigma_{\bfC}$ is a multiple of a rational vector. Because $\cM$ is geminal we have that $\sigma_{\bfC}\in T_{(X,\omega)}(\cM)$. Since $\sigma_{\bfC}$ is not contained in $\ker(p)$, \cite[Theorem 5.1]{Wfield} then gives that $T_{(X,\omega)}(\cM)$ is defined over $\bQ$.
\end{proof}

\begin{lem}\label{L:GeminalBoundary}
The boundary of a geminal orbit closure is geminal.
\end{lem}
\begin{proof}
Let $C$ be any cylinder on a surface $(X', \omega')$ in a boundary orbit closure $\cM'$ of a geminal orbit closure $\cM$. It is possible to find a sequence surfaces $(X_n, \omega_n)$ in $\cM$ that degenerate to $(X', \omega')$, with cylinders $C_n$ that limit to $C$.

If infinitely many of the $C_n$ are free, then it follows from  \cite[Theorem 2.9]{MirWri} and \cite[Theorem 1.2]{ChenWright} that $C$ is free.  (Referring to the terminology of those papers, we note that the space $V$ of vanishing cycles is generated by a collection of classes not crossing the core curve of $C_n$, so the standard deformation in $C_n$ is in $\Ann(V)$.)  

If infinitely many of the $C_n$ have twins $C_n'$, then these twins $C_n'$ limit to a cylinder $C'$ isometric to $C$  (as in  \cite[Lemma 2.15]{MirWri}), and again it follows from  \cite[Theorem 2.9]{MirWri} and \cite[Theorem 1.2]{ChenWright} that $C$ and $C'$ are twins. 
\end{proof}

\begin{defn}\label{D:hGeminal}
An invariant subvariety $\cM$ will be called $h$-geminal if, for every $(X, \omega)$ in $\cM$, every cylinder $C$ is either free or possesses a twin $C'$ such that  $C$ and  $C'$ have homologous core curves.
\end{defn}

\begin{lem}\label{L:HGeminalBoundary}
If $\cM$ is $h$-geminal and $\cM'$ is in its boundary, then $\cM'$ is $h$-geminal.
\end{lem}
\begin{proof}
This follows as in the proof of Lemma \ref{L:GeminalBoundary}, with the additional observation that  homologous cylinders  degenerate to homologous cylinders.  (If $(X_n, \omega_n) \to (X,\omega)$, then there are ``collapse maps" $f_n : (X_n,\omega_n)\to (X,\omega)$ which induce maps on relative homology, as discussed in \cite[Proposition 2.4]{MirWri}.)   
\end{proof}

\begin{lem}\label{L:StableTwins}
In any geminal subvariety $\cM$, a pair of cylinders $C, C'$ with core curves $\gamma, \gamma'$ are twins if and only if $$\gamma+\gamma' \in H_1(X\setminus \Sigma) \simeq H^1(X,\Sigma)$$
is in the tangent space to $\cM$ and $\gamma$ is not in the tangent space. 
\end{lem}

This lemma will be implicit in our analysis. It rules out the subtle possibility that a pair of twins might cease to be twins following a large deformation after which the cylinders persist but are bounded by ``different" saddle connections and zeros.  

\begin{proof}
If $C, C'$ are twins, then the cylinder deformation that deforms them equally has derivative $\gamma+\gamma'$. And $\gamma$ cannot be in the tangent space, because deforming in $i$ times this direction would give a surface in $\cM$ where $C$, $C'$ are not isometric. 

Conversely, if $\gamma+\gamma'$ is in the tangent space, then deforming in $i$ times this direction changes the modulus $C$ and $C'$ without changing any other parallel cylinder. Hence we see that $C$ cannot be twinned with any cylinder except possibly $C'$, because twins must have the same modulus. 
Hence if $C$ and $C'$ are not twins with each other, then each must be free, but this cannot be so because $\gamma$ is not in the tangent space. 
\end{proof}

Recall that subequivalence classes were defined in Definition \ref{D:SE}, and for geminal subvarieties they are either a single free cylinder or a pair of twins. 

\begin{lem}\label{L:Disconnect}
Suppose that $\bfC$ is a subequivalence class of generic cylinders on a surface $(X, \omega)$  in a geminal orbit closure $\cM$. 


If $\Col_{\bfC}(X, \omega)$ is disconnected, then $\cM_{\bfC}$ is a subset of $\cH_1 \times \cH_2$ where $\cH_1$ and $\cH_2$ are components of strata of Abelian differentials. Moreover, the projection from $\cM_{\bfC}$ to $\cH_i$ is a local  diffeomorphism for $i \in \{1, 2\}$ and $\cH_1$ and $\cH_2$ have the same rank and dimension. 
\end{lem}
\begin{rem}
It is tempting to claim that the conclusions imply that $\cH_1=\cH_2$ and $\cM_{\bfC}$ is a diagonal, but this is not quite true. One can start with the diagonal in $\cH_1\times \cH_1$, and then rotate the second surface by $\pi$. If this stratum is not hyperelliptic, then this gives an orbit closure that isn't the diagonal. Furthermore, if $\cH_1$ is genus 1, then one can apply the times $k$ map to one factor and the times $\ell$ map to the other, for any $k, \ell \in \bZ$ (after picking a point on each side to serve as the origin of the elliptic curve). We do not know if there are more exotic non-diagonal examples; see \cite[Conjecture 8.35]{ApisaWrightDiamonds}. 
\end{rem}
\begin{proof}
The proof will use and require familiarity with the statement of \cite[Theorem 1.3]{ChenWright}. 

Our first claim is that $\cM_{\bfC}$ is prime in the sense of \cite[Theorem 1.3]{ChenWright}, which means precisely that it is not a product.  Indeed, because of the assumption that all saddle connection on the boundary of $\bfC$ are $\cM$-parallel to $\bfC$,  \cite[Theorem 2.9]{MirWri} and \cite[Theorem 1.2]{ChenWright} give that all the saddle connections in $\Col_\bfC(\bfC)$ are $\cM_{\bfC}$-parallel to each other;  see Apisa-Wright \cite[Lemma 6.5]{ApisaWrightHighRank} for a detailed proof. Since there is at least one saddle connection of $\Col_\bfC(\bfC)$ on each component, this gives the claim.   See also \cite[Lemma 9.1.]{ApisaWrightHighRank} for a generalization of this claim. 

It follows from \cite[Theorem 1.3]{ChenWright} that in $\cM_\bfC$, the absolute periods of any component of $\Col_{\bfC}(X, \omega)$ locally determine the absolute periods of any other component. Our second  claim is that this implies that there are exactly two components, and each cylinder on one component has a twin on another component. 

Indeed, if $\bfC'$ is a subequivalence class on a surface in $\cM_\bfC$, it is either a single free cylinder or a pair of twins. Deforming $\bfC'$ changes the absolute periods of the one or two components which contain cylinders of $\bfC'$. Hence we get that there must be exactly two components, each with one cylinder of $\bfC'$, and the second claim is proved. 

Now we have that $\cM_\bfC$ is contained in a product $\cH_1\times \cH_2$, and we can consider the projections $\pi_i(\cM_{\bfC})\subset \cH_i$. By \cite[Theorem 1.3]{ChenWright}, the closure $\cM_i$ of $\pi_i(\cM_{\bfC})$ is an invariant subvariety, and $\pi_i(\cM_{\bfC})$ contains an open dense subset of $\cM_i$. 

In this  $\cM_i$, every cylinder is free. Hence \cite[Theorem 1.5]{MirWri2} gives that $\cM_i=\cH_i$. 

It remains only to show that the projections are local  diffeomorphisms. If not, there is deformation of a surface in $\cM_\bfC$ that only affects one of the two components. Nudging the surface if necessary, we can assume it is square-tiled. Any deformation of a square-tiled surface must change the dimensions of one of the horizontal or vertical cylinders, so we see that our deformation contradicts the fact that every cylinder has a twin on the other component. 
\end{proof}

\begin{lem}\label{L:GeminalTwinAdjacency}
Let $C$ and $C'$ be twins on a surface $(X, \omega)$ in a geminal subvariety $\cM$. If $\bfD$ is a subequivalence class of cylinders, then $C$ shares a boundary saddle connection with a cylinder in $\bfD$ if and only if $C'$ does.
\end{lem}

 The proof will ``over-collapse $\bfD$ to attack $C$", following a technique used in \cite{ApisaWrightDiamonds}. Other arguments using over-collapses were used in  \cite[Section 4.2, Figure 7]{AN}, \cite[Figure 3.1]{Apisa-Rk1Hyp}, \cite[Figure 2.14]{Apisa-PeriodicPointsG=2}, \cite[Figure 15]{AN2}, \cite[Figure 5.7]{ApisaWright}.

\begin{proof}
First rotate so that $C$ is a horizontal cylinder. Suppose that $C$ shares a boundary saddle connection with a cylinder $D$ in $\bfD$. Next,  shear $\bfD$ so that it does not contain a vertical saddle connection, then vertically collapse it to make the height of the cylinders in $\bfD$ zero while ensuring that a zero lands on the saddle connection shared by $C$ and $D$. Since no zeros have collided at this point, we may continue this vertical collapse deformation  by continuing linearly in period coordinates; this moves the singularities on the saddle connection shared by $C$ and $D$ into the interior of $C$. This changes the height of $C$. 

By Lemma \ref{L:StableTwins}, $C$ and $C'$ are twins even after this deformation, so it follows that the height of $C'$ must change too, implying that it also borders a cylinder in $\bfD$.
\end{proof}

\begin{cor}\label{C:DisconnectR1}
Suppose that $\bfC$ is a subequivalence class of generic cylinders on a  surface $(X, \omega)$ in a geminal invariant subvariety $\cM$. Suppose that $\cM_{\bfC}$ has rank one and that $\Col_{\bfC}(X, \omega)$ is disconnected. Then there is a stratum $\cH$ of flat tori with marked points such that $\cM_{\bfC}$ is the diagonal or anti-diagonal embedding of $\cH$ into $\cH \times \cH$.
\end{cor}
\begin{proof}
By Lemma \ref{L:Disconnect}, $\cM_{\bfC}$ is a subset of $\cH_1 \times \cH_2$ where $\cH_1$ and $\cH_2$ are components of strata of Abelian differentials. Moreover, the projection from $\cM_{\bfC}$ to $\cH_i$ is a local diffeomorphism for $i \in \{1, 2\}$ and $\cH_1$ and $\cH_2$ have the same rank and dimension. 

The only rank one strata of Abelian differentials are those of the form $\cH(0^n)$ for some positive integer $n$. These strata have dimension $n+1$. Since $\cH_1$ and $\cH_2$ are both rank one and have identical dimension, it follows that $\cH_1 = \cH_2 = \cH(0^n)$.

  We may deform $\Col_{\bfC}(X, \omega)$ to obtain a surface in $\cM_{\bfC}$ each of whose components has $n$ horizontal cylinders and one vertical one.
These cylinders come in twin pairs since $\cM_{\bfC}$ is geminal.  

\begin{figure}[h!]
\includegraphics[width=0.9\linewidth]{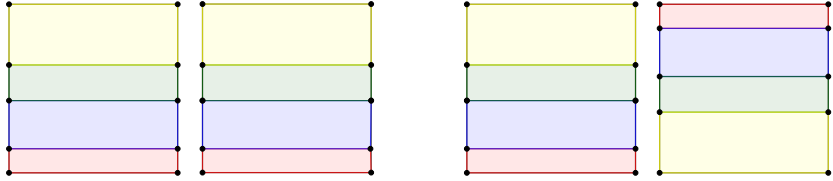}
\caption{The proof of Corollary \ref{C:DisconnectR1}. Left: A diagonal. Right: An anti-diagonal.}
\label{F:DiagAntiDiag}
\end{figure}

Lemma \ref{L:GeminalTwinAdjacency} implies that the cyclic order of the horizontal cylinders on each component of  the deformed surface is the same or opposite as the cyclic order of the twin cylinders on the other component; see Figure \ref{F:DiagAntiDiag}. If the cyclic order is the same, $\cM$ is a diagonal, and if the cyclic order is opposite, $\cM$ is an anti-diagonal. 
\end{proof}

\subsection{Good and optimal maps} A useful tool to study geminal subvarieties will be the notion of ``good" and ``optimal" maps. Recall from Definition \ref{D:CylAndBoundary} that cylinders do not contain their boundary. 

\begin{defn}\label{D:GoodAndOptimal}
Let $(X, \omega)$ be a translation surface. A translation cover $f: (X, \omega) \rightarrow (X', \omega')$ will be called \emph{good} if every cylinder $C$ on $(X, \omega)$ is the preimage of its image under $f$. The cover will be called \emph{optimal} if it is good and any other good map is a factor of it.
\end{defn} 

\begin{war}\label{W:OptNeqOpt}
After completing this preprint, we discovered that M\"oller \cite{M5} and Aulicino-Norton \cite{AulicinoNorton} use ``optimal map" to mean the minimal degree map from a torus cover to a torus. Our definition applies more broadly, even to covers of higher genus surfaces, and, even in the case where both definitions apply, they typically do not  agree.\footnote{ In the torus cover case, we will see that if the orbit closure of a torus cover $(X, \omega)$ is geminal, but not $h$-geminal, then the optimal map is the identity (see Lemma \ref{L:nhGeminal}). When the orbit closure is $h$-geminal, the optimal map is the minimal degree map to the torus (see Lemma \ref{L:HGeminalOptimal}).}  
\end{war}

\begin{ex}\label{E:SubtleCyls}
Consider the example in Figure \ref{F:SubtleCyls}, and keep in mind that according to Definition \ref{D:CylAndBoundary} cylinders are open and maximal. The image of $C_1$ is not a cylinder because it contains the marked point $p$. The image of $C_2$ is a cylinder, but the preimage of the image of $C_2$ is not a union of cylinders because it contains a non-maximal subset of $C_1$. 
\begin{figure}[h!]
\includegraphics[width=0.6\linewidth]{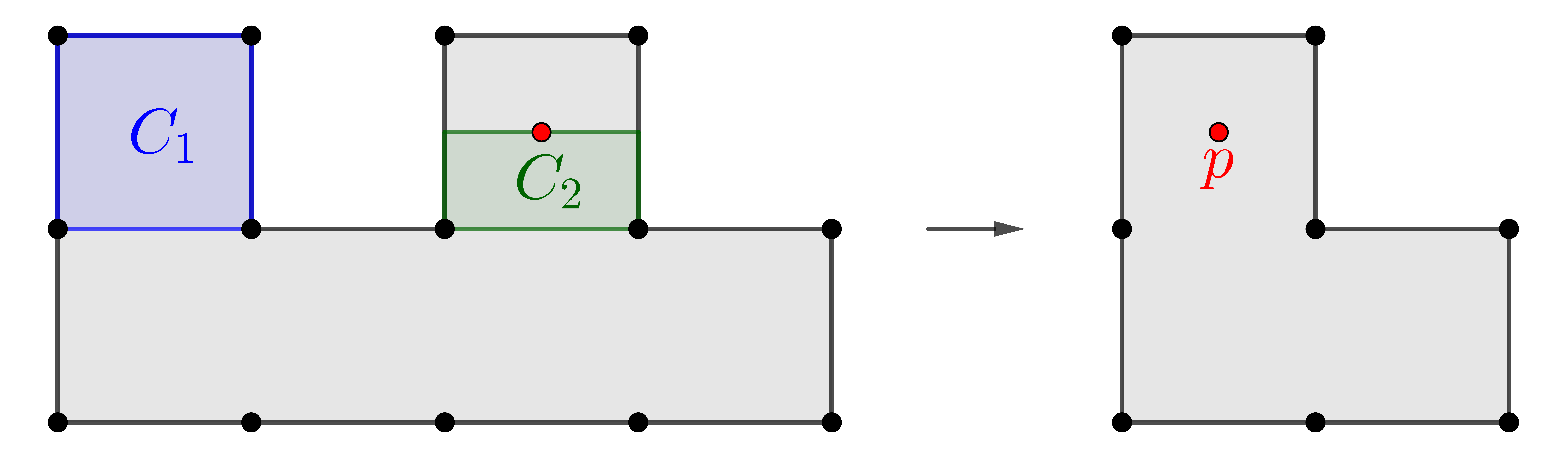}
\caption{  A degree 2 cover of a surface in $\cH(2,0)$ where the marked point $p$ is not a branch point, and exactly one preimage of  $p$ is marked.  }
\label{F:SubtleCyls}
\end{figure}
\end{ex}

\begin{ex}\label{E:WeirdH0}
If $(X,\omega)\in \cH(0)$, the quotient by the two torsion subgroup is not a good map, even though every cylinder is contained in the  preimage of its image; see Figure \ref{F:2quotient}. This is an extremely special example. 
\begin{figure}[h!]
\includegraphics[width=0.4\linewidth]{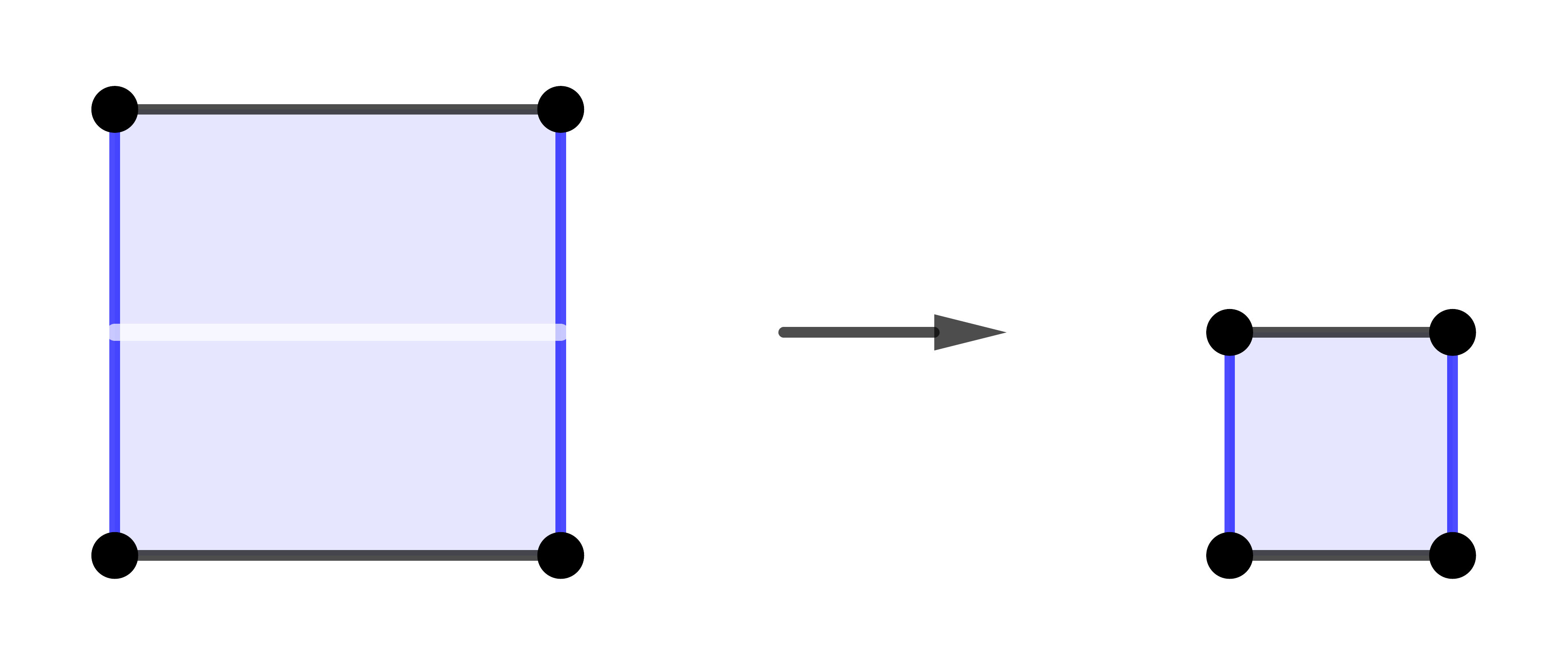}
\caption{  Example \ref{E:WeirdH0}. The preimage of the horizontal cylinder on the right is not a cylinder according to Definition \ref{D:CylAndBoundary}.  }
\label{F:2quotient}
\end{figure}
\end{ex}

\begin{rem}\label{R:Good}
Consider a translation cover $f: (X, \omega) \rightarrow (X', \omega')$, and recall from Definition \ref{D:Covering} that each zero or marked point of $(X, \omega)$ must map to a zero or marked point of $(X', \omega')$, and that every marked point on $(X', \omega')$ must have a preimage that is a marked point or a singularity. A cylinder $C$ on $(X,\omega)$ can fail to be the preimage of its image in two ways. 
\begin{enumerate}
\item\label{R:Good:MarkedPoint} The image $f(C)$ of $C$  contains a marked point. In this case,  $f^{-1}(f(C))$ contains a marked point or zero, proving $f^{-1}(f(C))\neq C$. 
\item\label{R:Good:MultipleComponents}  $f(C)$ is a cylinder on $(X', \omega')$, but the open set $f^{-1}(f(C))$ has more than one connected component. (The fact that $f(C)$ is a cylinder implies that $f^{-1}(f(C))$ is a subset of the union of cylinders parallel to $C$. The open set $f^{-1}(f(C))$  could fail to be a union of cylinders or could be a union of more than one cylinder.)  
\end{enumerate}
Consequently, if a map $f$ is not good, then any sufficiently small deformation of $f$ is also not good. 
\end{rem} 

\begin{rem}\label{R:Good2}
Equivalently, $f: (X, \omega) \rightarrow (X', \omega')$ is good if the preimage of each cylinder on  $(X', \omega')$ is a single cylinder on $(X,\omega)$.
\end{rem} 

In general, few translation covers are good, but nonetheless good and optimal covers will play a central role in the remainder of this paper. If the only good cover is the identity map, then the identity map is optimal, but in general it is not clear if optimal maps exist. 

We will now give some cases where the identity is the optimal map.

\begin{lem}\label{L:IsGood}
Suppose that $(X,\omega)$ contains a pair of cylinders $H, V$ whose core curves intersect exactly once. Then the identity is the optimal map for $(X,\omega)$. 
\end{lem}
\begin{proof}
Suppose $g:(X,\omega)\to (Y,\eta)$ is good. Suppose the core curves  of $g(H)$ and $g(V)$ intersect $k$ times. Then, since $g$ is good, the core curves of $H$ and $V$ intersect $k\cdot d$ times, where $d$ is the degree of $g$. Hence $d=1$ and $g$ is the identity. 
\end{proof}

\begin{lem}\label{L:GoodInH(0^n)}
For any surface in $\cH(0^n)$, the only good map is the identity. 
\end{lem}
\begin{proof}
This follows immediately from Lemma \ref{L:IsGood}. 
\end{proof}

\begin{lem}\label{L:DisconnectingUnderOptimalMap}
Suppose that $\bfC$ is a subequivalence class of cylinders on a surface $(X, \omega)$ in an invariant subvariety $\cM$. Suppose that $f: (X, \omega) \ra (Y, \eta)$ is a good map. Then $\Col_{\bfC}(X, \omega)$ is disconnected if and only if $\Col_{f(\bfC)}\left( f(X, \omega) \right)$ is disconnected.
\end{lem}
  Actually, the proof only requires that the standard deformation of $\bfC$ remains in $\cM$, a weaker assumption than that $\bfC$ is a subequivalence class. We state it as is since we will only apply it to subequivalence classes. Note also that the assumption that $f$ is good is (more than) enough for Lemma \ref{L:ColExists} to apply.  
\begin{proof}
Suppose to a contradiction that $\Col_{f(\bfC)}\left( f(X, \omega) \right)$ is connected and $\Col_{\bfC}(X, \omega)$ is disconnected. Let $C$ be a cylinder on $\Col_{\bfC}(X, \omega)$. By assumption, $C$ is the preimage of its image, but that contradicts the fact that the other component of $\Col_{\bfC}(X, \omega)$ surjects onto $\Col_{f(\bfC)}\left( f(X, \omega) \right)$.

 This establishes one implication in the claim; the other is trivial.
\end{proof}

To avoid  technical annoyances, it is convenient to make the following definitions. 

\begin{defn}
Given $(X, \omega)$ in an invariant subvariety $\cM$, and a half-translation covering map $f$ defined on $(X,\omega)$, say that $f$ is $\cM$-generic if the map $f$ can be deformed to every nearby surface in $\cM$. Equivalently, $\cM$ is contained in a full locus of covers of a component of a stratum associated to $f$. 
\end{defn} 

\begin{ex}
Let $\cM=\cH$ be a component of a stratum of Abelian differentials of genus at least 2, and let $(X,\omega)\in \cM$ be square tiled, so there is a map $f$ from $(X,\omega)$ to a torus. This $f$ is not $\cM$-generic. 
\end{ex}

\begin{rem}
If $(X, \omega)$ has orbit closure $\cM$, then it is easy to check that every map $f$ defined on $(X,\omega)$ is $\cM$ generic.
\end{rem}

\begin{defn}\label{D:MGoodAndOptimal}
If $(X,\omega)\in \cM$, a map defined on $(X,\omega)$ will be called $\cM$-good if it is $\cM$-generic and the deformations of the map are good on all deformations of $(X, \omega)$ in $\cM$.

If $(X,\omega)\in \cM$, a map defined on $(X,\omega)$ will be called $\cM$-optimal if it is $\cM$-good and any  other $\cM$-good map is a factor of it. 
\end{defn}

\begin{rem}
If the orbit of $(X,\omega)$ is dense in $\cM$, then a map is good if and only if it is $\cM$-good, and it is optimal if and only if it is $\cM$-optimal. Thus, for most surfaces in $\cM$, the notions in Definition \ref{D:GoodAndOptimal} are equivalent to their $\cM$ adapted versions in Definition \ref{D:MGoodAndOptimal}. 

The difference between $\cM$-optimal and optimal could be omitted entirely if we made sufficiently strong genericity assumptions on all our surfaces. 
\end{rem}

\begin{ex}
There are many examples of surfaces $(X,\omega)$ contained in invariant subvarieties $\cN$ such that $(X,\omega)$ has an $\cN$-generic good map that isn't $\cN$-good. (For example, consider a surface $(X,\omega)$ whose orbit closure $\cM$ is geminal, and for which the optimal map $\pi$ has degree greater than one. Let $\cN$ denote the locus of all covers of surfaces in the stratum of $\pi(X,\omega)$; so by construction $\pi$ is $\cN$-generic, and $\cN$ is a full locus of covers of a component of a stratum. If $\pi$ were $\cN$-good, it would be possible to show that every cylinder on every surface in $\cN$ is free, contradicting \cite[Theorem 1.5]{MirWri2}. One can take $(X,\omega)$ to be the Eierlegende-Wollmilchsau.)
%
%
\end{ex}

\begin{lem}\label{L:DegeneratingOptimalMap}
Suppose that $\bfC$ is a subequivalence class of cylinders on a surface $(X, \omega)$ contained in an invariant subvariety $\cM$.   
  If $\pi$ is a good map on $(X, \omega)$, then $\Col_{\bfC}(\pi)$ is an good map on $\Col_{\bfC}(X, \omega)$. Moreover if $\pi$ is $\cM$-good and $\bfC$ remains a subequivalence class on deformations of $(X, \omega)$ in $\cM$ and $\bfC$ consists of $\cM$-generic cylinders, then $\Col_{\bfC}(\pi)$ is $\cM_{\bfC}$-good. 

\end{lem}

Recall that $\Col_{\bfC}(\pi)$ was defined in Lemma \ref{L:ColExists} following \cite[Lemma 2.2]{ApisaWrightDiamonds}.

\begin{proof}
By assumption, $\pi$ may be deformed to any surface in $\cM$. By definition, $\displaystyle{\Col_{\bfC}(X,\omega) = \lim_{t\to -\infty} a_t^\bfC(X,\omega)}$. Let $\pi_t$ denote the map arising from deforming $\pi$ to $a_t^\bfC(X,\omega)$. 

Any cylinder $C$ on $\Col_{\bfC}(X, \omega)$ is a limit of cylinders  $C_t$ on $(X_t, \omega_t)$, as in   \cite[Lemma 2.15]{MirWri}. If $C$ is not the preimage of its image under $\Col_{\bfC}(\pi)$, then the same must be true for $\pi_t$ for large enough $t$, contradicting the fact that $\pi_t$ is good. This shows that $\Col_{\bfC}(\pi)$ is good.

To see that $\Col_{\bfC}(\pi)$ is $\cM_\bfC$-good rather than merely good, it suffices to note that a neighborhood of $\Col_\bfC(X,\omega)$ can be obtained via the same construction. (In much greater generality, this follows from \cite[Proposition 2.6]{MirWri}, together with the main results of \cite{MirWri, ChenWright}.)
%
\end{proof}

\begin{lem}\label{L:DisconnectedDegenerationMaps}
Suppose that $\bfC$ is a subequivalence class of cylinders on a surface $(X, \omega)$ in a geminal invariant subvariety $\cM$. If $\Col_{\bfC}(X, \omega)$ is disconnected, then the $\cM_\bfC$-optimal map on $\Col_{\bfC}(X, \omega)$ is the identity and the $\cM$-optimal map on $(X, \omega)$ is the identity.  
\end{lem}
\begin{proof}
By Lemma \ref{L:Disconnect}, $\cM_{\bfC}$ is a subset of $\cH_1 \times \cH_2$, where $\cH_1$ and $\cH_2$ are components of strata of Abelian differentials. Moreover, the projection from $\cM_{\bfC}$ to $\cH_i$ is a local diffeomorphism for $i \in \{1, 2\}$ and $\cH_1$ and $\cH_2$ have the same rank and dimension.

Recall that, for each  $i \in \{1, 2\}$, the identity is the only good map for generic surfaces in $\cH_i$, by Lemma \ref{L:GoodInH(0^n)} in the rank 1 case and  Corollary \ref{C:MinimalCover} in the higher rank case.  

Hence, any $\cM_{\bfC}$-good map for $\Col_{\bfC}(X, \omega)$ would have to be degree 1 on each component. If such a map were not the identity, it would have to identify the two components, contradicting the requirement that each cylinder is equal to the preimage of its image. It follows that the identity is $\cM_{\bfC}$-optimal for $\Col_{\bfC}(X, \omega)$.

 It follows immediately that the $\cM$-optimal map on $(X, \omega)$ must be the identity since otherwise $\Col_{\bfC}(X, \omega)$ would admit a non-identity $\cM$-good map by Lemma \ref{L:DegeneratingOptimalMap}, contradicting the fact that the identity map is $\cM_{\bfC}$-optimal for $\Col_{\bfC}(X, \omega)$. 
\end{proof}

\section{Diamonds}\label{S:Diamond}

Our approach to proving Theorem \ref{T:geminal} will be inductive. Diamonds will be the primary tool that powers this inductive approach. In this section, we recall from our previous paper \cite{ApisaWrightDiamonds} the definition of diamonds and summarize the results about them that we will use. The main result of this section is Proposition \ref{P:DiamondsAreCool}, which adapts some of the  general main results of \cite{ApisaWrightDiamonds} to our  more specific context.  

\begin{defn}\label{D:Diamond}
Say that $((X,\omega), \cM, \bfC_1, \bfC_2)$ is a \emph{diamond} if $\cM$ is an invariant subvariety and $(X,\omega)\in \cM$ has two collections of cylinders $\bfC_1$ and $\bfC_2$ such that 
\begin{enumerate}
\item $\bfC_1$ and $\bfC_2$ are disjoint and do not share any boundary saddle connections,
\item the standard dilations of each $\bfC_i$ remain in $\cM$, and
\item\label{I:Collapse}  the collapses of each $\bfC_i$  cause the surface to degenerate.
\end{enumerate}
\end{defn}

Condition \eqref{I:Collapse}  holds if and only if $\overline\bfC_i$ contains a saddle connection perpendicular to its core curves. Since most surface do not have any pairs of perpendicular saddle connections, it will be helpful to modify this condition. 

Given a collection of parallel cylinders $\bfC$ on $(X,\omega)$ endowed with a choice of direction not parallel to these cylinders, we can define the \emph{cylinder collapse} $\Col_\bfC(X,\omega)$ to be the result of applying a cylinder deformation of $\bfC$ to reduce the area of $\bfC$ to zero, while not changing the foliation in the given direction. If there is a saddle connection in this direction contained in $\overline\bfC$ this will degenerate the surface. 

This notation $\Col_\bfC(X,\omega)$ risks some confusion but is very convenient. The risk of confusion is that the notation has one meaning when $\bfC$ is not endowed with a choice of direction (namely a perpendicular collapse), and has a different meaning when a direction is implicitly specified (namely a collapse in the given direction). In the latter case $\Col_\bfC(X,\omega)$ can be obtained by first shearing $\bfC$ so the given direction becomes perpendicular to the core curves of $\bfC$, and then performing the perpendicular collapse. 

We may modify Definition \ref{D:Diamond}, by, for each $i$, specifying a direction in which $\overline\bfC_i$ contains a saddle connection, and use the collapses in this direction instead of the perpendicular collapses. In this case, we will say that $((X,\omega), \cM, \bfC_1, \bfC_2)$ is a \emph{skew diamond}, it being implicit that the $\bfC_i$ are each equipped with a choice of direction.

\begin{rem}
All results that we know of that hold for diamonds also hold for skew diamonds. One can typically immediately reduce from the case of skew diamonds to the case of diamonds by shearing the $\bfC_i$; and also the proofs work equally well in the skew case. In fact, the only reason we have previously used diamonds rather than skew diamonds was to simplify the definition, to avoid the (implicit) choice of directions.
 \end{rem}

 The most basic result about diamonds is the following result, which is called the Diamond Lemma \cite[Lemma 2.3]{ApisaWrightDiamonds}. Here we will denote by $\Col_{\bfC_i}(\bfC_{i+1})$ the collection of cylinders on $\Col_{\bfC_i}(X,\omega)$ arising from $\bfC_{i+1}$, as in \cite[Section 1.1]{ApisaWrightDiamonds}.  

\begin{lem}\label{L:diamond}
Suppose that $f_i$ is a (half)-translation cover whose domain is $\Col_{\bfC_i}(X, \omega)$. Suppose that 
$$\overline{\Col_{\bfC_i}(\bfC_{i+1})}=f_i^{-1}\left(f_i\left( \overline{\Col_{\bfC_i}(\bfC_{i+1}) }\right)\right)$$

and that 
$$\Col_{\ColOne(\bfC_2)}(f_1) = \Col_{\ColTwo(\bfC_1)}(f_2).$$ Then $(X,\omega)$ admits a covering map $f$ to a quadratic differential, with $\overline\bfC_{i}=f^{-1}(f(\overline\bfC_{i}))$, and $f_i = \Col_{\bfC_{i}}(f)$.
\end{lem}

\begin{defn}\label{D:CP}
We will say that a cover of translation or half-translation surfaces satisfies \emph{Assumption CP} (for Cylinder Pre-image) if the pre-image of every cylinder is a union of cylinders. (Because of our conventions and definitions, if the preimage of a cylinder $C$ consists of cylinders, then all these cylinders must have the same height as $C$.) 
\end{defn}



\begin{rem}\label{R:OptimalDiamond}
Every good map satisfies Assumption CP. (Good maps are defined in Definition \ref{D:GoodAndOptimal}.) Moreover, if $((X,\omega), \cM, \bfC_1, \bfC_2)$ is a diamond and $f_i$ is a good translation cover whose domain is $\Col_{\bfC_i}(X, \omega)$, then $\Col_{\bfC_i}(\bfC_{i+1})=f_i^{-1}(f_i\left( \Col_{\bfC_i}(\bfC_{i+1}) \right))$ is automatically true, since for a good cover every cylinder is the preimage of its image. In particular, the first displayed assumption of Lemma \ref{L:diamond} holds.  
\end{rem}




\noindent In the sequel we will mainly use the following type of diamond: 

\begin{defn}\label{D:GenericDiamond}
A diamond will be called a \emph{generic} if 
\begin{enumerate}
\item\label{E:genericSE} each $\bfC_i$ is a subequivalence class of generic cylinders, and  
\item\label{E:one} $\cM_{\bfC_i}$ has dimension exactly one less than $\cM$ for each $i \in \{1, 2\}$.
\end{enumerate}
\end{defn} 

Subequivalence classes and generic cylinders are defined in Definition \ref{D:SE}.


 The following result says that generic diamonds abound \cite[Lemma 3.31]{ApisaWrightDiamonds}.

\begin{lem}\label{L:GenericDiamond}
Let $\cM$ be an invariant subvariety of rank at least 2. Then there exists a surface $(X,\omega)\in \cM$ with collections $\bfC_1, \bfC_2$ of cylinders that form a generic diamond.  

Moreover, up to shearing $\bfC_i$, $(X, \omega)$ may be assumed to be any surface  on which all parallel saddle connections are $\cM$-parallel, and $\bfC_1$ may be any subequivalence class. 
\end{lem}


\noindent The following result allows for the determination of an invariant subvariety given a special kind of diamond \cite[Lemma 8.31]{ApisaWrightDiamonds}.

\begin{lem}\label{L:IntroFull}
Suppose that $\left( (X, \omega), \cM, \bfC_1, \bfC_2 \right)$ is a generic diamond and that $\MOne$ and $\MTwo$  are full loci of covers of strata of Abelian differentials that satisfy Assumption CP. Then $\cM$ is a full locus of covers of a stratum of Abelian differentials. If $f$ (resp $f_i$) denotes the cover on $(X, \omega)$ (resp. $\Col_{\bfC_i}(X, \omega)$), then $\Col_{\bfC_i}(f) = f_i$.
\end{lem}

\noindent The following result is the main mechanism that makes an inductive argument possible. 

\begin{prop}\label{P:DiamondsAreCool}
Suppose that $((X, \omega), \cM, \bfC_1, \bfC_2)$ is a generic diamond where $\cM$ is geminal and $\ColOneTwoX$ is connected. Suppose that for $i \in \{1, 2\}$, $\cM_{\bfC_i}$ is one of the following: a stratum of Abelian differentials, an Abelian double, or a quadratic double. Then $\cM$ is also a stratum of Abelian differentials, an Abelian double, or a quadratic double. 

Moreover, if $\MOne$ and $\MTwo$ are both quadratic doubles, so is $\cM$.  
\end{prop}

The reader may wish to begin by taking Proposition \ref{P:DiamondsAreCool} as a black box. The proof follows from results in \cite{ApisaWrightDiamonds}, where a similar result is proved without the assumption that $\cM$ is geminal. Without the geminal assumption, there are some additional possibilities for $\cM$, with explicit but sometimes lengthy descriptions. (There are even more possibilities without the assumption that $\ColOneTwoX$ is connected.) To prove Proposition \ref{P:DiamondsAreCool}, we need to show that the additional possibilities are not geminal.

In the sequel, as in Definition \ref{D:F}, we will let $\For(X, \omega)$ denote a translation surface $(X, \omega)$ once all marked points on it are forgotten. Similarly, if $(X, \omega)$ has dense orbit in an invariant subvariety $\cM$, we will let $\For(\cM)$ denote the orbit closure of $\For(X, \omega)$. 

\begin{proof}
We will proceed in three cases.

\bold{Case 1: At least one of $\MOne$ and $\MTwo$ is a component of a stratum of Abelian differentials.} Without loss of generality, assume that $\MOne$ is a component of a stratum of Abelian differentials.

We begin with two observations that can be found in \cite[Proposition 5.1]{ApisaWrightDiamonds}. First, if $\MTwo$ is a also a component of a stratum of Abelian differentials, then so is $\cM$. Second, it is not possible for $\MTwo$ to be an Abelian double. 

This leaves the case where $\MTwo$ is a quadratic double. Let $\cH$ be the component of the stratum of Abelian differentials containing $\cM$. By \cite[Proposition 5.1]{ApisaWrightDiamonds}, one of the following occurs: 
\begin{enumerate}
\item $\cM$ is a component of a stratum of Abelian differentials or 
\item $\For(\cM) = \For(\cH)$ is a hyperelliptic component and there is at most one free marked point on surfaces in $\cM$ with the remaining marked points being a collection of either one marked point fixed or two marked points exchanged by the hyperelliptic involution. 
\item $\For(\cM)$ is a codimension one hyperelliptic locus in $\For(\cH)$ and there is at most one marked point on surfaces in $\cM$, which is free.
\end{enumerate}

Our analysis will make use of the following. 

\begin{lem}\label{L:NotAllFree}
Any geminal subvariety that is not a component of a stratum of Abelian differentials contains a surface with a pair of twin cylinders. 
\end{lem} 
Note that Abelian doubles can never be components of strata, but quadratic doubles can be hyperelliptic connected components of strata of Abelian differentials.
\begin{proof}
Otherwise every cylinder on every surface would be free, and \cite[Theorem 1.5]{MirWri2} would give that the  subvariety is in fact a component of a stratum of Abelian differentials. 
\end{proof}

Suppose first that $\cM$ is described by the third possibility above. If the surfaces in $\cM$ do not have free marked points, then $\cM$ is a quadratic double, as desired. Suppose therefore, in order to derive a contradiction, that the surfaces in $\cM$ contain a free marked point. Since $\For(\cM)$ is codimension 1, $\cM$ is not a stratum, and hence Lemma \ref{L:NotAllFree} gives the existence of a pair of twin cylinders on some $(X,\omega)\in \cM$. Moving the free marked point into such a pair on $(X,\omega)$ alters the height of one cylinder, but not its twin. (Note that a free marked point cannot lie on the boundary of a pair of twin cylinders.) This contradicts the assumption that $\cM$ is geminal. 


It remains to consider the second case listed above, i.e. where $\For(\cM) = \For(\cH)$. If there is no free marked point, then $\cM$ is a quadratic double, as desired. So suppose in order to derive a contradiction that $p$ is a free marked point on $(X, \omega)$. Let $\For'(X, \omega)$ denote $(X, \omega)$ with $p$ forgotten and let $\For'(\cM)$ denote its orbit closure. 

The remaining one or two marked points show that $\For'(\cM)$ is not a stratum. So, since $\For'(\cM)$ is a quadratic double, by Lemma \ref{L:NotAllFree} it contains a surface with a pair of twin cylinders. As before, moving $p$ into one of these cylinders alters its height, but not that of its twin, contradicting the assumption that $\cM$ is geminal. 

\bold{Case 2: At least one of $\MOne$ and $\MTwo$ is an Abelian double, and neither is a component of a stratum.} Without loss of generality, assume $\MOne$ is an Abelian double.  

If $\MTwo$ is also an Abelian double, $\cM$ is a full locus of double covers by Lemma \ref{L:IntroFull}, since Assumption CP holds for all Abelian doubles. To show that $\cM$ is an Abelian double we must show that if $f: (X, \omega) \ra (Y, \eta)$ is the double cover, then for any marked point $p$ on $(Y, \eta)$, both preimages of $p$ under $f$ are marked on $(X, \omega)$. Suppose in order to find a contradiction that, for some marked point $p$, one preimage of $p$ is marked but not the other.  

By Lemma \ref{L:NotAllFree}, we may assume that there is cylinder $C$ on $(Y, \eta)$ that has two preimages on $(X, \omega)$. Moving the marked preimage of $p$ into a preimage of $C$ alters the height of one cylinder, but not its twin. This  contradicts the assumption that $\cM$ is geminal.

Therefore, suppose that $\MTwo$ is a quadratic double. Since $\ColOneTwoX$ is connected, Apisa-Wright \cite[Theorem 10.1]{ApisaWrightDiamonds}, implies that $\cM$ is an Abelian double, a quadratic double, or that $\cM$ is a full locus of covers of a codimension one locus $\cN$ in a component of a stratum of Abelian differentials $\cH$, where $\For(\cN) = \For(\cH)$ and $\For(\cH)$ is a hyperelliptic component of rank at least two. It suffices to rule out the final case. Suppose, therefore, in order to derive a contradiction, that $\cM$ is not a quadratic double and that it is described by the final case. We have the following,  
\begin{itemize} 
\item surfaces in $\cM$ have a pair of marked points $Q$ that project, under the cover, to a pair of points $Q'$ exchanged by the hyperelliptic involution, by  \cite[Theorem 10.60 (3b-3) and Lemma 10.62]{ApisaWrightDiamonds},
\item the set $P'$ of remaining marked points on surfaces in $\cN$  is either empty or consists of a single free point, and 
\item if $P'$ is empty, then $\cM$ is contained in a quadratic double of a stratum $\cQ$ of surfaces with four poles\footnote{There is a possibility in \cite[Theorem 10.60 (3b-1)]{ApisaWrightDiamonds} with only two poles, but it is not relevant here, since in that case there are no marked points and, when $\ColOneTwoX$ is connected, there must be marked points on $(X, \omega)$.}, and more precisely $\For(\cM)$ is the set of holonomy double covers of surfaces in a codimension one hyperelliptic locus in $\For(\cQ)$, by  \cite[Theorem 10.60 (3b-1)]{ApisaWrightDiamonds}. 
\end{itemize}

We begin with the following, which will allow us to make use of the last point above. 


\begin{sublem}\label{SL:Pempty}
  $P'$ is empty.
\end{sublem}
\begin{proof}
Suppose not. Let $T_0$ denote the translation involution on $(X, \omega)$ whose quotient is a surface in $\cN$. Since every cylinder on $\cF((X,\omega)/T_0)$ is fixed by the hyperelliptic involution, we can assume both points of $Q'$ are in the same cylinder of $\cF((X,\omega)/T_0)$. Move $P$ into this cylinder, as in Figure \ref{F:PQnotGeminal}. Considering the different possibilities for the preimage of this cylinder on $(X,\omega)$ shows that $\cM$ is not geminal. 
\begin{figure}[h]
\includegraphics[width=0.3\linewidth]{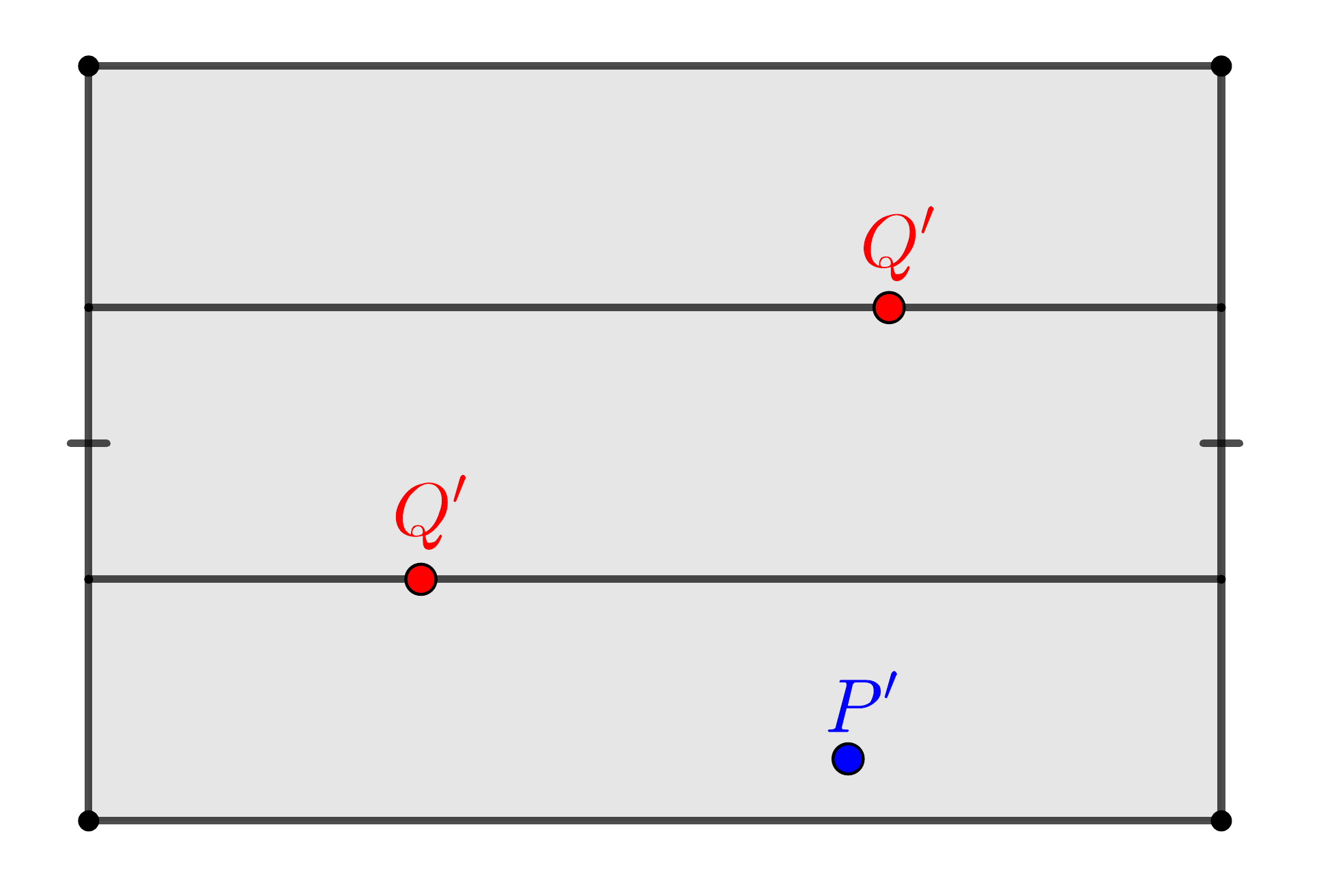}
\caption{The proof of Sublemma \ref{SL:Pempty}.}
\label{F:PQnotGeminal}
\end{figure}  
\end{proof}

Therefore, $\cM$ is contained in a quadratic double of a component $\cQ$ of a stratum with at least four poles; moreover, $\For(\cM)$ corresponds to a codimension one hyperelliptic locus in $\For(\cQ)$. 

\begin{lem}\label{L:FixedTwins}
Let $\cQ\neq \cQ(-1^4)$ be a component of a stratum of quadratic differentials with at least four poles and no marked points, and let $\cL$ be a nonempty component of the hyperelliptic locus in $\cQ$. Then $\cL$ contains a surface whose holonomy double cover, considered as a surface without marked points, has a pair of twin cylinders each of which is fixed by the holonomy involution.
\end{lem} 

\begin{proof}
Let $(Q, q)\in \cQ$ be hyperelliptic, and let $(Q_0, q_0)$ be its quotient by the hyperelliptic involution. 

Note that a pole cannot be fixed by the hyperelliptic involution. Hence, $(Q_0, q_0)$ has a pair of poles $z_1, z_2$ over which the map $(Q, q) \to (Q_0, q_0)$ is not ramified. Deforming $(Q_0, q_0)$, we can assume without loss of generality that there is a saddle connection joining $z_1$ to $z_2$. (For example, if the two poles in question are moved close together, while the location on $\bP^1$ of the remaining poles and zeros are fixed, there will be such a saddle connection.)
There is a cylinder  $C_0$ with this saddle connection in its boundary; in \cite{ApisaWrightDiamonds}  we call this type of cylinder an envelope.

Since $(Q,q)\notin \cQ(-1^4)$, the envelope $C_0$ lifts to a pair of envelopes on $\cF(Q,q)$. The preimages of these two cylinders on the holonomy double cover give the desired cylinders.
\end{proof}

Let $C_1$ and $C_2$ be a pair of twins given by Lemma \ref{L:FixedTwins}. Let $J$ denote the holonomy involution on $(X, \omega)$. Since the marked points on $(X, \omega)$ are $J$-invariant we can move them into $C_1$, as in Figure \ref{F:FixedTwins}. This creates a cylinder that is not free and that does not have a twin,  contradicting the fact that $\cM$ is geminal.
\begin{figure}[h]
\includegraphics[width=0.6\linewidth]{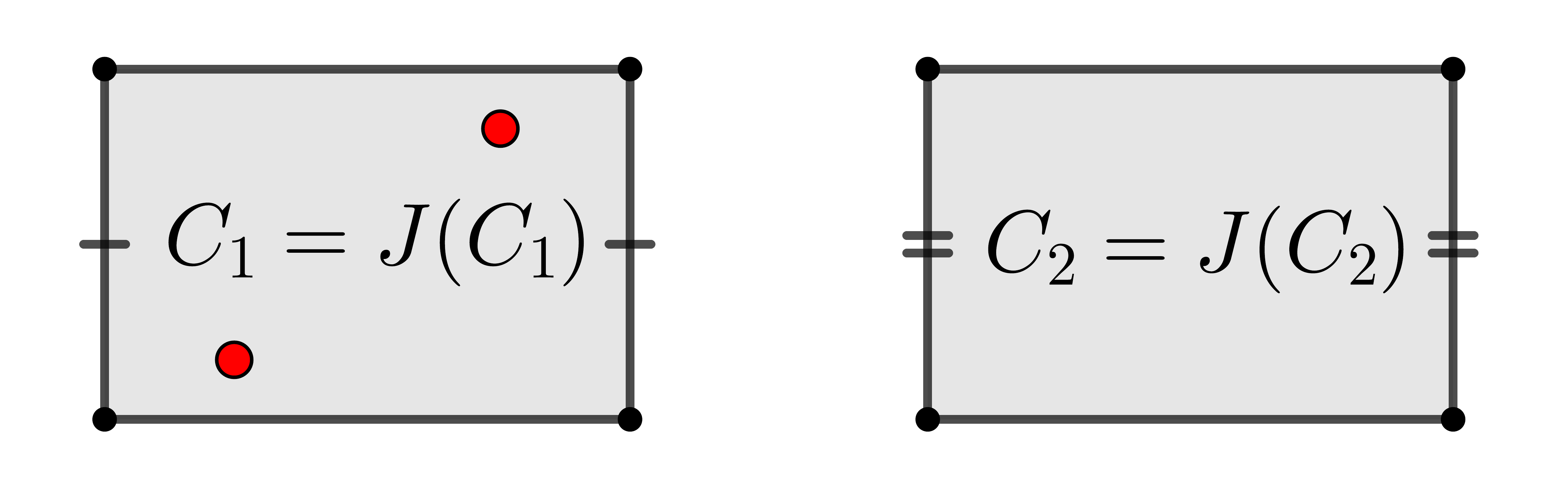}
\caption{The conclusion of Case 2 in the proof of Proposition \ref{P:DiamondsAreCool}.}
\label{F:FixedTwins}
\end{figure}

\bold{Case 3: Both $\MOne$ and $\MTwo$ are quadratic doubles.}  By Apisa-Wright \cite[Theorem 7.1]{ApisaWrightDiamonds}, since $\ColOneTwoX$ is connected, $\cM$ is a quadratic double. 
\end{proof}

\section{Geminal  subvarieties of rank one}\label{S:GeminalRank1}

This section classifies geminal invariant subvarieties of rank one. The main result is Proposition \ref{P:geminalrk1}. 

By Lemma \ref{L:GeminalField}, any geminal invariant subvariety $\cM$ has $\bk(\cM) = \mathbb{Q}$. By Lemma \ref{L:R1Arithmetic}, it follows that every rank one geminal invariant subvariety is a locus of torus covers. For this reason it will be useful to begin our investigation of rank one geminal invariant subvarieties by analyzing the geminal invariant subvarieties in strata of flat tori with marked points. 

The simplest examples of geminal subvarieties in strata of flat tori are given by the strata themselves. The next simplest are Abelian and quadratic doubles. A slightly less obvious example is the following.

\begin{defn}\label{D:TxT}
For any positive integer $n$, an invariant subvariety $\cM$ in $\cH(0^n)$ is called a \emph{$T \times T$ locus} if it is a full locus of regular covers of a stratum (necessarily $\cH(0^m)$ for some $m$) with deck group $\bZ/2\bZ \times \bZ / 2\bZ$, where at least 3 of the 4 points in the fiber over each marked point are marked. 
\end{defn}

\begin{rem}\label{R:TxTCP}
The condition on marked points is equivalent to the $\bZ/2\bZ \times \bZ / 2\bZ$ covers satisfying Assumption CP, which was defined in Definition \ref{D:CP}.  Indeed, if there are two unmarked points $p_1, p_2$ in the fiber of a marked point, then there is cylinder whose core curve contains $p_1$, $p_2$ which shows that the cover does not satisfy Assumption CP; see Figure \ref{F:NotCPp1p2}. 
\begin{figure}[h!]
\includegraphics[width=.4\linewidth]{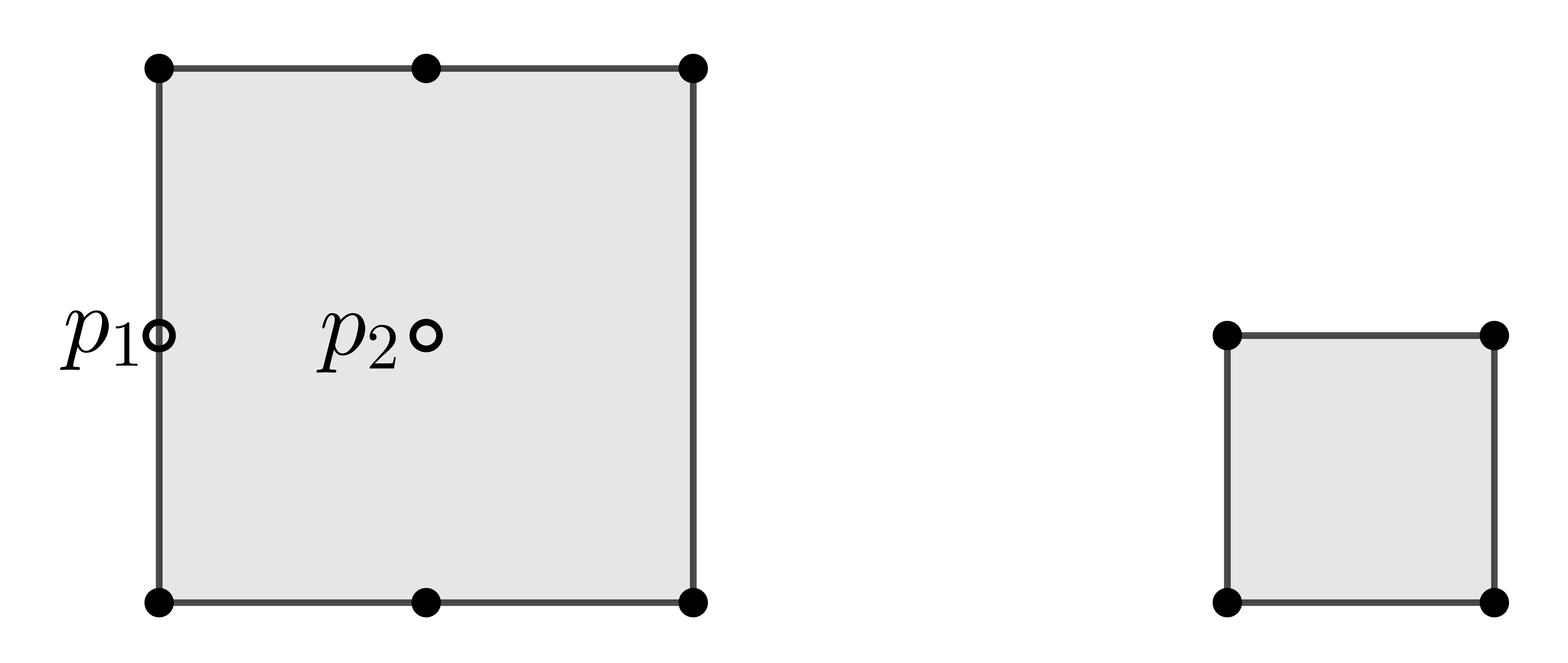}
\caption{A cover not satisfying Assumption CP.}
\label{F:NotCPp1p2}
\end{figure}
Conversely, if the marked point condition is satisfied, then in every cylinder direction on the cover there are two isometric cylinders that are exchanged by an element of the deck group. 
\end{rem}

The following result says that every rank one geminal subvariety is a full locus of covers of one of the examples just discussed. Recall that optimal maps are defined in Definitions \ref{D:GoodAndOptimal} and \ref{D:MGoodAndOptimal}, and $h$-geminal is defined in Definition \ref{D:hGeminal}. Subequivalence classes and generic cylinders are defined in Definition \ref{D:SE}, and here subequivalence classes consist of either a free cylinder or a pair of twins. 

\begin{prop}\label{P:geminalrk1}
Let $\cM$ be a geminal orbit closure of rank 1. 
\begin{enumerate}
\item\label{I:geminalrk1:WhatIsM} $\cM$ is one of the following:
\begin{enumerate}
\item\label{I:geminalrk1:ComponentOrDouble} A connected component of a stratum of Abelian differentials or an Abelian or quadratic double. If $\cM$ is not $h$-geminal then it is described by this case. 
\item\label{I:geminalrk1:CoverOfQuadDouble} A  full locus of covers of a quadratic double, where the quadratic double is contained in $\cH(0^n)$ for some $n$. 
\item\label{I:geminalrk1:CoverOfTxT} A full locus of covers of a $T \times T$ locus. 
\end{enumerate}
\item\label{I:geminalrk1:PiOpt}  Any surface in $\cM$ has an $\cM$-optimal map $\pi_{opt}$, and:
\begin{enumerate}
\item\label{I:geminalrk1:ComponentOrDouble:PiOpt} In Case \eqref{I:geminalrk1:ComponentOrDouble}, $\pi_{opt}$ is the identity. 
\item\label{I:geminalrk1:CoverOfQuadDouble:PiOpt} In  Case \eqref{I:geminalrk1:CoverOfQuadDouble}, $\pi_{opt}$ is the covering map to surfaces in the quadratic double, and is moreover the minimal degree map to a torus.
\item\label{I:geminalrk1:CoverOfTxT:PiOpt} In Case \eqref{I:geminalrk1:CoverOfTxT}, $\pi_{opt}$ is the covering map to surfaces in the $T\times T$ locus, and is moreover the minimal degree map to a torus.
\end{enumerate}
\item\label{I:geminalrk1:TwoCylinders} If the degree of the optimal map is greater than one, then each subequivalence class contains exactly two cylinders. 
\item\label{I:geminalrk1:PiOptDegeneratesToPiOpt} If $\bfC$ is a subequivalence class of $\cM$-generic cylinders on a surface $(X,\omega)$ in $\cM$, then $\Col_{\bfC}(\pi_{opt})$ is the $\cM_{\bfC}$-optimal map for $\Col_{\bfC}(X, \omega)$. 
\end{enumerate}
\end{prop}

%
%

\begin{rem}
The Eierlegende-Wollmilchsau is an example of a surface described by Case \eqref{I:geminalrk1:CoverOfTxT} where the degree of the optimal map is greater than one. The image of this surface under the optimal map is a torus $E$ with four marked points, all of which differ by two-torsion. Let $T_1$ and $T_2$ be two distinct maps that are not the identity and that preserve the marked points on $E$. Let $P$ be any collection of marked points with the property that if $p$ belongs to $P$ then at least two points in $\{T_1(p), T_2(p), T_1T_2(p)\}$ are also marked. It is easy to see that marking all the preimages of $P$ on the Eierlegende-Wollmilchsau produces a geminal orbit closure with arbitrarily large rel and whose image under the optimal map is a $T \times T$ locus that is not a quadratic double.
\end{rem}

\begin{rem}
Regarding \eqref{I:geminalrk1:CoverOfQuadDouble:PiOpt} and \eqref{I:geminalrk1:CoverOfTxT:PiOpt}, one may wish to keep in mind that some orbit closures may be full loci of covers in more than one way. In \eqref{I:geminalrk1:CoverOfQuadDouble:PiOpt} and \eqref{I:geminalrk1:CoverOfTxT:PiOpt}, $\pi_{opt}$ is the minimal degree map to a torus and this map   realizes $\cM$ as a locus of covers in the indicated way. We do not  claim however that there is no other way to realize $\cM$ as a locus of covers in the indicated way.
\end{rem}

A corollary of Proposition \ref{P:geminalrk1}, which will actually be an ingredient in its proof, is the following.

\begin{prop}\label{P:0n}
If $\cM$ is a geminal subvariety of $\cH(0^n)$ then $\cM$ is one of the following: a stratum, an Abelian double, a quadratic double, or a $T \times T$ locus.
\end{prop}

There is a very limited amount of overlap between different possibilities in Proposition \ref{P:0n}.

\begin{lem}\label{L:Types}
The only geminal subvarieties of $\cH(0^n)$ that are described simultaneously by two of the possibilities listed in Proposition \ref{P:0n} are the following:
\begin{enumerate}
\item The quadratic double of $\cQ(-1^4)$ with one preimage of a pole marked is the stratum $\cH(0)$.
\item The quadratic double of $\cQ(-1^4)$ with two preimages of poles marked is an Abelian double of $\cH(0)$.
\item The quadratic double of $\cQ(-1^4)$ with three or four preimages of poles marked is a $T \times T$ locus in $\cH(0^3)$ or $\cH(0^4)$. 
\item The quadratic double of $\cQ(-1^4, 0)$ with no preimages of poles marked is the stratum $\cH(0,0)$. 
\end{enumerate}
\end{lem}
\begin{proof}
Let $\cM$ be a subvariety of $\cH(0^n)$ where $n$ is a positive integer. We begin by observing that it is impossible for $\cM$ to be simultaneously more than one of the following: a stratum, an Abelian double, or a $T \times T$ locus. 

Therefore, suppose that $\cM$ is a quadratic double. Since $\cM$ contains flat tori it follows that $\cM$ is a quadratic double of $\cQ(-1^4, 0^m)$ for some integer $m$. 

\bold{If $\cM$ is also a stratum:} By definition of a quadratic double, if the preimage of a pole is marked then any other marked points cannot be free. Since $\cM$ is also a stratum, there aren't any other marked points and we get that $\cM$ is the first possibility.  Similarly, if no preimages of poles are marked then $m = 1$ and $\cM$ is the final possibility. 

\bold{If $\cM$ is also a $T\times T$ locus:} On surfaces in $\cM$ all marked points that are not fixed points of the holonomy involution can be moved arbitrarily close together while fixing the underlying unmarked torus. Hence if $\cM$ is an Abelian double or a $T \times T$ locus, only fixed points of the holonomy involution are marked so $m = 0$ and $\cM$ is the third possibility.

\bold{If $\cM$ is also an Abelian double:} A similar analysis shows $\cM$ is the second possibility. 
\end{proof}

\subsection{Proof of Proposition \ref{P:0n}}\label{SS:CQLR1} 

Suppose that $\cM$ is a geminal subvariety contained in $\cH(0^n)$ where $n$ is a positive integer. Let $k-1$ denote the rel of $\cM$.  We will proceed by induction on $k$, with the base case being $k=1$. 

By  \cite[Theorem 1.10]{Wcyl}, there is a surface $(X, \omega)$ in $\cM$ that is horizontally and vertically periodic with $k$ subequivalence classes of horizontal cylinders. Label the subequivalence classes of horizontal cylinders by $\{1, \hdots, k\}$, and let $\bfC_i$ denote the subequivalence class labelled by $i$. 

Consider the cyclic order in which the subequivalence classes of horizontal cylinders appear along a leaf of the vertical foliation. Since subequivalence classes contain at most two elements, each label appears at most twice in the cyclic ordering. Moreover, by Lemma \ref{L:GeminalTwinAdjacency}, if a cylinder labeled $i$ is adjacent to a cylinder labeled $j$, then every cylinder labeled $i$ is adjacent to a cylinder labeled $j$. The only cyclic orderings satisfying these two conditions are, up to re-indexing:

\begin{itemize}
\item $1, 2, \ldots, k$,
\item $1,2, \ldots, k-1, k, 1, 2, \ldots, k-1, k$,
\item $1,2, \ldots, k-1, k, k, k-1, \ldots, 2, 1$, 
\item $1,2, \ldots, k-1, k, k-1, \ldots, 2, 1$, 
\item $1,2, \ldots, k-1, k, k-1, \ldots, 2$, for $k>2$.
\end{itemize}

\begin{rem}\label{R:LastCase} 
We will want to consider the cyclic order $1,2$ as an instance of the first case and not the final case, which is why we require $k>2$ in the final case. Note that up to re-indexing, $1,2,1$ is the same cyclic order as $1,1,2$ or $1, 2, 2$.  
\end{rem}

  If, for every surface $(X,\omega)\in \cM$ as above, the cyclic order is the first one above, then \cite[Theorem 1.5]{MirWri2} gives  that $\cM$ is a stratum. So we can assume that the cyclic ordering is not the first one listed.   We now establish the base of our induction.


\begin{lem}\label{L:R1Base}
When $k=1$, then $\cM$ is a quadratic double of $\cQ(-1^4)$ with two, three, or four preimages of poles marked.

In particular, Proposition \ref{P:0n} holds when $k=1$.
\end{lem}
\begin{proof}
When $k=1$, the horizontal direction on $(X, \omega)$ has exactly two cylinders - call them $C$ and $C'$ - which are isometric. We can  assume the vertical direction has either one or two cylinders, each of which crosses each of $C$ and $C'$ only once, as in Figure \ref{F:kEquals2}. 

\begin{figure}[h!]
\includegraphics[width=.4\linewidth]{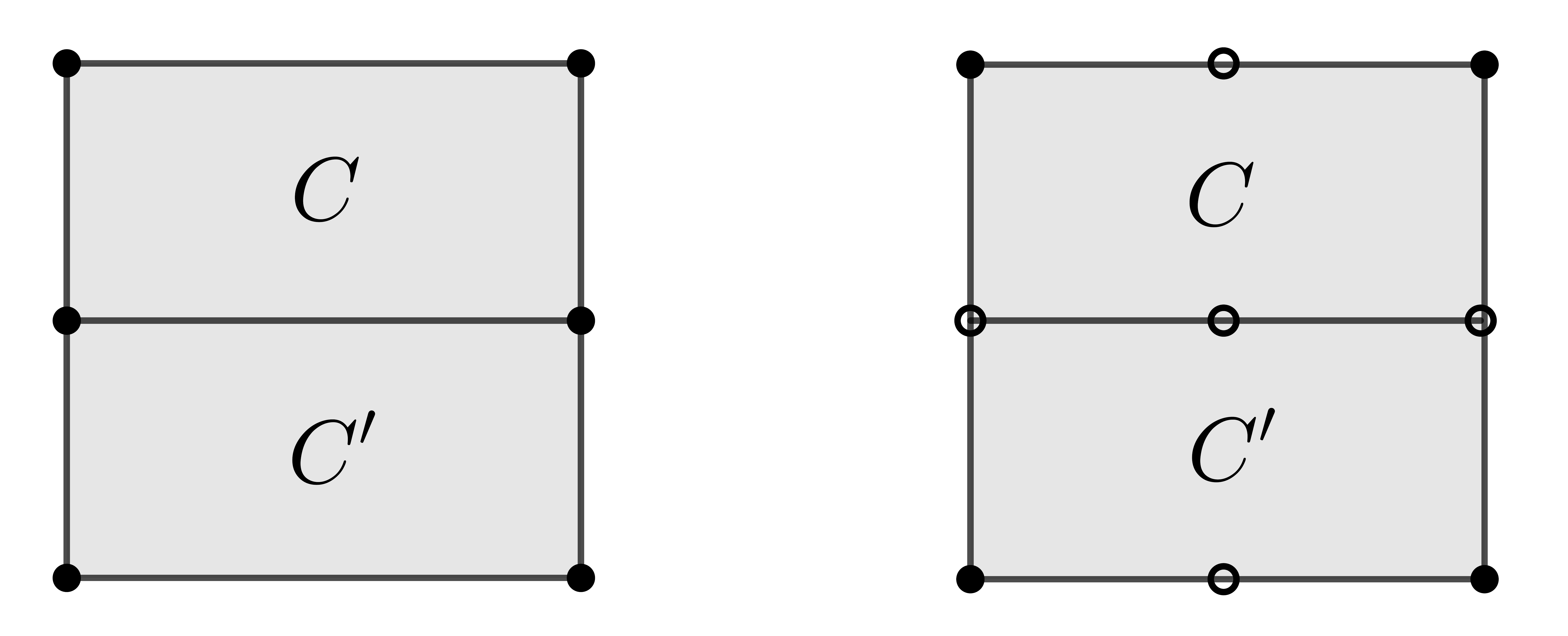}
\caption{The proof of Lemma \ref{L:R1Base}. Left: the case of one vertical cylinder. Right: The case of two vertical cylinders. }
\label{F:kEquals2}
\end{figure}

If the vertical direction has only one cylinder, $\cM$ is contained in $\cH(0,0)$ and is simultaneously an Abelian double of $\cH(0)$ and a quadratic double of $\cQ(-1^4)$ with two preimages of poles marked.  If the vertical direction contains two cylinders, then $\cM$ is a quadratic double of $\cQ(-1^4)$ with two, three or four preimages of poles marked.
\end{proof}

\begin{lem}\label{L:TorTxT}
If the cyclic order is $1, \hdots, k, 1 \hdots, k$. Then $\cM$ is an Abelian double or a $T \times T$ locus. 
\end{lem}
\begin{proof}
Proceed by induction on $k$. The base case is Lemma \ref{L:R1Base} combined with the observations in Lemma \ref{L:Types} that the relevant quadratic doubles of $\cQ(-1^4)$ are simultaneously Abelian doubles or $T \times T$ loci. 

Suppose now that $k > 1$. After replacing $(X,\omega)$ with a version where each $\bfC_i$ has been sheared, we may assume that each $\bfC_i$ contains a cylinder that contains a vertical saddle connection. 

By the induction hypothesis, for each $i$, the surface $\Col_{\bfC_i}(X, \omega)$ belongs to an Abelian double or $T \times T$ locus. Since each $\Col_{\bfC_j}(\bfC_i)$, for $i \ne j$, contains a cylinder that contains a vertical saddle connection, it follows that there is one (resp. two) vertical cylinder(s) that covers $\Col_{\bfC_j}(X, \omega)$ when $\cM_{\bfC_j}$ is an Abelian double (resp. $T \times T$ locus). The surface $(X, \omega)$ is also covered by the same number (one or two) of vertical cylinders, as in Figure \ref{F:123123}.

\begin{figure}[h!]
\includegraphics[width=.2\linewidth]{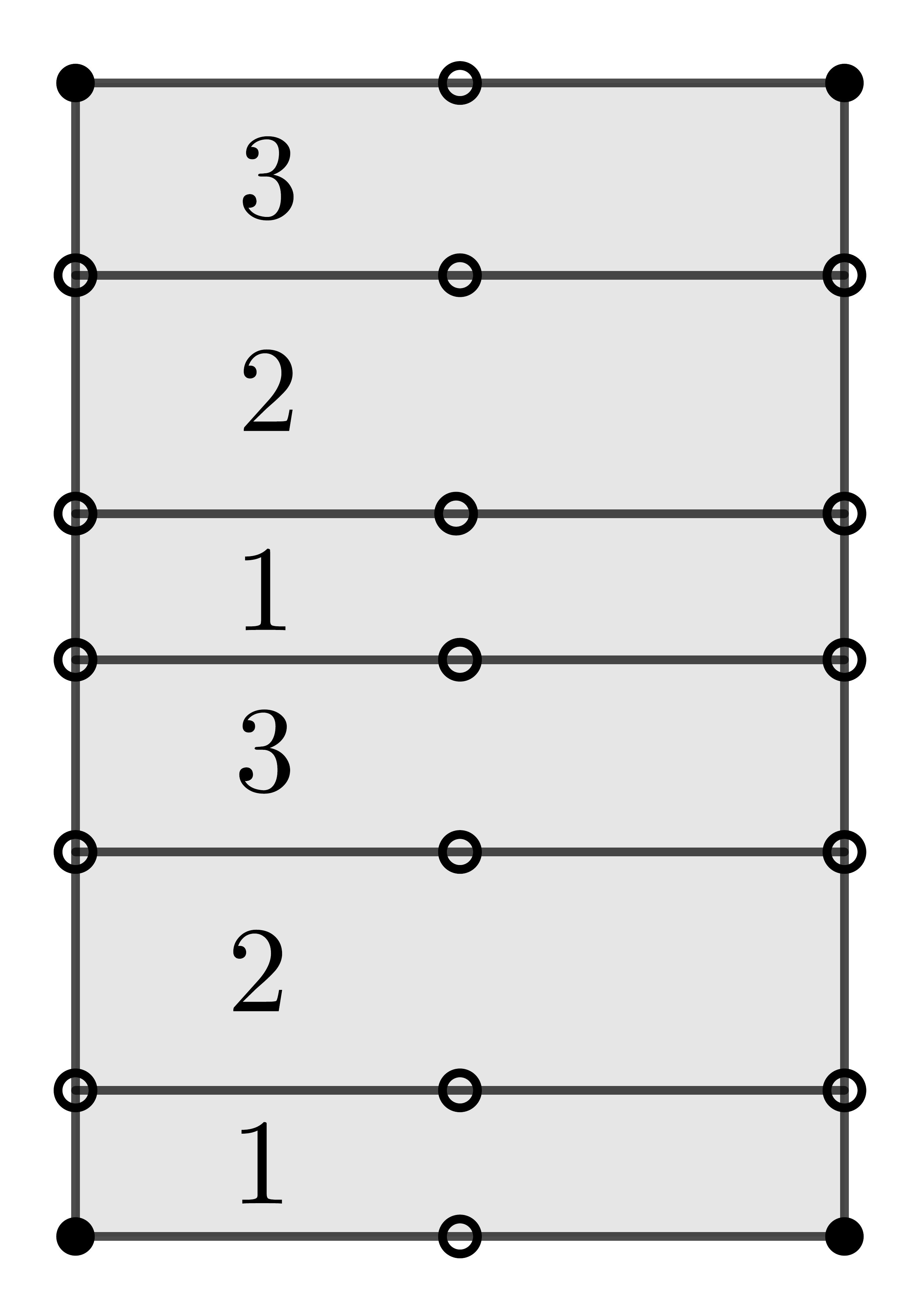}
\caption{The proof of Lemma \ref{L:TorTxT} when $k=3$.}
\label{F:123123}
\end{figure}

In the case that there is a single vertical cylinder that covers $(X, \omega)$ it is easy to see that $\cM$ is an Abelian double. 

Suppose therefore that $(X, \omega)$ is covered by two vertical cylinders. Let $T_1$ be the translation involution that fixes each vertical cylinder but doesn't fix any horizontal cylinder, and let $T_2$ be the translation involution that fixes each horizontal cylinder and exchanges the two vertical cylinder, as in Figure \ref{F:T1T2}.

\begin{figure}[h!]
\includegraphics[width=.35\linewidth]{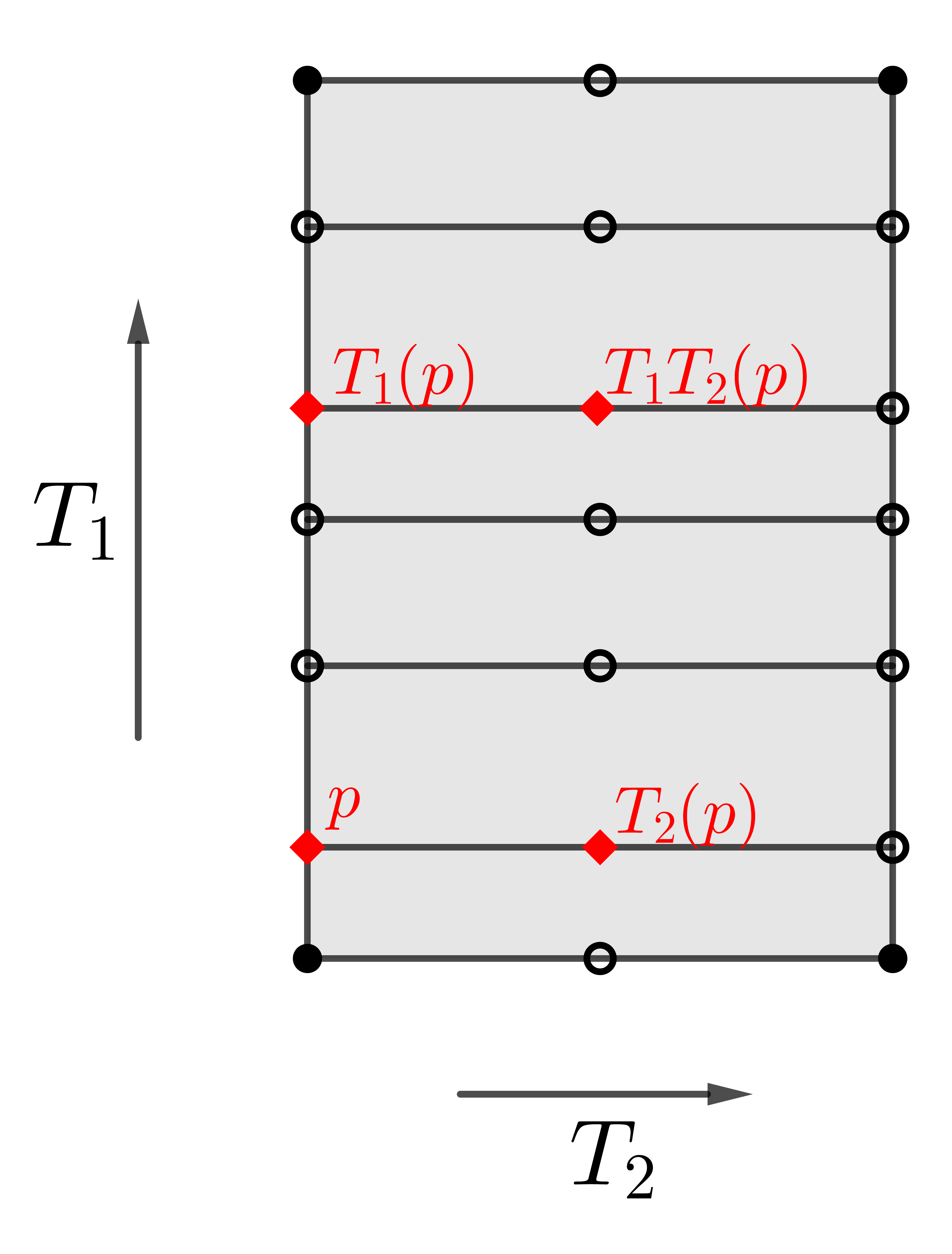}
\caption{The proof of Lemma \ref{L:TorTxT} when $k=3$.}
\label{F:T1T2}
\end{figure}

Suppose now that $\cM$ is not an Abelian double. We need to show that $\cM$ is a $T \times T$ locus. We will require the following two sublemmas.

\begin{sublem}\label{SL:TxTConstraint}
For any marked point $q$ on $(X, \omega)$ at least one of $T_1T_2(q)$ and $T_2(q)$ is marked. Moreover, at least one of $T_1(q)$ and $T_1T_2(q)$ is marked.
\end{sublem}
\begin{proof}
The second claim is immediate since $T_1$ exchanges horizontal twin cylinders. In particular, if $q$ lies on the top boundary of a horizontal cylinder $C$, then in order for $T_1(C)$ to be a cylinder, either $T_1(q)$ or $T_1T_2(q)$ must be a marked point lying on the top boundary of $T_1(C)$. We now turn to the first claim. 

We begin by showing that there is some marked point for which the first claim holds. Suppose to a contradiction that this is not the case. In particular, suppose that for every marked point $p$, neither $T_1T_2(p)$ nor $T_2(p)$ are marked. This would imply that for every marked point $p$, $T_1(p)$ is marked, which, given the supposition, implies that $\cM$ is an Abelian double. We may therefore suppose that there is some marked point $p$ such that at least one of $T_1T_2(p)$ or $T_2(p)$ is marked. 

We are now ready to prove the first claim. Suppose to a contradiction that there is a point $q$ such that neither $T_1T_2(q)$ nor $T_2(q)$ is marked. We have already seen that in this case $T_1(q)$ must be marked. Moreover, it is obvious that $q \ne p$ where $p$ is the point produced by the preceding paragraph. Then it is possible to move $\{q, T_1(q)\}$ into one vertical cylinder, contradicting the fact that the two vertical cylinders are twins. This shows that for every marked point $q$ on $(X, \omega)$ either $T_1T_2(q)$ or $T_2(q)$ is marked. 
\end{proof}

\begin{sublem}\label{SL:TxTConstraint2}
There is some marked point $q$ on $(X, \omega)$ so that at least two points in $\{ T_1(q), T_2(q), T_1T_2(q) \}$ are marked.
\end{sublem}
\begin{proof}
Suppose to a contradiction that the claim does not hold. By Sublemma \ref{SL:TxTConstraint}, for every marked point $p$ the only point in $$\{T_1(p), T_2(p), T_1T_2(p)\}$$ that is marked is $T_1T_2(p)$. In particular, this means that $\cM$ is an Abelian double contrary to our hypotheses. 
\end{proof}

In light of Sublemma \ref{SL:TxTConstraint2}, we fix a marked point $q$ so that at least two points in $\{ T_1(q), T_2(q), T_1T_2(q) \}$ are marked. If every marked point had this property then $\cM$ would be a $T \times T$ locus as desired. Therefore, suppose to a contradiction that there is some marked point $p$ so that the only point in $\{T_1(p), T_2(p), T_1T_2(p)\}$ that is marked is $T_1T_2(p)$. 

Suppose that $\bfC_1$ and $\bfC_2$ are the two subequivalence classes of cylinders that contain $p$ and $T_1 T_2(p)$ in their boundary. Now we collide all other marked points with the $\langle T_1, T_2 \rangle$ orbit of $q$, i.e. we form $\Col_{\bfC_3, \hdots, \bfC_k}(X, \omega)$. Note that $\cM_{\bfC_3, \hdots, \bfC_k}$ remains geminal by Lemma \ref{L:GeminalBoundary}. Without loss of generality the underlying unmarked translation surface is a square torus. The marked points can then be moved as in Figure \ref{F:NotGeminal}, showing that $\cM$ is not geminal and giving a contradiction. 
\begin{figure}[h]
\includegraphics[width=.3\linewidth]{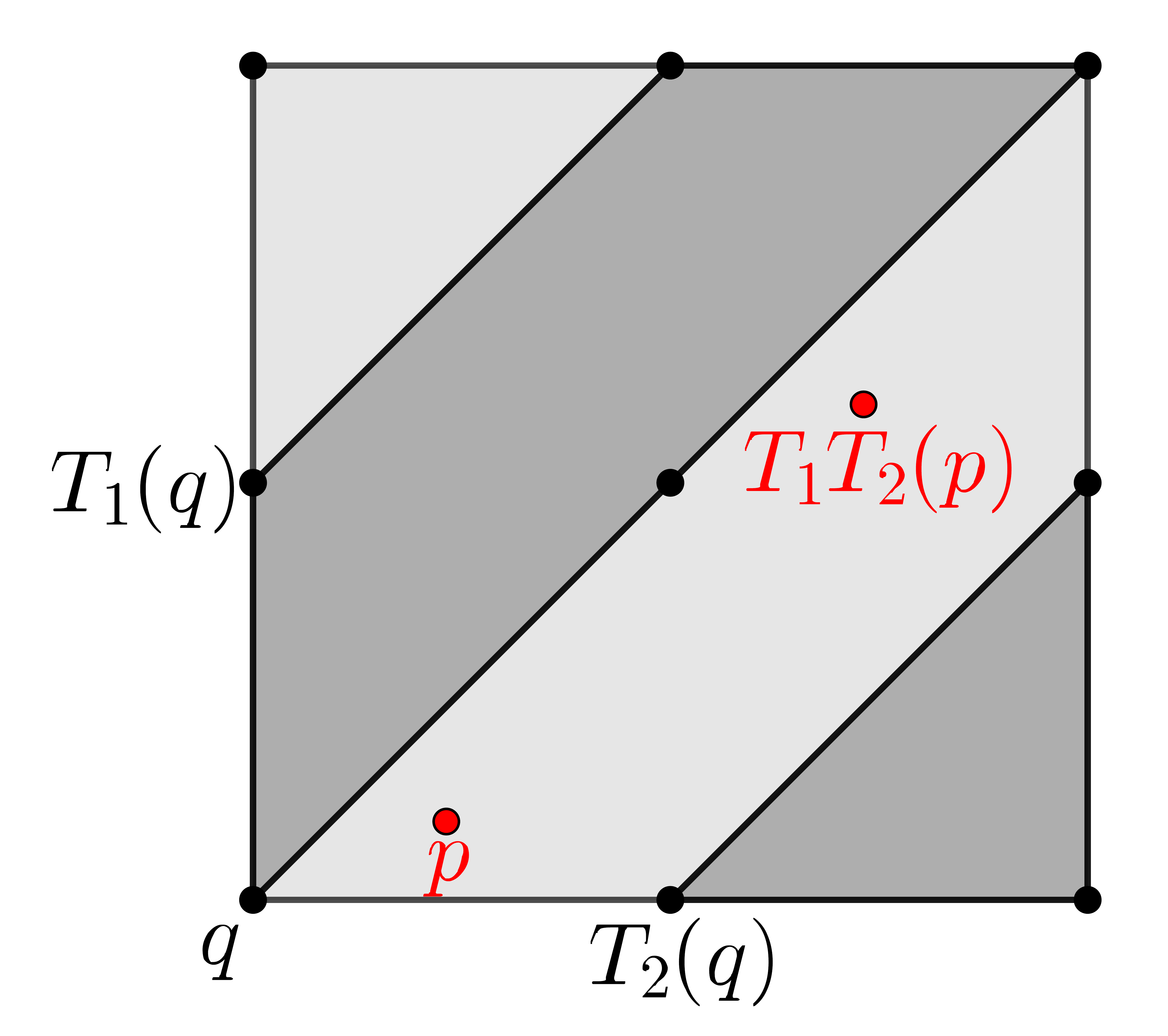}
\caption{The proof of Lemma \ref{L:TorTxT}. The slope 1 direction shows this does not describe a geminal orbit closure.}
\label{F:NotGeminal}
\end{figure}
\end{proof}

The proof will be complete once we establish the following. 

\begin{lem}
If $k > 1$ and the cyclic order is one of the last three, then $\cM$ is a quadratic double. 
\end{lem}
\begin{proof}
Proceed by induction on $k$. We give the inductive step before the base case, since the inductive step is short. Without loss of generality assume that a standard shear has been applied to each horizontal subequivalence class so that each horizontal subequivalence class contains a cylinder containing a vertical saddle connection. We will suppose moreover that there is a cylinder in $\bfC_1$ and a cylinder in $\bfC_2$ that contain vertical saddle connections, called $\sigma_1$ and $\sigma_2$ respectively, so that an endpoint of $\sigma_1$ coincides with an endpoint of $\sigma_2$.

\begin{rem}\label{R:23cyclic} 
  It may be helpful to keep in mind Remark \ref{R:LastCase} and note that the last three cyclic orderings for $k=3$ are $1,2,3,3,2,1$ and $1,2,3,2,1$ and $1,2,3,2$. Degenerating $\bfC_3$ results in $1,2,2,1$ and $1,2,2,1$ and $1,2,2$ respectively, and degenerating $\bfC_1$ results in $2,3,3,2$ and $2,3,2$ and $2,3,2$ respectively.  In particular, the prohibited $k=2$ case of the fifth cyclic ordering does not appear here. 
 
\end{rem}

Notice that when $k > 2$ that $\left( (X, \omega), \cM, \bfC_1, \bfC_3 \right)$ is a generic diamond where $\ColOneThreeX$ is connected and $\MOne$ and $\MThree$ are quadratic doubles. It follows in this case by Proposition \ref{P:DiamondsAreCool} that $\cM$ is a quadratic double.

Therefore, it remains to establish the $k = 2$ base case. In this case, the cyclic ordering is either $1,2,1$ or $1,2,2,1$. 
By Lemma \ref{L:R1Base}, $\ColTwoX$ is a torus  consisting of two horizontal cylinders and that has two, three, or four points marked, any two of which differ by two-torsion.

\bold{Case 1: $\ColTwoX$ has two marked points.} Figure \ref{F:H(0^n)EasyPartOfHardBaseCase} (left) illustrates $\ColTwoX$. In this case, $\ColTwo(\bfC_2)$ consists of one saddle connection and so $\bfC_2$ consists of one or two simple cylinders, i.e. cylinders both of whose boundaries consist of a single saddle connection.  We are done by Figure \ref{F:H(0^n)EasyPartOfHardBaseCase}. 

\begin{figure}[h!]
\includegraphics[width=0.7\linewidth]{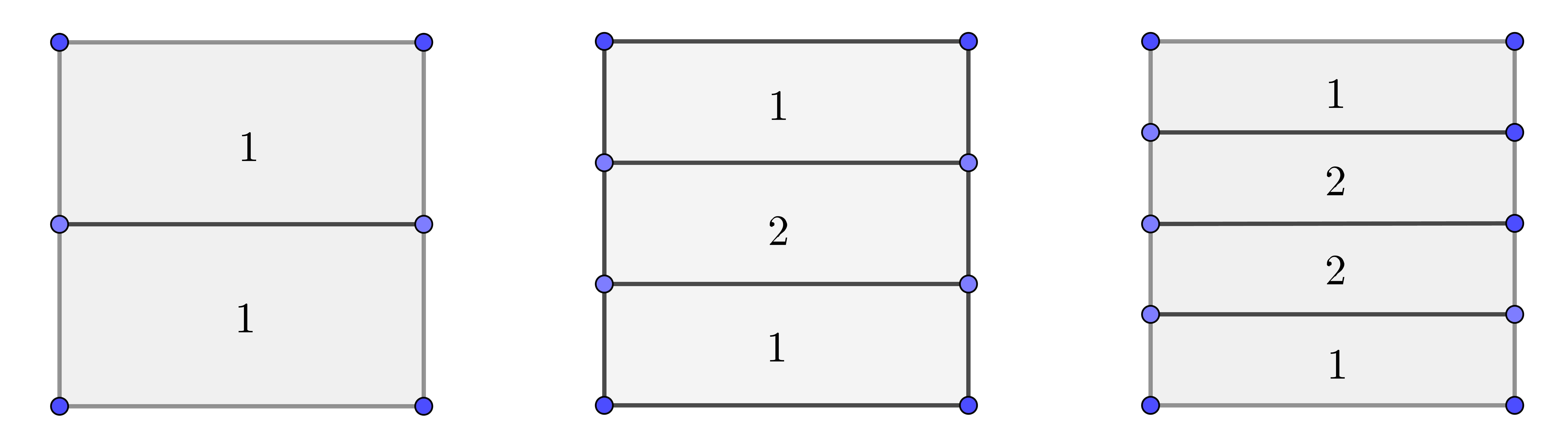}
\caption{Case 1. Left: $\ColTwoX$. Middle and right: the two possibilities for $(X,\omega)$.}
\label{F:H(0^n)EasyPartOfHardBaseCase}
\end{figure}

\bold{Case 2: $\ColTwoX$ has more than two marked points.} If there are two cylinders in $\bfC_2$ they share boundary saddle connections. Therefore $\ColTwo(\bfC_2)$ consists of either one or two saddle connections. 

\bold{Case 2a: $\ColTwo(\bfC_2)$ consists of a single saddle connection.}
If $\ColTwo(\bfC_2)$ consists of a single saddle connection then the cylinders in $\bfC_2$ were simple and the holonomy involution on $\ColTwoX$ extends to $(X, \omega)$ and preserves marked points; see Figure \ref{F:H(0^n)EasyPart2OfHardCase}.

\begin{figure}[h!]
\includegraphics[width=0.5\linewidth]{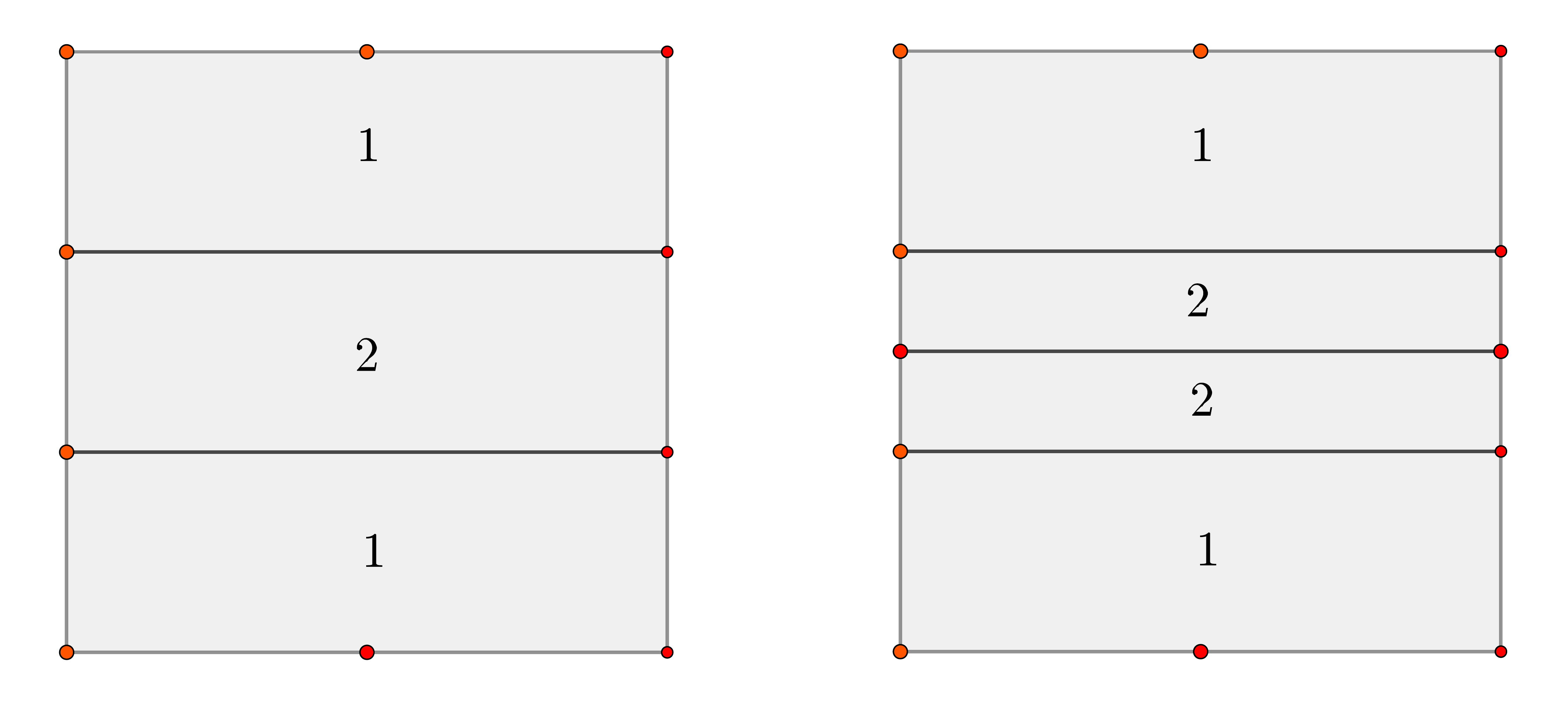}
\caption{The two possibilities for the surface $(X, \omega)$ in Case 2a.   }
\label{F:H(0^n)EasyPart2OfHardCase}
\end{figure}

\bold{Case 2b: $\ColTwo(\bfC_2)$ consists of two saddle connections and $\bfC_2$ has two cylinders.} In this case the original surface $(X, \omega)$ is depicted in Figure \ref{F:H(0^n)HardPart1HardCase}.

\begin{figure}[h!]
\includegraphics[width=0.4\linewidth]{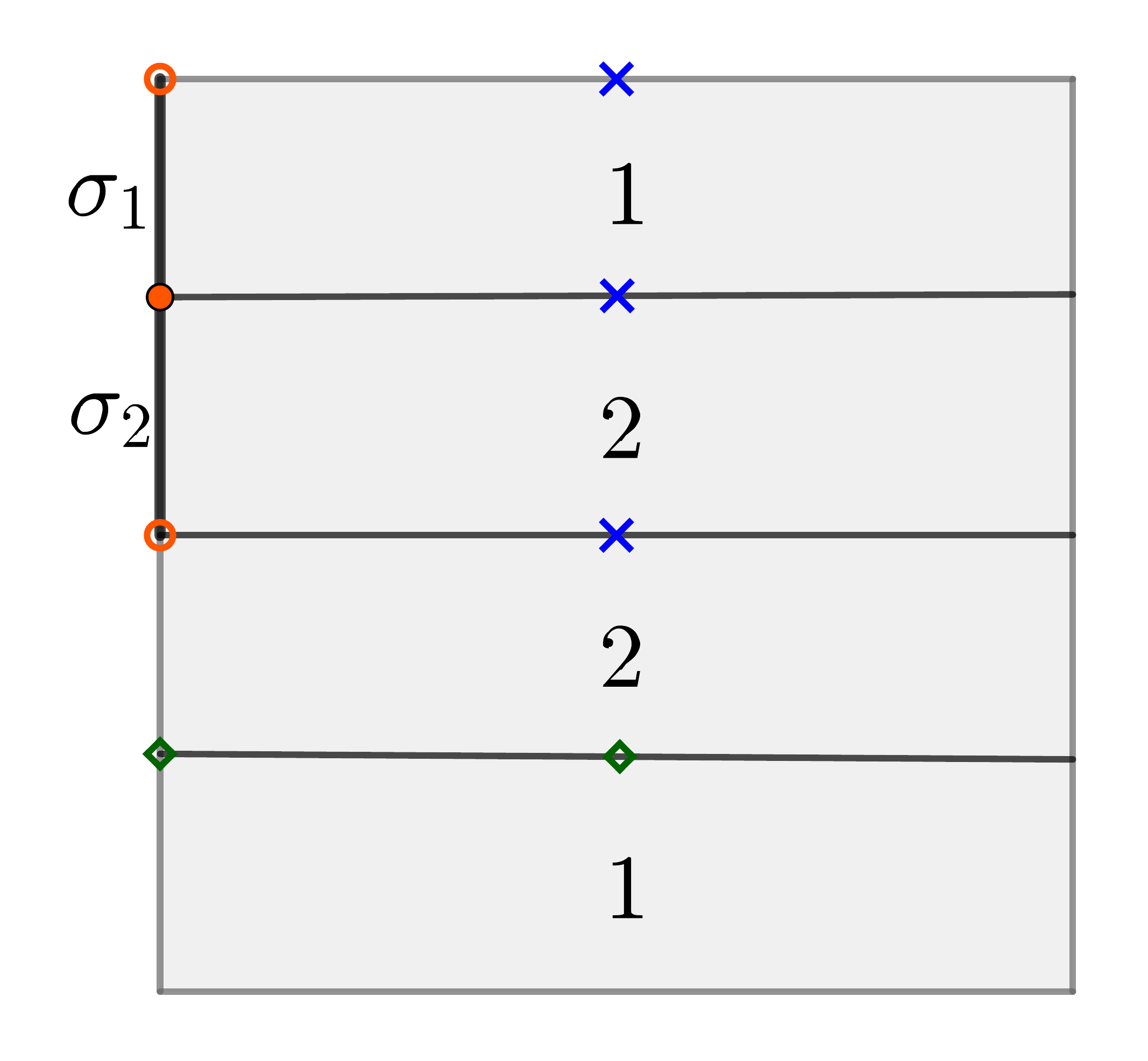}
\caption{The unfilled dots and the filled-in dot are marked. At least one of the two unfilled diamonds must be marked. 
Any remaining marked points must be at a location labelled with an ``x".
}
\label{F:H(0^n)HardPart1HardCase}
\end{figure}

By applying the standard shear in opposite directions to $\bfC_1$ and $\bfC_2$ we can move the marked points while fixing the underlying flat torus as in Figure \ref{F:H(0^n)Hardest2}. 

\begin{figure}[h!]
\includegraphics[width=0.4\linewidth]{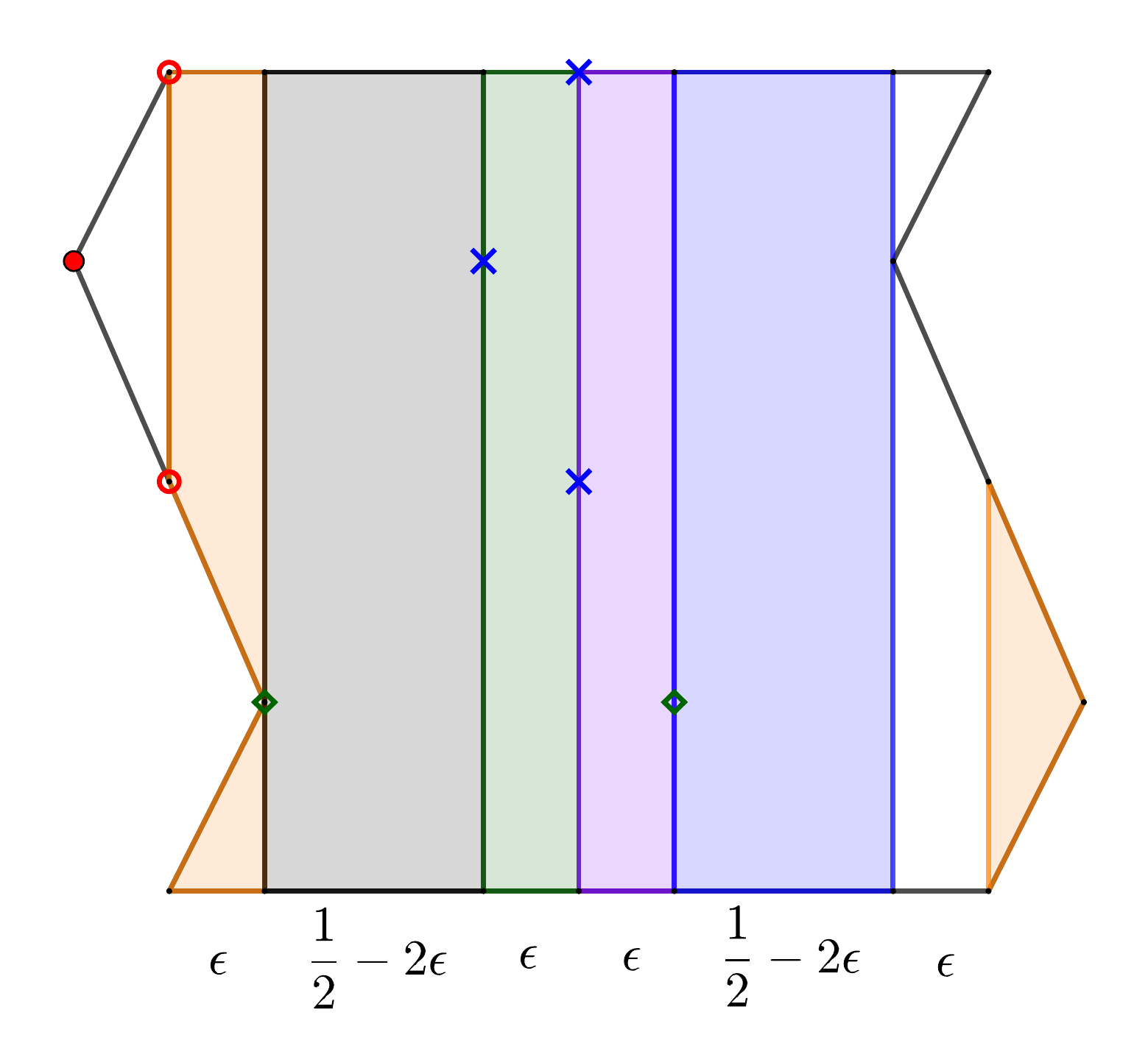}
\caption{The labels at the bottom of the figure indicate the widths of the six (subsets of) vertical  cylinders, where $\epsilon \in (0, \frac{1}{4})$.}
\label{F:H(0^n)Hardest2}
\end{figure}

For simplicity we assume without loss of generality that the horizontal cylinders on $(X, \omega)$ have circumference one. If the rightmost of the two points in Figure \ref{F:H(0^n)Hardest2} labelled with diamonds is marked then there is one vertical cylinder of width $\epsilon$ (passing through $\sigma_1$ and $\sigma_2$) and another of width $\frac{1}{2} - 2 \epsilon$ (passing to the right of the rightmost diamond) on $(X, \omega)$.  Since the vertical direction is covered by two subequivalence classes of vertical cylinders, each of which consists of one or two cylinders of equal width, we see that the sum of the widths of all subequivalence classes of vertical cylinders is at most 
\[ 2\cdot \epsilon + 2 \cdot \left( \frac{1}{2} - 2\epsilon \right) = 1 - 2\epsilon < 1\]
which contradicts the fact that the vertical subequivalence classes cover the surface.

Therefore, the rightmost diamond is not marked and so the leftmost diamond is marked. Thus there is a second vertical cylinder of width $\epsilon$, which passes to the left of the leftmost diamond. This must be the twin of the vertical cylinder passing through the $\sigma_i$.  The remainder of the surface is either covered by two vertical cylinders of width $\frac{1}{2} - \epsilon$ or one of width $1 - 2\epsilon$. 
Thus the surface is as illustrated in Figure \ref{F:H(0^n)Finale}  with 0, 1, or 2 of the points labelled with an ``x" marked. (Actually, by our Case 2 assumption, 1 or 2 of these points are marked.) 
We see that $\cM$ is a quadratic double. 

\begin{figure}[h]
\includegraphics[width=0.4\linewidth]{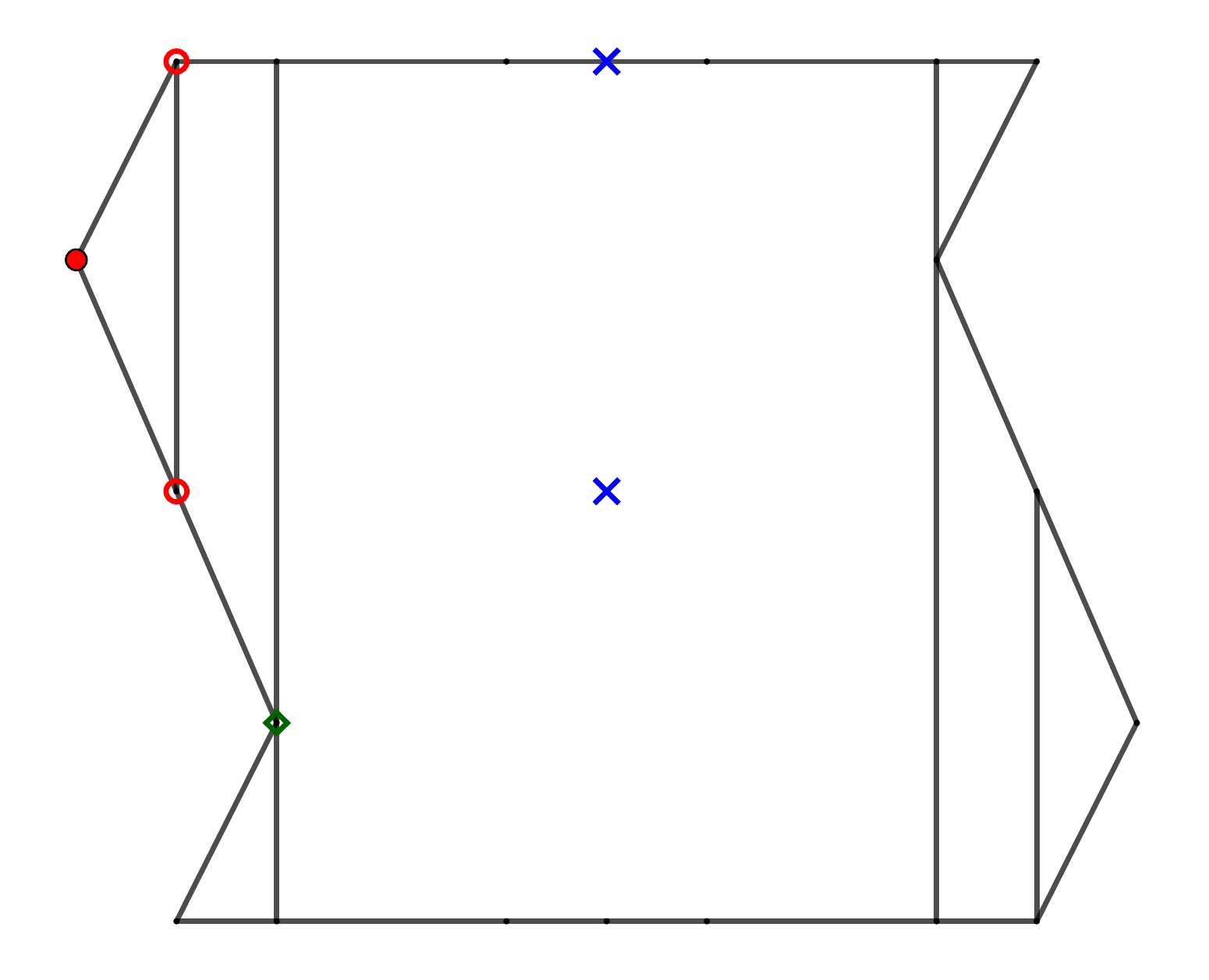}
\caption{ One or both of the points labelled with an ``x" are marked. }
\label{F:H(0^n)Finale}
\end{figure}

\bold{Case 2c: $\ColTwo(\bfC_2)$ consists of two saddle connections and $\bfC_2$ has one cylinder.}

Because $\ColTwo(\bfC_2)$ consists of two saddle connections, the one cylinder in $\bfC_2$ has either three or four points marked on its boundary. We get that $\cM$ is the orbit closure of the surface illustrated in Figure \ref{F:0nLast}, which is easily seen to not be geminal by analysing the vertical cylinders as follows. 
\begin{figure}[h]
\includegraphics[width=0.4\linewidth]{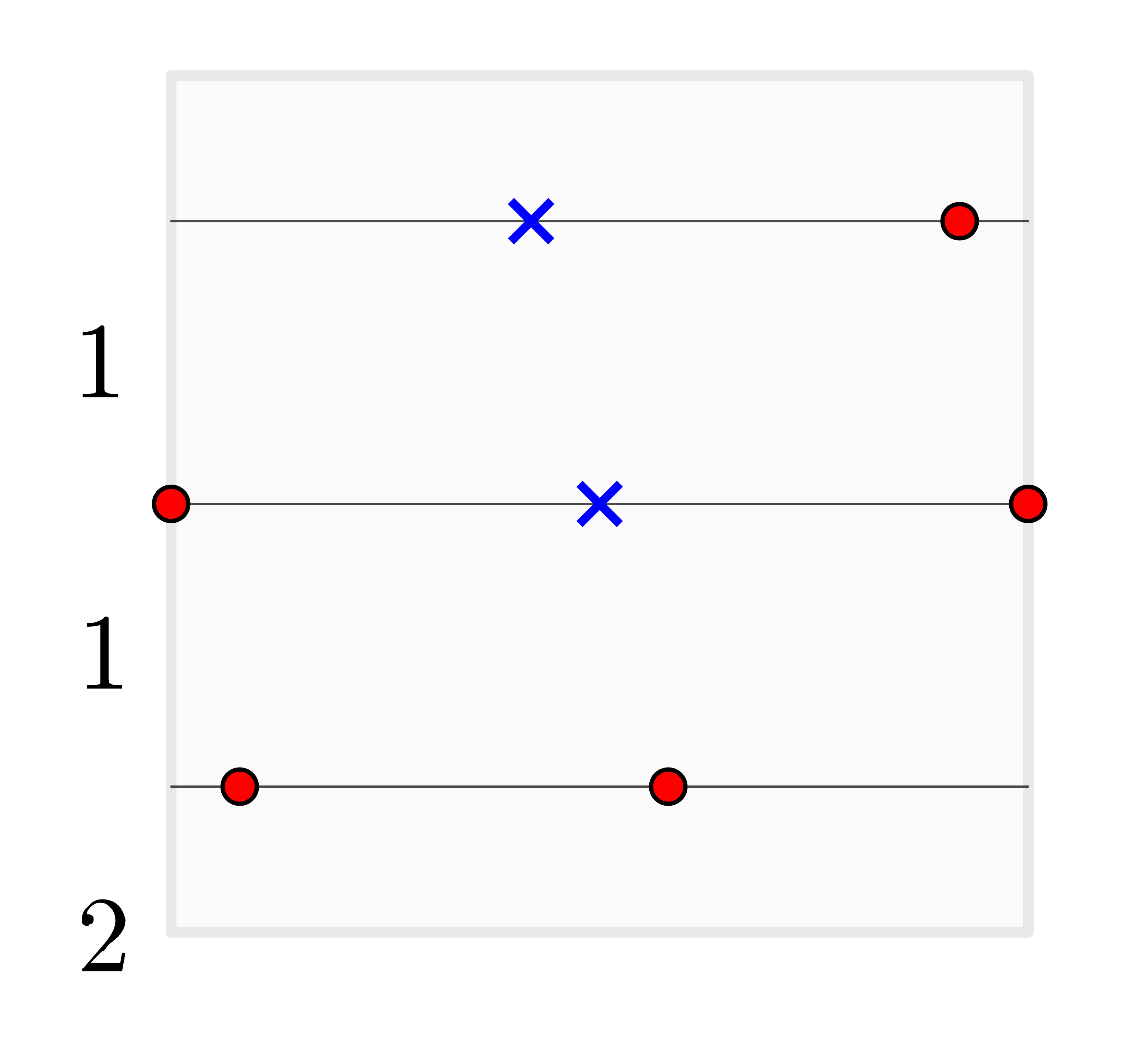}
\caption{The points with an ``x" may or may not be marked.}
\label{F:0nLast}
\end{figure}
Since there can be at most four vertical cylinders, we get that neither ``x" is marked. In that case there are four vertical cylinders, but two of them aren't free and aren't isometric to any other. 
\end{proof}

\subsection{Classification of rank one $h$-geminal subvarieties}

\begin{lem}\label{L:HGeminalPrelude}
Suppose that the orbit closure $\cM$ of $(X, \omega)$ is rank one and that, for any surface in $\cM$, parallel cylinders have homologous core curves. Then $(X, \omega)$ is a torus cover and the minimal degree map $\pi_{abs}$ to a torus is optimal.

Moreover, for any subequivalence class $\bfC$ of generic cylinders on $(X, \omega)$, $\Col_{\bfC}(\pi_{abs})$ is the optimal map for the connected surface $\Col_{\bfC}(X, \omega)$.
\end{lem}
\begin{rem}
We will apply Lemma \ref{L:HGeminalPrelude} only in the case that $\cM$ is geminal, but the fact that the proof applies more broadly allows us to connect to some previous work: since the hypotheses of the lemma are satisfied by Shimura-Teichm\"uller curves (see M\"oller \cite{M5} for a definition), Lemma \ref{L:HGeminalPrelude} (or more specifically the weaker statement that $\pi_{abs}$ is good) recovers a result of Aulicino-Norton \cite[Theorem 3.4]{AulicinoNorton}, which itself is a strengthening of a result of M\"oller \cite[Lemma 4.17]{M5}. Lemma \ref{L:HGeminalPrelude}  extends these results to a broader class of translation surfaces. Our approach is fundamentally different than the approaches of Aulicino-Norton  and  M\"oller. See Warning \ref{W:OptNeqOpt} before comparing to \cite{M5} or \cite{AulicinoNorton}.
\end{rem}
\begin{proof}
By \cite[Theorem 1.9]{Wcyl}, $\mathbf{k}(\cM) = \mathbb{Q}$. Since $\cM$ is rank one, it is a locus of torus covers by Lemma \ref{L:R1Arithmetic}. Let $\pi_{abs}$ denote the minimal degree map from $(X, \omega)$ to a torus. Suppose without loss of generality that $(X, \omega)$ is horizontally periodic.

Since parallel cylinders have homologous core curves, if we cut the core curves of each horizontal cylinder, we obtain a collection of sub-surfaces, each with exactly two boundary components. The surface $(X,\omega)$ is formed by gluing these subsurfaces together in a cyclic order. Equivalently, the surface is obtained by gluing the horizontal cylinders together in a cyclic order. The top of each cylinder is glued to the bottom of the cylinder above it along a collection of saddle connections; this gluing is reminiscent of an interval exchange transformation.

Let $H$ be the sum of the heights of the horizontal cylinders. Because the surface is formed by gluing the horizontal cylinders together in a cyclic order, every absolute period has an imaginary part that is an integer multiple of $H$. This shows that each individual horizontal cylinder is the preimage of its image under $\pi_{abs}$. Since this argument holds in every periodic direction it follows that $\pi_{abs}$ is good.

Now we will show that $\pi_{abs}$ is optimal, i.e. that any other good map factors through it. Suppose that $g: (X, \omega) \rightarrow (X', \omega')$ is a good map. It follows that parallel cylinders on $(X', \omega')$ have homologous core curves and so $\pi_{abs}' \circ g$ is a good map, where $\pi_{abs}'$ is the quotient by absolute periods for $(X', \omega')$. Since any translation cover from $(X, \omega)$ to a flat torus factors through $\pi_{abs}$, we have the commutative diagram of Figure \ref{F:Commutative}.
\begin{figure}[h!]
\begin{tikzcd}
  (X,\omega) \arrow[r, "\pi_{abs}"] \arrow[d, "g"]
    & (E, \eta) \arrow[d, "g'" ] \\
 (X', \omega') \arrow[r, "\pi_{abs}'"]
&(E',\eta') \end{tikzcd}
\caption{The proof of Lemma \ref{L:HGeminalPrelude}}
\label{F:Commutative}
\end{figure}
where $g': (E ,\eta) \rightarrow (E', \eta')$ is a good map since $\pi_{abs}' \circ g$ is. (Note that it is immediate from the definition of ``good map" that compositions of good maps are good and that factors of good maps are good as well.)  The map $g'$ has degree one by Lemma \ref{L:GoodInH(0^n)}. Therefore, $g'$ is the identity and so $g$ is a factor of $\pi_{abs}$ as desired. 

Suppose now that $\bfC$ is a subequivalence class of generic cylinders on $(X, \omega)$. By Lemma \ref{L:DisconnectingUnderOptimalMap}, $\Col_{\bfC}(X, \omega)$ is connected since $\pi_{abs}$ is good and $\Col_{\pi_{abs}(\bfC)}\left( \pi_{abs}\left(X, \omega \right) \right)$ is connected. 

\begin{sublem}
Parallel cylinders on $\Col_{\bfC}(X, \omega)$ have homologous core curves. 
\end{sublem}
\begin{proof}
Suppose in order to deduce a contradiction that this is not the case. Let $D_1$ and $D_2$ be two cylinders that are parallel on $\Col_{\bfC}(X, \omega)$ but whose core curves are not homologous. After perhaps applying the standard dilation to $\bfC$ (in order to assume that the surface is sufficiently close to the boundary), we may assume that there are cylinders $D_1'$ and $D_2'$ on $(X, \omega)$ with disjoint core curves and so that $\Col_{\bfC}(D_i') = D_i$. Since $D_1$ and $D_2$ are not homologous the same holds for $D_1'$ and $D_2'$ (as in the proof of Lemma \ref{L:HGeminalBoundary}), implying that $D_1'$ and $D_2'$ are not parallel to each other. 

  Since $(X, \omega)$ is a torus cover, it is covered by the cylinders parallel to $D_1'$, and also by the cylinders parallel to $D_2'$. Since parallel cylinders are homologous, and since the core curves of $D_1'$ and $D_2'$ do not intersect, no cylinder parallel to $D_1'$ intersects a cylinder parallel to $D_2'$, which is a contradiction.  
\end{proof}


Because of our previous arguments, the sublemma implies that the quotient by the absolute period lattice is the optimal map for $\Col_{\bfC}(X, \omega)$. 

By Lemma \ref{L:DegeneratingOptimalMap}, $\Col_{\bfC}(\pi_{abs})$ is a good map. Since $\Col_{\bfC}(\pi_{abs})$ is a map to a torus, it has the quotient by the absolute period lattice of $\Col_{\bfC}(X, \omega)$ as a factor. However, the quotient by the absolute period lattice is the optimal map and hence $\Col_{\bfC}(\pi_{abs})$ must be the optimal map as desired.
\end{proof}

\begin{lem}\label{L:HGeminalOptimal}
Suppose that the orbit closure of $(X, \omega)$ is a rank one $h$-geminal  subvariety $\cM$.
\begin{enumerate}
\item\label{I:R1HGeminal1} The minimal degree map $\pi_{abs}$ from $(X, \omega)$ to a torus is the  optimal map.
\item\label{I:R1HGeminal2} The orbit closure $\cM_{abs}$ of $\pi_{abs}(X, \omega)$ is one of the following: a stratum of flat tori, an Abelian double, a quadratic double, or a $T \times T$ locus.
\item\label{I:R1HGeminal3} The optimal map is the identity when $\cM_{abs}$ is a stratum or Abelian double. Moreover, when the optimal map has degree greater than one, each subequivalence class contains exactly two cylinders.
\item\label{I:R1HGeminal3.5} The optimal map is the identity when $\cM$ is an Abelian or quadratic double. 
\item\label{I:R1HGeminal4} For any subequivalence class $\bfC$ of generic cylinders on $(X, \omega)$, $\Col_{\bfC}(\pi_{abs})$ is the optimal map for $\Col_{\bfC}(X, \omega)$. 
\end{enumerate} 
\end{lem}
\begin{rem}\label{R:HGeminalOptimal}
This shows that when $\cM$ is $h$-geminal, Proposition \ref{P:geminalrk1} holds. In particular, Proposition \ref{P:geminalrk1} (\ref{I:geminalrk1:WhatIsM}) is implied by Lemma \ref{L:HGeminalOptimal} (\ref{I:R1HGeminal2}) and (\ref{I:R1HGeminal3}); Proposition \ref{P:geminalrk1} (\ref{I:geminalrk1:PiOpt}) is implied by Lemma \ref{L:HGeminalOptimal} (\ref{I:R1HGeminal1}) and (\ref{I:R1HGeminal3.5}); Proposition \ref{P:geminalrk1} (\ref{I:geminalrk1:TwoCylinders}) is implied by Lemma \ref{L:HGeminalOptimal} (\ref{I:R1HGeminal3}); and Proposition \ref{P:geminalrk1} (\ref{I:geminalrk1:PiOptDegeneratesToPiOpt}) is implied by Lemma \ref{L:HGeminalOptimal} (\ref{I:R1HGeminal4}). 
\end{rem}

\begin{proof}
Let $(X, \omega)$ be a horizontally periodic surface in $\cM$.

\begin{sublem}\label{SL:homologous} 
The core curves of any two horizontal cylinders $C$ and $C'$ are homologous. 
\end{sublem}

\begin{proof}
Recall the notion of standard shear, defined before Theorem \ref{T:CDT}. 

Let $\sigma$ and $\sigma'$ be the standards shears in the subequivalence classes of $C$ and $C'$ respectively. Since $\cM$ is $h$-geminal, $p(\sigma)$ is proportional to the Poincare dual (in absolute cohomology) of the core curve of $C$, and the analogous statement holds for $p(\sigma')$. (Recall from Section \ref{SS:RankRelAndCovers} that $p$ is the projection from relative to absolute cohomology.)  

We now recall that, since $\cM$ has rank 1,  $p(\sigma)$ and $p(\sigma')$ are proportional. Indeed, this can be derived from Lemma \ref{L:R1Deformations}, and is a special case of \cite[Theorem 1.5]{MirWri}, which is recalled in \cite[Theorem 3.12]{ApisaWrightDiamonds}.

Hence we get that the core curves of $C$ and $C'$ are proportional in absolute homology. Since two non-separating simple closed curves define proportional homology classes if and only if they are homologous, this gives the result. 
\end{proof}

By Sublemma \ref{SL:homologous}, the hypotheses of Lemma \ref{L:HGeminalPrelude} hold and so \eqref{I:R1HGeminal1} and \eqref{I:R1HGeminal4} are immediate. Since $\pi_{abs}$ is good, $\cM_{abs}$ is geminal in a stratum of tori. By Proposition \ref{P:0n}, $\cM_{abs}$ must be one of the following: a stratum, an Abelian double, a quadratic double, or a $T \times T$ locus (this shows (\ref{I:R1HGeminal2})). It remains to show (\ref{I:R1HGeminal3}) and (\ref{I:R1HGeminal3.5}).

\begin{sublem}\label{SL:HGeminalRulingOut}
Suppose that $\cM_{abs}$ contains a surface where at least one subequivalence class of cylinders has only one cylinder. Then the optimal map is the identity.
\end{sublem}
Note that the assumption is automatically true if $\cM_{abs}$ is a stratum or Abelian double. Therefore, Sublemma \ref{SL:HGeminalRulingOut} implies (\ref{I:R1HGeminal3}).

The proof will use the following general remark, which is illustrated in Figure \ref{F:ColOntoC} and does not use any of our ambient assumptions. 

\begin{figure}[h]
\includegraphics[width=\linewidth]{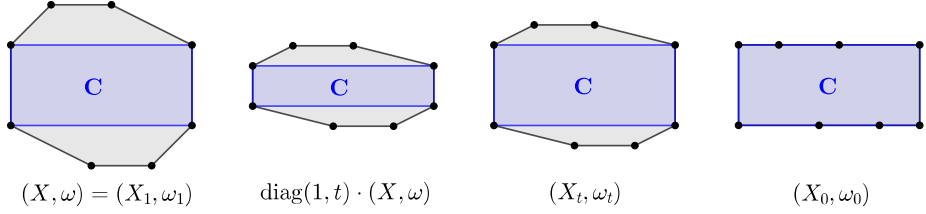}
\caption{  The procedure described in Remark \ref{R:CollapseOntoC}.  }
\label{F:ColOntoC}
\end{figure}

\begin{rem}\label{R:CollapseOntoC}
Let $\bfC$ be a collection of cylinders on a surface $(X,\omega)$ in an arbitrary orbit closure $\cM$. Assume that the standard deformation of $\bfC$ stays in $\cM$, which is to say $\sigma_\bfC \in T_{(X,\omega)} \cM$. Assume that $\bfC$ is horizontal and that the complement of $\bfC$ does not contain any horizontal saddle connections (this can always be arranged using the $GL(2,\bR)$ action). We now describe a general construction to ``collapse $(X,\omega)$ onto $\bfC$", resulting in a new surface $(X_0, \omega_0)\in \cM$ which is horizontally periodic and whose horizontal cylinders exactly correspond to the cylinders in $\bfC$. To describe this construction, first define the surface $(X_t, \omega_t), t\in (0,1]$  via the following two step process. First scale the vertical direction by a factor of $t$ with the $GL(2,\bR)$ action, using the diagonal matrix $\operatorname{diag}(1, t)\in GL(2,\bR)$. Second, use a standard deformation of $\bfC$ to scale the vertical direction in $\bfC$ only by $1/t$ to obtain $(X_t, \omega_t)$. So $(X_1, \omega_1)=(X,\omega)$, the path $(X_t, \omega_t)$ remains in $\cM$, and the cylinders $\bfC$ have constant direction, circumference, and modulus along the path. Since the complement of $\bfC$ does not have vertical saddle connections, the surfaces $(X_t, \omega_t)$ do not degenerate as $t\to 0$, and we can define $(X_0, \omega_0)\in \cM$ to be the limit. 
This surface $(X_0, \omega_0)$ is the desired result of ``collapsing $(X,\omega)$ onto $\bfC$". 

If $(X,\omega)$ happens to be horizontally periodic, then $(X_0, \omega_0)$ can be equivalently obtained by vertically collapsing the horizontal cylinders not in $\bfC$, but the construction applies regardless of whether this is the case. 
\end{rem}

\begin{proof}
Given Remark \ref{R:CollapseOntoC}, the assumption of the lemma implies that there is a surface in $\cM_{abs}$ that is covered by a single horizontal cylinder. Because $\pi_{abs}$ is a good map, the corresponding surface, call it $(Y, \eta)$, in $\cM$ is also covered by a single horizontal cylinder. After shearing $(Y,\eta)$, we can assume that there is a vertical cylinder that crosses the horizontal exactly once. Lemma \ref{L:IsGood} now gives the result.
\end{proof}

It remains to show (\ref{I:R1HGeminal3.5}), i.e. that if $\cM$ is an Abelian or quadratic double, then the optimal map is the identity. Since $\pi_{abs}$ is the optimal map, it suffices to show that $\cM$ is a locus of tori. Suppose to a contradiction that this is not the case. 

Suppose first that $\cM$ is an Abelian double; since it has rank 1, $\cM$ is necessarily an Abelian double of a stratum of flat tori. Since $\pi_{abs}$ is the minimal degree map to a torus, it follows that the cover whose domain is $(X, \omega)$ and that arises from $\cM$ being an Abelian double must be $\pi_{abs}$. In particular, $\cM_{abs}$ must be a stratum. However, by Sublemma \ref{SL:HGeminalRulingOut} this implies that the optimal map is the identity, which is a contradiction.


Suppose now that $\cM$ is a quadratic double. Since $\cM$ does not contain flat tori, $\cM$ is a quadratic double of $\cQ(2, -1^2, 0^k)$ or $\cQ(2,2, 0^k)$ for some nonnegative integer $k$. In these cases, it is easy to see that $\cM$ is not $h$-geminal as in Figure \ref{F:D1D2Double}, a contradiction.
\end{proof}

\subsection{Classification of rank one non $h$-geminal subvarieties.}

\begin{lem}\label{L:nhGeminal}
Let $(X, \omega)$ be contained in a rank one geminal subvariety $\cM$ that is not $h$-geminal. Then the following hold: 
\begin{enumerate}
\item\label{I:R1OptimalIdentity} The $\cM$-optimal map for $(X,\omega)$ is the identity. 
\item\label{I:R1AQD} $\cM$ is an Abelian or quadratic double.
\item\label{I:R1NHFinal} If $\bfC$ is a subequivalence class of generic cylinders on $(X, \omega)$, then the identity is the $\cM_{\bfC}$-optimal map for $\Col_{\bfC}(X, \omega)$.
\end{enumerate}
\end{lem}

Before proceeding with the proof of Lemma \ref{L:nhGeminal} we will show how it implies Proposition \ref{P:geminalrk1}. 

\begin{proof}[Proof of Proposition \ref{P:geminalrk1} assuming Lemma \ref{L:nhGeminal}:]
Let $\cM$ be a rank one geminal invariant subvariety. If $\cM$ is $h$-geminal, then Proposition \ref{P:geminalrk1} holds by Lemma \ref{L:HGeminalOptimal} (see Remark \ref{R:HGeminalOptimal}). Suppose therefore that $\cM$ is not $h$-geminal. Proposition \ref{P:geminalrk1} (\ref{I:geminalrk1:WhatIsM}) holds by Lemma \ref{L:nhGeminal} (\ref{I:R1AQD}); Proposition \ref{P:geminalrk1} (\ref{I:geminalrk1:PiOpt}) holds by Lemma \ref{L:nhGeminal} (\ref{I:R1OptimalIdentity}), which also shows that the assumptions for Proposition \ref{P:geminalrk1} (\ref{I:geminalrk1:TwoCylinders}) never hold when $\cM$ is not $h$-geminal; and Proposition \ref{P:geminalrk1} (\ref{I:geminalrk1:PiOptDegeneratesToPiOpt}) holds by Lemma \ref{L:nhGeminal} (\ref{I:R1NHFinal}). 
\end{proof}

\begin{proof}
By assumption there is a surface $(X, \omega)$ in $\cM$ that has two non-homologous  twin cylinders; call them $C_1$ and $C_2$ and assume they are horizontal. Using Remark \ref{R:CollapseOntoC} we replace $(X,\omega)$ with a surface that is entirely covered by these cylinders. Our proof will analyse certain cylinders transverse to $C_1, C_2$, eventually achieving a sufficiently precise understanding of the surface to verify the lemma.



If every horizontal saddle connection on the boundary of $C_1$ also lay on the boundary of $C_2$, then $C_1$ and $C_2$ would be homologous, contrary to our hypothesis. Therefore, there is some saddle connection $s$ that appears on the top and bottom boundary of $C_1$. Apply the standard shear to $\{C_1, C_2\}$ so that there is a vertical cylinder $V_1$ contained in $C_1$ that contains $s$ and passes through $C_1$ exactly once. If $V_1$ were free, deforming it would contradict the fact that $C_1$ and $C_2$ are twins. So $V_1$ has a twin, which we will call $V_2$. 

Since $V_1$ and $V_2$ are isometric, and $C_1$ and $C_2$ are isometric, and since $V_1$ is contained in $C_1$, we see that $V_2$ is contained in $C_2$. By Lemma \ref{L:IsGood} since $V_1$ and $C_1$ have core curves that intersect exactly once it follows that the identity is the optimal map. Since $V_1$ and $C_1$ persist and intersect once in an open neighborhood of $(X, \omega)$ it follows that the identity is the optimal map for every surface in this open neighborhood. It follows that the identity is an $\cM$-optimal map for any surface in $\cM$, which establishes Lemma \ref{L:nhGeminal} (\ref{I:R1OptimalIdentity}).

For the remainder of this proof we will let ``big cylinders" denote vertical cylinders that intersect both $C_1$ and $C_2$ and ``little cylinder" denote any vertical cylinder that is not big. 

\begin{sublem}\label{SL:BigIntersection}
A big cylinder intersects $C_1$ (and hence also $C_2$) exactly once. Similarly, a little cylinder contained in $C_i$ intersects $C_i$ exactly once. 
\end{sublem}
\begin{proof}
Suppose to a contradiction that there is some vertical cylinder $W$ that intersects $C_1$ more than once. Necessarily $W \notin \{V_1, V_2\}$. 
\begin{figure}[h!]
\includegraphics[width=0.5\linewidth]{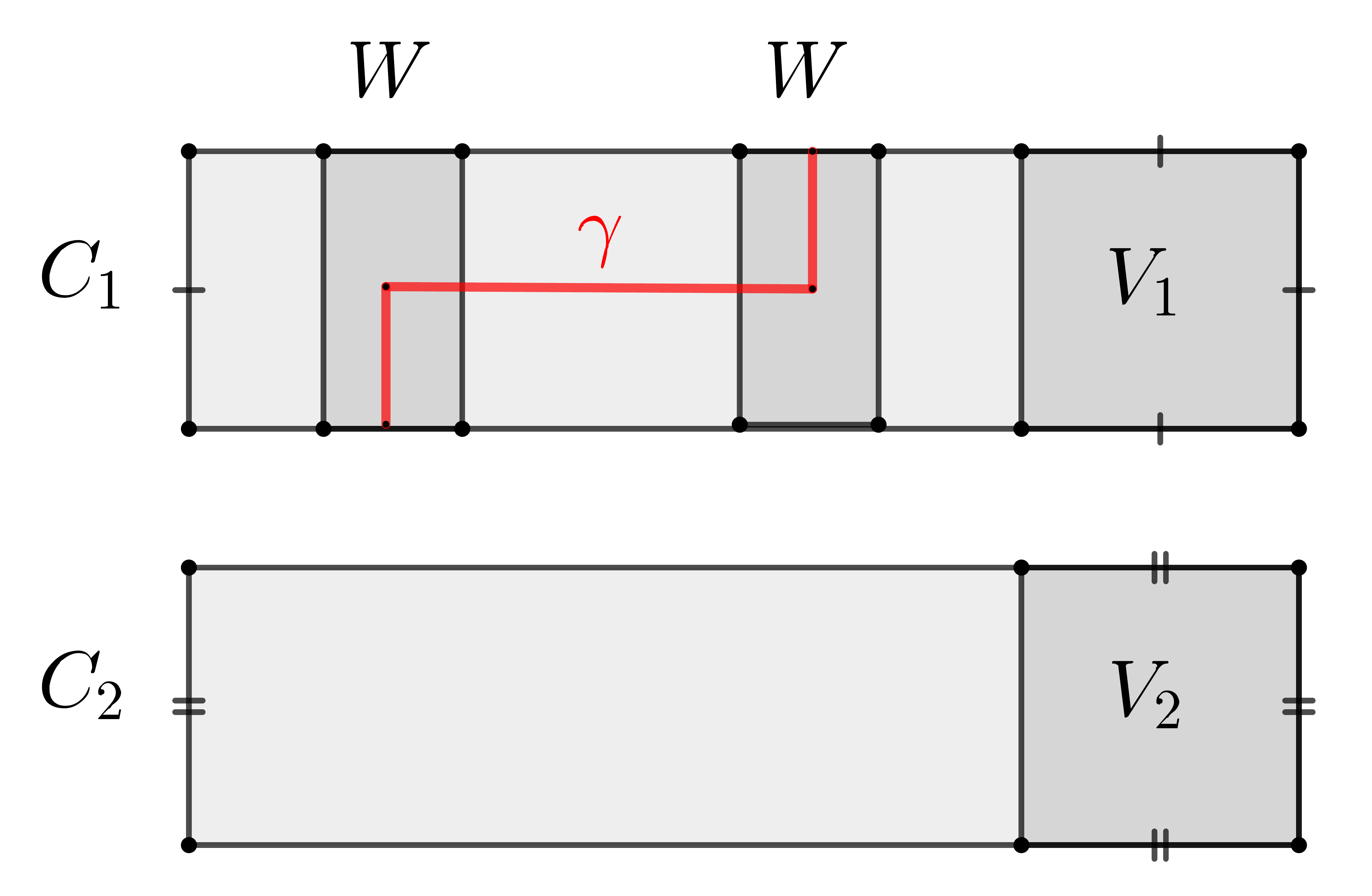}
\caption{  An illustration of (possibly only part of) $W$ and $\gamma$.  }
\label{F:Wgamma}
\end{figure}
Pick a segment of the core curve of $C_1$ that travels from the core curve of $W$ to another intersection with the core curve of $W$ without intersecting $V_1$. Complete this segment to a closed loop $\gamma$ using a segment of the core curve of $W$; see Figure \ref{F:Wgamma}. 

There is a deformation of the surface that horizontally stretches $V_1, V_2$ and then horizontally shrinks the rest of the surface to keep the area constant. By Lemma \ref{L:R1Deformations}, since this deformation does not change the holonomy of the core curves of the $C_i$ or $V_i$, 
this must be a rel deformation. But on the other hand this deformation reduces the real part of the period of the absolute homology class $[\gamma]$, giving a contradiction. 
\end{proof}

\begin{sublem}\label{SL:CyclicOrder1}
If $W_1$ and $W_2$ are big cylinders then the cyclic order of $W_1, W_2, V_1$ in $C_1$ is equal to the cyclic order of $W_1, W_2, V_2$ in $C_2.$
\end{sublem}
For example, if in $C_1$ one can pass from $W_1$ to $W_2$ by moving to the right without passing through $V_1$, then in $C_2$ the analogous statement is true. 
\begin{proof}
Suppose not. To be concrete, assume there is a subarc $\gamma_1$ of the core curve of $C_1$  that passes from $W_1$  to $W_2$  by traveling left to right without passing through $V_1$; and suppose in $C_2$ there is a subarc $\gamma_2$ of the core curve of $C_2$ that passes from $W_2$ to $W_1$ by traveling left to right without passing through $V_2$. These subarcs can be connected to form a closed loop  $\gamma$  by attaching part of the core curves of $W_1$ and $W_2$, as in Figure \ref{F:W1W2order}. 
\begin{figure}[h]
\includegraphics[width=0.5\linewidth]{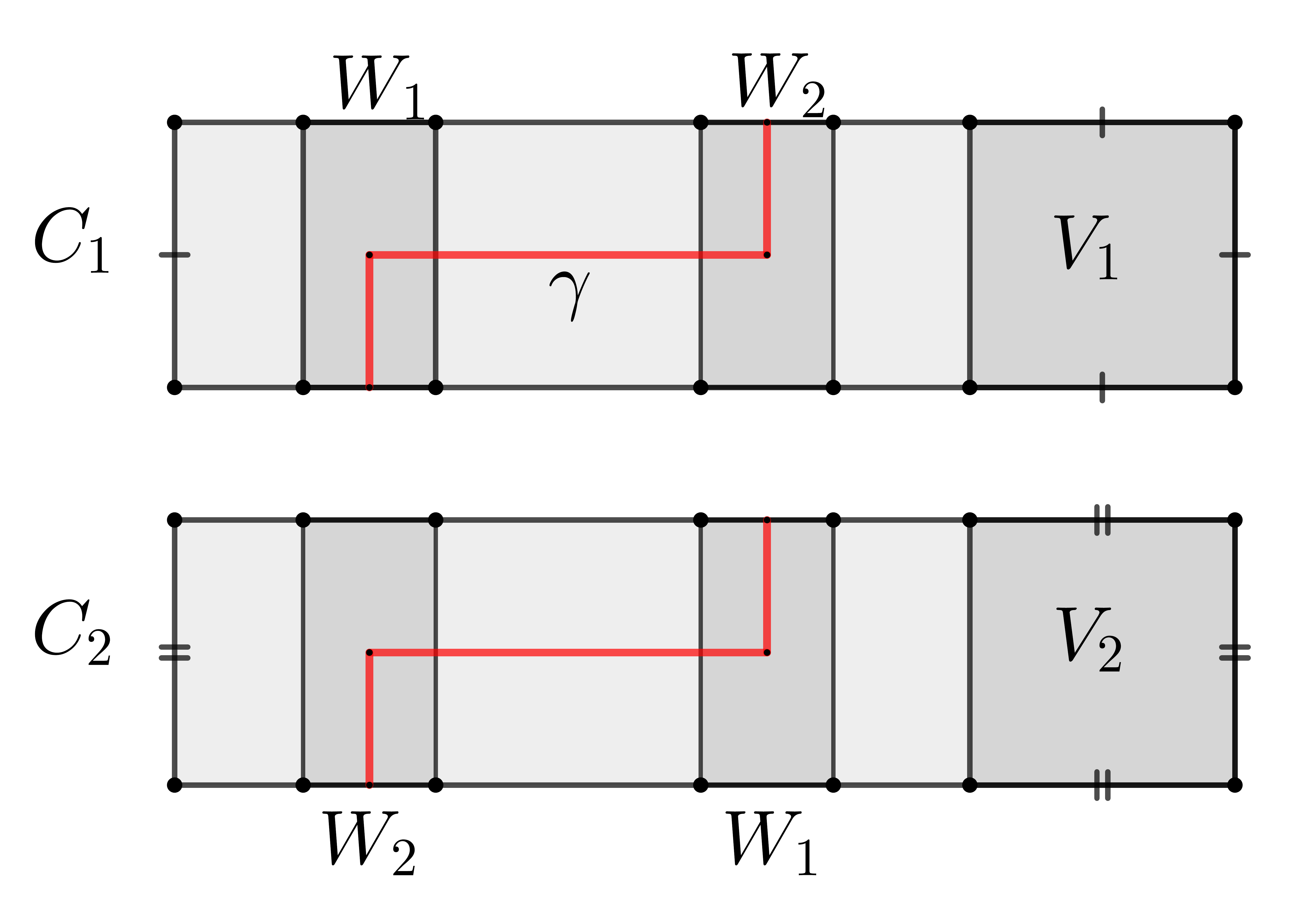}
\caption{The proof of Sublemma \ref{SL:CyclicOrder1}.}
\label{F:W1W2order}
\end{figure}
The same  deformation considered in the last sublemma gives a contradiction, since it must be rel but changes the period of $\gamma$.
\end{proof}

\begin{sublem}\label{SL:SameDist}
The distance between two big cylinders $W_1$ and $W_2$, measured along a segment of the core curve of $C_1$ not intersecting $V_1$, is the same as the corresponding distance in $C_2$. 
\end{sublem}
\begin{proof}
The proof is almost identical to the previous two proofs, except one uses the curve $\gamma$ shown in Figure \ref{F:W1W2dist}. 
\begin{figure}[h]
\includegraphics[width=0.5\linewidth]{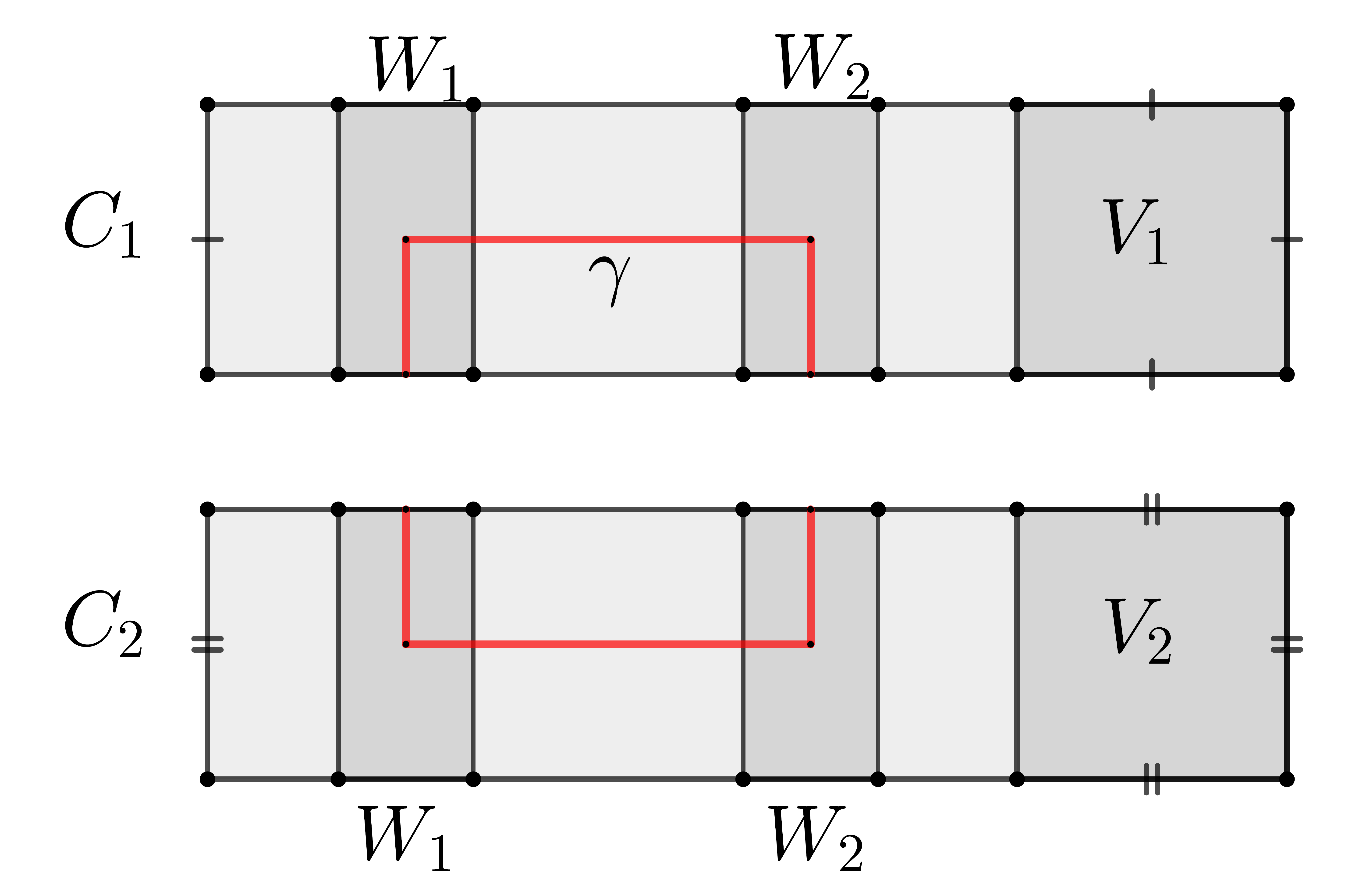}
\caption{The proof of Sublemma \ref{SL:SameDist}.}
\label{F:W1W2dist}
\end{figure}
\end{proof}

  Equipped with these sublemmas, we will complete the proof by showing that $\cM$ is an Abelian or quadratic double and that, if $\bfC$ is a subequivalence class of generic cylinders on $(X, \omega)$, then the identity is the $\cM_{\bfC}$-optimal map for $\Col_{\bfC}(X, \omega)$. We will show this via induction on the rel of $\cM$, which we denote by $k$. Note that $k>0$, since if $k=0$ the surface would consist entirely of $V_1$ and $V_2$, and would hence be disconnected. Thus $k=1$ will be the base case.


So suppose $k = 1$.  There are two vertical equivalence classes of cylinders - $\bfD_1$ and $\bfD_2$. We can assume that $\bfD_1 = \{V_1, V_2\}$ and that $\bfD_2$ contains either one or two big cylinders. 

By Sublemmas \ref{SL:BigIntersection} and \ref{SL:CyclicOrder1}, $\cM$ is a quadratic double of $\cQ(2,-1^2)$, as depicted in Figure \ref{F:D1D2Double}. Notice that when there are no marked points, $\cM$ is simultaneously an Abelian double.

\begin{figure}[h]
\includegraphics[width=0.5\linewidth]{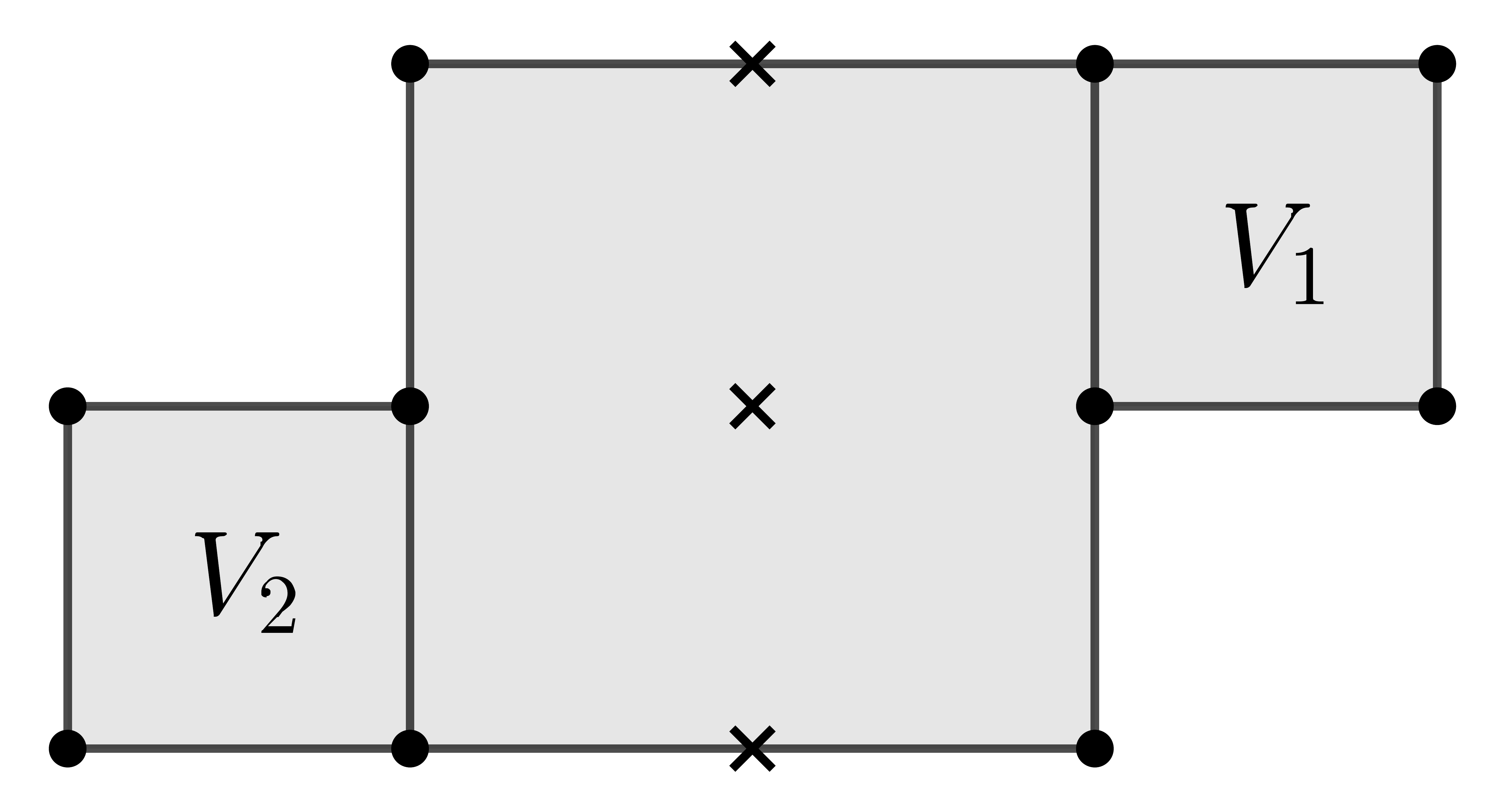}
\caption{A double of $\cQ(2,-1^2)$. }
\label{F:D1D2Double}
\end{figure}

If $\bfC$ is a subequivalence class of cylinders on $(X, \omega)$ then, because $\cM$ is a quadratic double of $\cQ(2,-1^2)$,  either $\Col_{\bfC}(X, \omega)$ is disconnected or a flat torus with marked points. In either case, the optimal map for $\Col_{\bfC}(X, \omega)$ is also the identity by Lemmas \ref{L:DisconnectedDegenerationMaps} and \ref{L:GoodInH(0^n)}. This establishes the base case.



We require some additional analysis for the inductive step. By Sublemma \ref{SL:BigIntersection}, each little cylinder intersects $C_1$ or $C_2$ exactly once. Therefore, any pair of adjacent little cylinders $A_1, B_1$ in $C_1$ are separated by a marked point. Lemma \ref{L:GeminalTwinAdjacency} gives that the corresponding twins $A_2$ and $B_2$ in $C_2$ also are separated by a marked point. We can consider the deformation that increases the width of the $A_i$ at unit speed and decreases the width of the $B_i$ at unit speed. If $A_1$ is to the left of $B_1$ and $A_2$ is to the left of $B_2$, then this moves the pair of marked points equal amounts in the same direction. If $A_1$ is to the left of $B_1$ and $A_2$ is to the right of $B_2$, then this moves the pair of marked points equal amounts in opposite directions. It is now helpful to forget about cylinders for a moment, and summarize by saying that the pair of marked points can be moved equally in the same direction or equally in opposite directions. Such deformations can be continued arbitrarily, as long as the pair of marked points don't hit any other marked points or singularities. 

Since every marked point not in a big cylinder separates a pair of adjacent little cylinders, we may use such deformations to move all marked points into big cylinders. This means in particular that, by performing such deformations, we may assume that no little cylinders are adjacent to each other in $C_1$ or $C_2$, so each little cylinder has a big cylinder on either side. We continue our analysis with this assumption.  

\begin{sublem}\label{SL:CyclicOrder2}
If $A_i\subset C_i$ is a little cylinder with big cylinder $L$ to the left and $R$ to the right, then the twin $A_{i+1}$ of $A_i$ also has $L$ to the left and $R$ to the right. 
\end{sublem}
\begin{proof}
If $A_{i+1}\subset C_{i+1}$ was not in between $L$ and $R$, then by deforming $\{A_1, A_2\}$ we could contradict Sublemma \ref{SL:SameDist}. Similarly if any other subequivalence class was between $L$ and $R$ in exactly one of $C_1$ or $C_2$, we could contradict Sublemma \ref{SL:SameDist}. 
\end{proof}

We will now focus on showing that $\cM$ is an Abelian or quadratic double. That is, we will establish Lemma \ref{L:nhGeminal} (\ref{I:R1AQD}) now and then return to proving (\ref{I:R1NHFinal}). 

  \bold{Case 1: There are at least 2 subequivalence classes of little cylinders.} If $\bfD_1$ and $\bfD_2$ are two distinct subequivalence classes of little cylinders, then $\left( (X, \omega), \cM, \bfD_1, \bfD_2 \right)$ is a generic diamond where $\Col_{\bfD_1, \bfD_2}(X, \omega)$ is connected and, by the induction hypothesis, $\cM_{\bfD_1}$ and $\cM_{\bfD_2}$ are Abelian or quadratic doubles. By Proposition \ref{P:DiamondsAreCool}, $\cM$ is also an Abelian or quadratic double. 

  \bold{Case 2: There is only 1 subequivalence class of little cylinders.} 
This subequivalence class is necessarily the one containing $V_1$ and $V_2$, which we will denote $\bfD_1$. 

  \bold{Case 2a: There is a subequivalence class $\bfD_2$ of big cylinders that do not share any boundary saddle connections with the cylinders in $\bfD_1$.}   
Then $\left( (X, \omega), \cM, \bfD_1, \bfD_2 \right)$ is a generic diamond where, by the induction hypothesis, $\cM_{\bfD_1}$ and $\cM_{\bfD_2}$ are Abelian or quadratic doubles. Since $k>1$, and there are $k+1$ subequivalence classes of vertical cylinders, and since we have assumed there is only one little subequivalence class, there is at least one big subequivalence class besides $\bfD_2$, and so we see that  $\Col_{\bfD_1, \bfD_2}(X, \omega)$ is connected. By Proposition \ref{P:DiamondsAreCool}, $\cM$ is also an Abelian or quadratic double. 

  \bold{Case 2b: There does not exist a subequivalence as in Case 2a.} Label the big cylinders to the right and left of $V_1$ by $R$ and $L$ respectively. We may suppose now, by Lemma \ref{L:GeminalTwinAdjacency} and Sublemma \ref{SL:CyclicOrder2}, that $R$ and $L$ are the only big cylinders, and $\bfD_1=\{V_1, V_2\}$ are the only small cylinders.

Therefore, there are four vertical cylinders - two little and two big. Since $k>1$, the two big cylinders are both free. By Sublemmas \ref{SL:BigIntersection}, \ref{SL:CyclicOrder1}, and \ref{SL:CyclicOrder2}  it follows that  $\For(X, \omega)$ is a surface in a double of $\cQ(2, -1^2)$, and the only marked points on $(X,\omega)$ are either one free point $p$ or a pair of points $\{p, p'\}$ whose only constraint is that they differ by the translation involution. See Figure \ref{F:NonHGeminalEnd}.
\begin{figure}[h]
\includegraphics[width=0.5\linewidth]{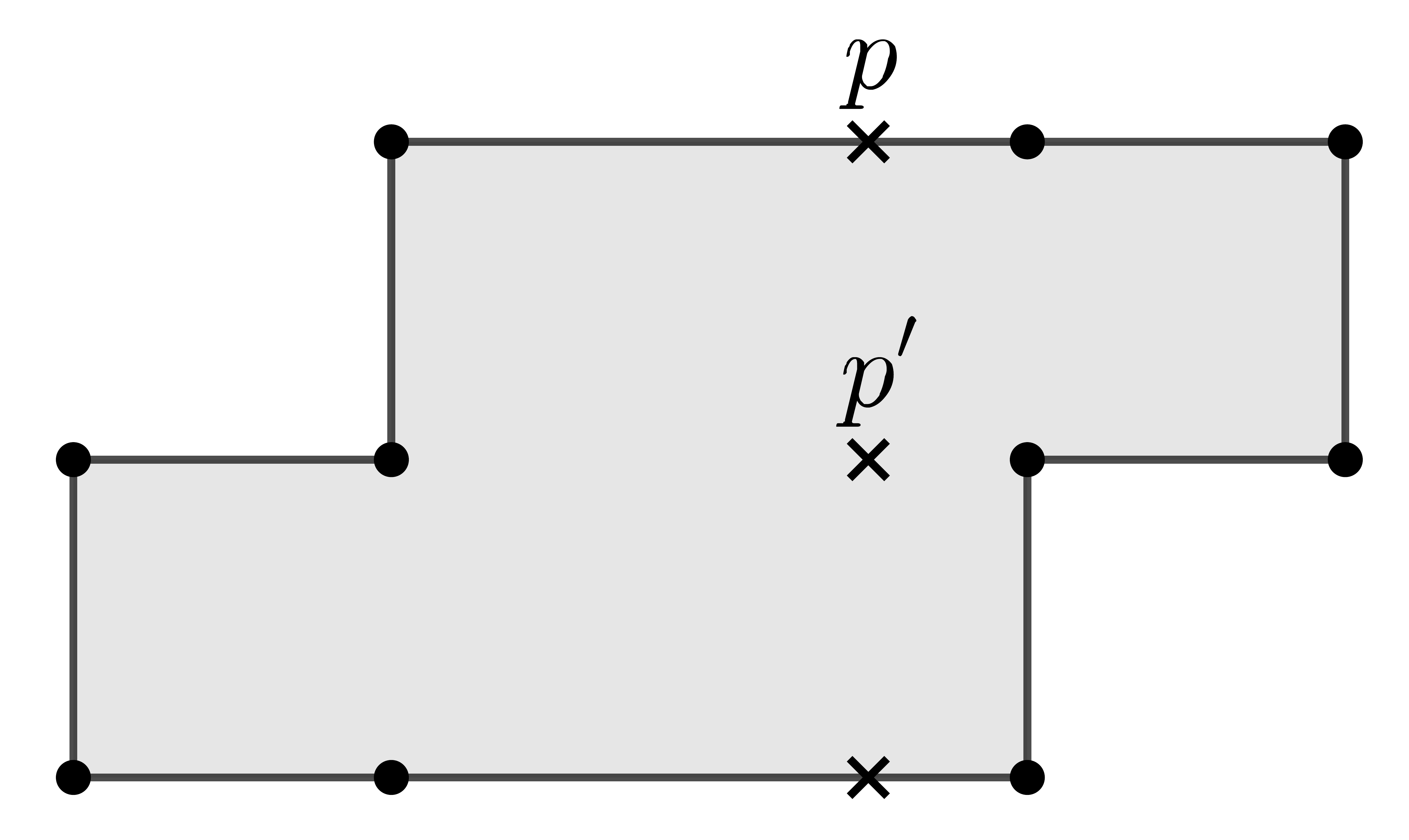}
\caption{}
\label{F:NonHGeminalEnd}
\end{figure}
In the latter case, $\cM$ is an Abelian double and, in the former, it is clear that $\cM$ is not geminal as can be seen by moving the marked point into one little cylinder but not the other. 

  The above case analysis shows that  $\cM$ is an Abelian or quadratic double, concluding the proof of Lemma \ref{L:nhGeminal} (\ref{I:R1AQD}). We now complete the proof of Lemma \ref{L:nhGeminal} by proving Lemma \ref{L:nhGeminal} (\ref{I:R1NHFinal}).  We will now let $(X, \omega)$ denote an arbitrary surface in $\cM$. Let $\bfC$ be a subequivalence class of cylinders on $(X, \omega)$. If $\Col_{\bfC}(X, \omega)$ is disconnected, then the $\cM_{\bfC}$-optimal map on $\Col_{\bfC}(X, \omega)$ is also the identity by Lemma \ref{L:DisconnectedDegenerationMaps}. Suppose therefore that $\Col_{\bfC}(X, \omega)$ is connected. Then $\cM_{\bfC}$ remains an Abelian or quadratic double. If $\cM_{\bfC}$ is not $h$-geminal then the identity is the $\cM_{\bfC}$-optimal map by the induction hypothesis; if it is $h$-geminal then the identity is the $\cM_{\bfC}$-optimal map by Lemma \ref{L:HGeminalOptimal} (\ref{I:R1HGeminal3.5}). 
\end{proof}

\section{Geminal subvarieties that degenerate to $T \times T$ loci}\label{S:TT}

The main result of this section is the following. Recall that $T\times T$ loci were defined in Definition \ref{D:TxT} and by definition exist only in genus one, and generic diamonds were defined in Definition \ref{D:GenericDiamond}. 

\begin{prop}\label{P:TxTDiamond}
Suppose that $\cM$ is geminal of rank at least two. Suppose that $\left( (X, \omega), \cM, \bfC_1, \bfC_2 \right)$ is a generic diamond. If $\MOne$ is a $T \times T$ locus, then $\MTwo$ cannot be any of the following: a component of a stratum of Abelian differentials, an Abelian double, a quadratic double, or a locus in a stratum of tori. 
\end{prop}

We will defer a proof of this result to the end of the section. Part of the proof will be based on the following simpler lemma.

\begin{lem}\label{L:Z2xZ2covers}
Let $\cM$ be a full locus of unbranched $\bZ/2 \times \bZ/2$ covers of surfaces in $\cH(2)$. Then $\cM$ is not geminal. In particular, there is an equivalence class consisting of four generically isometric cylinders.
\end{lem}

We use Lemma \ref{L:Z2xZ2covers} only to deduce the following. 

\begin{cor}\label{C:NoRankTwoTxT}
Suppose that $\cM$ is geminal of rank at least two. Suppose that $\left( (X, \omega), \cM, \bfC_1, \bfC_2 \right)$ is a generic diamond. Then the $\cM_{\bfC_i}$ cannot both be full loci of covers of $T \times T$  loci.
\end{cor}
\begin{proof}[Proof of Corollary \ref{C:NoRankTwoTxT} assuming Lemma \ref{L:Z2xZ2covers}]
Suppose not. Let $\cM$ be the smallest dimensional counterexample to the claim. 

By Proposition \ref{P:geminalrk1} \eqref{I:geminalrk1:CoverOfTxT:PiOpt}, for each surface in $\cM_{\bfC_i}$ there is an $\cM_{\bfC_i}$-optimal map to a $T \times T$ locus. We will denote these covers by $$\pi_i: \Col_{\bfC_i}(X, \omega) \ra (Z_i, \zeta_i).$$ By Remark \ref{R:TxTCP}, a $T \times T$ locus is a full locus of covers satisfying Assumption CP, via the degree 4 covers arising as the quotient of  the $\bZ/2\times \bZ/2$ group of translation symmetries. We will denote these covers by $f_i: (Z_i, \zeta_i) \ra (Y_i, \eta_i)$. Therefore, $\cM_{\bfC_i}$ is a full locus of covers satisfying Assumption CP, via the map $f_i\circ \pi_i$. 

Hence  Lemma \ref{L:IntroFull} implies that $\cM$ is a full locus of covers of a stratum $\cH$ of Abelian differentials.  These covers are precisely $\pi_{X_{min}}$ by Corollary \ref{C:MinimalCover}. The stratum $\cH$ must have a codimension one boundary component that is a stratum of genus 1 surfaces. This implies that $\cH = \cH(2, 0^m)$ for some nonnegative integer $m$.

 By definition, each $\pi_i$ is an optimal map. Thus Proposition \ref{P:geminalrk1} \eqref{I:geminalrk1:PiOptDegeneratesToPiOpt} gives that both $\Col_{\ColOne(\bfC_2)}(\pi_1)$ and $\Col_{\ColTwo(\bfC_1)}(\pi_2)$ are optimal maps for $\ColOneTwoX$. Since optimal maps are unique, we can conclude that    $\Col_{\ColOne(\bfC_2)}(\pi_1) = \Col_{\ColTwo(\bfC_1)}(\pi_2)$. Therefore by the Diamond Lemma (Lemma \ref{L:diamond}, see also Remark \ref{R:OptimalDiamond}), there is a map $\pi: (X, \omega) \ra (Z, \zeta)$ for which $\Col_{\bfC_i}(\pi) = \pi_i$. Since every translation cover is a factor of $\pi_{X_{min}}$ (Theorem \ref{T:MinimalCover}), $\pi$ is a factor of $\pi_{X_{min}}$ and hence that there is a degree four map $f: (Z, \zeta) \ra (Y, \eta)$ such that $f \circ \pi = \pi_{X_{min}}$. This implies that $(Y, \eta)\in  \cH(2, 0^m)$.


\begin{sublem}
$m=0$.
\end{sublem}

\begin{proof} 
The outline of the proof is that, if there are $m>0$ marked points on $(Y, \eta)$, we can  move a marked point to a zero to get a smaller dimensional counterexample, contradicting our assumption on $\cM$.

To realize this outline, we first observe that $$\overline{\pi_{X_{min}}(\bfC_i)}\subset (Y,\eta) \in \cH(2, 0^m)$$ does not contain any marked points. Indeed, our genericity assumptions imply that $\pi_{X_{min}}(\bfC_i)$ is simple, and if it contained a marked point in its boundary then degenerating it would give a surface in $\cH(2, 0^{m-1})$, contrary to our assumption that this degeneration has genus 1. 

Given a  marked point $p$ on $(Y,\eta)$ and a saddle connection $\gamma_p$ from the zero of $(Y,\eta)$ to $p$ that lies in the complement of $\pi_{X_{min}}(\bfC_1 \cup \bfC_2)$, we will let $\Col_{\gamma_p}$ denote the degeneration that fixes the rest of the surface while moving $p$ along $\gamma_p$ to the other endpoint of $\gamma_p$. We will call this collapsing $\gamma_p$ to a point, and  let $\Col_{\pi_{X_{min}}^{-1}(\gamma_p)}$ (resp. $\Col_{f^{-1}(\gamma_p)}$) denote the corresponding degeneration of $(X, \omega)$ (resp. $(Z, \zeta)$). We note that $\Col_{\pi_{X_{min}}^{-1}(\gamma_p)}$ commutes with $\Col_{\bfC_i}$. 

First we note that $\Col_{f^{-1}(\gamma_p)}(Z, \zeta)$ is not disconnected since $$\Col_{f^{-1}(\gamma_p)} \Col_{\pi(\bfC_i)}(Z, \zeta)$$ is connected. (Recall that $\Col_{\pi(\bfC_i)}(Z, \zeta)$ has genus $1$, so it does not admit disconnected degenerations.) 

Next we note that there is a covering map $$\Col_{\pi_{X_{min}}^{-1}(\gamma_p)}(X, \omega) \ra \Col_{f^{-1}(\gamma_p)}(Z, \zeta)$$ for which the preimage of the image of each cylinder in $\Col_{\pi_{X_{min}}^{-1}(\gamma_p)}(\bfC_1)$ is a single cylinder (this follows from the fact that $C$ is the preimage of its image under $\pi$ for any cylinder $C$ in $\bfC_1 \cup \bfC_2$; this in turn follows from the fact that $\Col_{\bfC_i}(\pi) = \pi_i$, which is a good map). This shows that $\Col_{\pi_{X_{min}}^{-1}(\gamma_p)}(X, \omega)$ is connected.

Letting $\cN$ be the orbit closure of $\Col_{\pi_{X_{min}}^{-1}(\gamma_p)}(X, \omega)$, we have that $\cN$ is a locus of connected surfaces of smaller dimension than $\cM$ and so 
$$\left( \Col_{\pi_{X_{min}}^{-1}(\gamma_p)}(X, \omega), \cN, \Col_{\pi_{X_{min}}^{-1}(\gamma_p)}(\bfC_1), \Col_{\pi_{X_{min}}^{-1}(\gamma_p)}(\bfC_2) \right)$$ is a diamond that forms a smaller dimensional counterexample to the claim. This is a contradiction, so we get $m=0$. 
\end{proof}

Let $\cM'$ be the orbit closure of $(Z, \zeta)$. We now have that $\cM'$ is a full locus of covers of $\cH(2)$ via the degree 4 cover $f$. For convenience we define $\bfC_i' = \pi(\bfC_i)$ and note that $\left( (Z, \zeta), \cM', \bfC_1', \bfC_2' \right)$ is a generic diamond. 

\begin{sublem}
The deck group of $f$ is $\bZ/2\times  \bZ/2$.
\end{sublem}

\begin{proof}Since $\Col_{\bfC_1'}(Z, \zeta), \Col_{\bfC_2'}(Z, \zeta)$ and $\Col_{\bfC_1', \bfC_2'}(Z, \zeta)$ are $T\times T$ loci, for every involution $\gamma\in \bZ/2\times \bZ/2$, we can find involutions $T_i^\gamma$  on $\Col_{\bfC_i'}(Z,\zeta)$ that agree one $\Col_{\bfC_1', \bfC_2'}(Z, \zeta)$ in the sense that 
$$\Col_{\Col_{\bfC_1'}(\bfC_2')} T_1^\gamma = \Col_{\Col_{\bfC_2'}(\bfC_1')} T_2^\gamma.$$
Keeping in mind that all degree two covering maps are the quotient by an involution, the Diamond Lemma thus gives the existence of an involution $T^\gamma$ on $(Z,\zeta)$ with $\Col_{\bfC_i'}(T^\gamma) = T^\gamma_i$. 

This gives a $\bZ/2\times \bZ/2$ group of symmetries on $(Z,\zeta)$, and the quotient by this group is the degree four map $f$. 
\end{proof}

By our definition of translation covers, $f$ can be branched at most over the zero of $(Y,\eta)\in \cH(2)$. However, a regular cover of a closed surface with abelian deck group cannot be branched over a single point, because the loop around that point is null-homologous in the surface with the single point removed. Hence, we get that $f$ is unbranched.

Lemma \ref{L:Z2xZ2covers} now implies that $(Z ,\zeta)$ has no marked points and contains an equivalence class of four generically isometric cylinders. Because of the lack of marked points, each cylinder on a surface in $\cM'$ lifts to a collection of cylinders of the same height on a surface in $\cM$, so this contradicts the fact that $\cM$ is geminal and concludes the proof.
\end{proof}

\begin{proof}[Proof of Lemma \ref{L:Z2xZ2covers}]
Let $(X',\omega') \in \cH(2)$. An unbranched $\bZ/2 \times \bZ/2$ cover of $X'$  is specified by a homomorphism from $\pi_1(X')$ to $\bZ/2 \times \bZ/2$, which is equivalent to specifying two nonzero distinct elements of $H^1(X', \bZ/2\bZ)$. There are finitely many choices for these two nonzero elements, and, roughly speaking, our proof will simply  check that each of them doesn't give a geminal orbit closure. It will be helpful to keep in mind that, by Poincare duality, $H^1(X', \bZ/2\bZ)$ is isomorphic to $H_1(X', \bZ/2\bZ)$. 

Consider the $\bZ$-module

$$  W = \left( \bZ (w_1-w_0) \oplus \cdots \oplus \bZ (w_5-w_0) \right) / V,$$
where the $w_i$ are formal variables and 
\[ V =  \span_{\bZ}\left( \left\{ 2w_i - 2w_j \right\}_{i,j \in \{0, \hdots, 5 \}} \cup \left\{\sum_{i=1}^5  (w_i - w_0) \right\} \right).  \]
We will recall an isomorphism $\phi$ from this $\bZ$-module to $H_1(X', \bZ/2\bZ)$ as follows.  

Associate to each $w_i$ a Weierstrass point of $X'$, which we will also denote $w_i$.  

 Notice that every nonzero element of $  W   $ may be written uniquely as $w_i - w_j$ where $i > j$ are elements of $\{0, \hdots, 5\}$. ($W$ is a $\bZ/2\bZ$ vector space of dimension 4, and hence contains $2^4-1=15$ non-zero elements.) For each $i>j$, associate to $w_i - w_j$ any path $\gamma_{i,j}$ from $w_i$ to $w_j$, and define  $$\phi(w_i - w_j) = \gamma_{i,j} - J(\gamma_{i,j}) \in H_1(X', \bZ)/2H_1(X', \bZ),$$ where $J$ is the hyperelliptic involution. Notice that $\gamma_{i,j} - J(\gamma_{i,j})$ is a closed curve.  
 
 Since $\gamma_{i,j}$ is well-defined up to concatenating by a closed curve $\gamma_i$ based at $w_i$ and a closed curve $\gamma_j$ based at $w_j$, and since 
\[ (  \gamma_{i,j} + \gamma_i + \gamma_j) - J(  \gamma_{i,j} + \gamma_i + \gamma_j) = \phi(w_i - w_j) + 2(\gamma_i + \gamma_j), \]
$\phi(w_i - w_j)$ is well defined. From this point it is easy to check that $\phi$ is an isomorphism. 

We may assume $w_0$ is the zero of $\omega'$. A loop in $\cH(2)$ acts by permutations on $\{w_0, \hdots, w_5\}$ that fix $w_0$. This induces an action of $\pi_1^{orb}\left( \cH(2) \right)$ on $H_1(X', \bZ/2\bZ)$ using the isomorphism above. The action agrees with the composition of the natural map from $\pi_1^{orb}$ to the mapping class group and the natural map from the mapping class group to $\mathrm{Sp}_4(\bZ/2\bZ)$, which acts on $H_1(X', \bZ/2\bZ)$. 

Given $1$-cochains $v_i \in H^1(X', \bZ/2\bZ)$ for $i \in \{1, 2\}$ which specify a degree four cover contained in $\cM$, $\cM$ also contains the degree four cover specified by $(g \cdot v_1, g \cdot v_2)$ for any $g \in \pi_1^{orb}(\cH(2))$. 

Note that the action of $\pi_1^{orb}(\cH(2))$ fixes $w_0$ and acts by a surjective homomorphism to $\mathrm{Sym}(5)$ on $\{w_1, \hdots, w_5\}$; the surjectivity can be seen for example by using Dehn twists in the horizontal and vertical direction on an $L$-shaped surface in $\cH(2)$. Each Dehn twist in a cylinder produces a transposition in $\mathrm{Sym}(5)$. Up to re-indexing the Weierstrass points, the transpositions produced by the horizontal and vertical cylinders on an $L$-shaped surfaces are $\{ (1 2), (2 3), (3 4), (4 5)\}$, which generate $\mathrm{Sym}(5)$.

Therefore, representatives of each orbit of $\pi_1^{orb}(\cH(2))$ on sets of two distinct nonzero elements of $H^1(X', \bZ/2\bZ)$ are given by the following list
\[ \left( w_1 - w_0, w_2 - w_0 \right),\quad \left( w_1 - w_0, w_2 - w_1 \right),\quad \left( w_1 - w_0, w_3 - w_2 \right), \] 
\[ \left( w_2 - w_1, w_3 - w_2 \right),\quad \left( w_2 - w_1, w_3 - w_4 \right). \]
  (This list can be understood by considering whether $w_0$ appears in zero, one, or two of the two nonzero elements, and how many $w_i, i>0$ appear in both nonzero elements.)  
The isomorphism from $H_1(X', \bZ/2\bZ)$ to $H^1(X', \bZ/2\bZ)$ maps a chain $\gamma$ to intersection number with $\gamma$ modulo $2$. The number of intersections between nonzero classes $w_i - w_j$ and $w_k - w_\ell$ (modulo $2$) is $1$ if $\{i, j\} \cap \{k, \ell\}$ is a singleton and zero otherwise. 

It is easy to see that there is an open dense set of $\cH(2)$ with a cylinder between any two Weierstrass points (excepting $w_0$). (In fact this open dense set is all of $\cH(2)$, but we will not need this.) Without loss of generality suppose that $(X',\omega')$ belongs to this open dense set. 

For the first four entries in the above list, we see that a cylinder passing through $w_4$ and $w_5$ must lift to four $\cM$-equivalent cylinders, hence these covers are not geminal. For the final entry, we see that a cylinder passing through $w_2$ and $w_1$  must lift to four $\cM$-equivalent cylinders and again these covers cannot be geminal. 
\end{proof}


\begin{lem}\label{L:NoRank2JandTT}
Suppose that $( (X, \omega), \cM, \bfC_1, \bfC_2)$ is a generic diamond where $\MOne$ is a $T \times T$ locus and  $\MTwo$ is a quadratic double. Then the surfaces in $\cM$ don't have marked points, and $\cM$ contains surfaces with an equivalence class of at least 3 cylinders that  have the same height on all nearby surface in $\cM$. 
\end{lem}

The proof gives complete information on $\cM$, but we state here only what we will apply. Lemma \ref{L:NoRank2JandTT} immediately implies the following, and will also be used in a more subtle application later. 

\begin{cor}\label{C:NoRank2JandTT} 
Under the same assumptions, $\cM$ cannot be geminal. 
\end{cor}

\begin{proof}[Proof of Lemma \ref{L:NoRank2JandTT}]
Suppose not. The base of the diamond is necessarily a quadratic double that is also type $T \times T$. By Lemma \ref{L:Types}, $\MOneTwo$ is necessarily a quadratic double of $\cQ(-1^4)$ and at least three preimages of poles are marked. (Note that this implies $\MOne$ consists of $\bZ/2\times \bZ/2$ covers of surfaces in $\cH(0,0)$.)

Let $f_1$ (resp. $f_2$) denote the quotient map, whose domain is $\ColOneX$ (resp. $\ColTwoX$), by the $\bZ/2\times \bZ/2$ action (resp. the holonomy involution).  Let $\Col(f_1)$ and $\Col(f_2)$ denote the induced maps on $\ColOneTwoX$. 

We now use redundant notation to make clear the connection to previous work. 
Let $g_1$ be the map defined on surfaces in the quadratic double of $\cQ(-1^4)$ that is the quotient by the two torsion subgroup, so the codomain is surfaces in $\cH(0)$. Let $g_2$ be the quotient by the holonomy involution, defined again on surfaces in the quadratic double of $\cQ(-1^4)$. Let $g$ be the identity map. Then $\Col(f_i)= g_i \circ g$. Now,  \cite[Lemma 8.24]{ApisaWrightDiamonds}  shows that $\cM$ is a locus of degree 2 covers of the Prym locus in $\cH(4)$, and moreover $\cM$ contains surfaces as in Figure \ref{F:Deg2Prym4Cyls}, with some gluing of the dotted horizontal edges that gives a cover of a Prym form. 

\begin{figure}[h]
\includegraphics[width=0.8\linewidth]{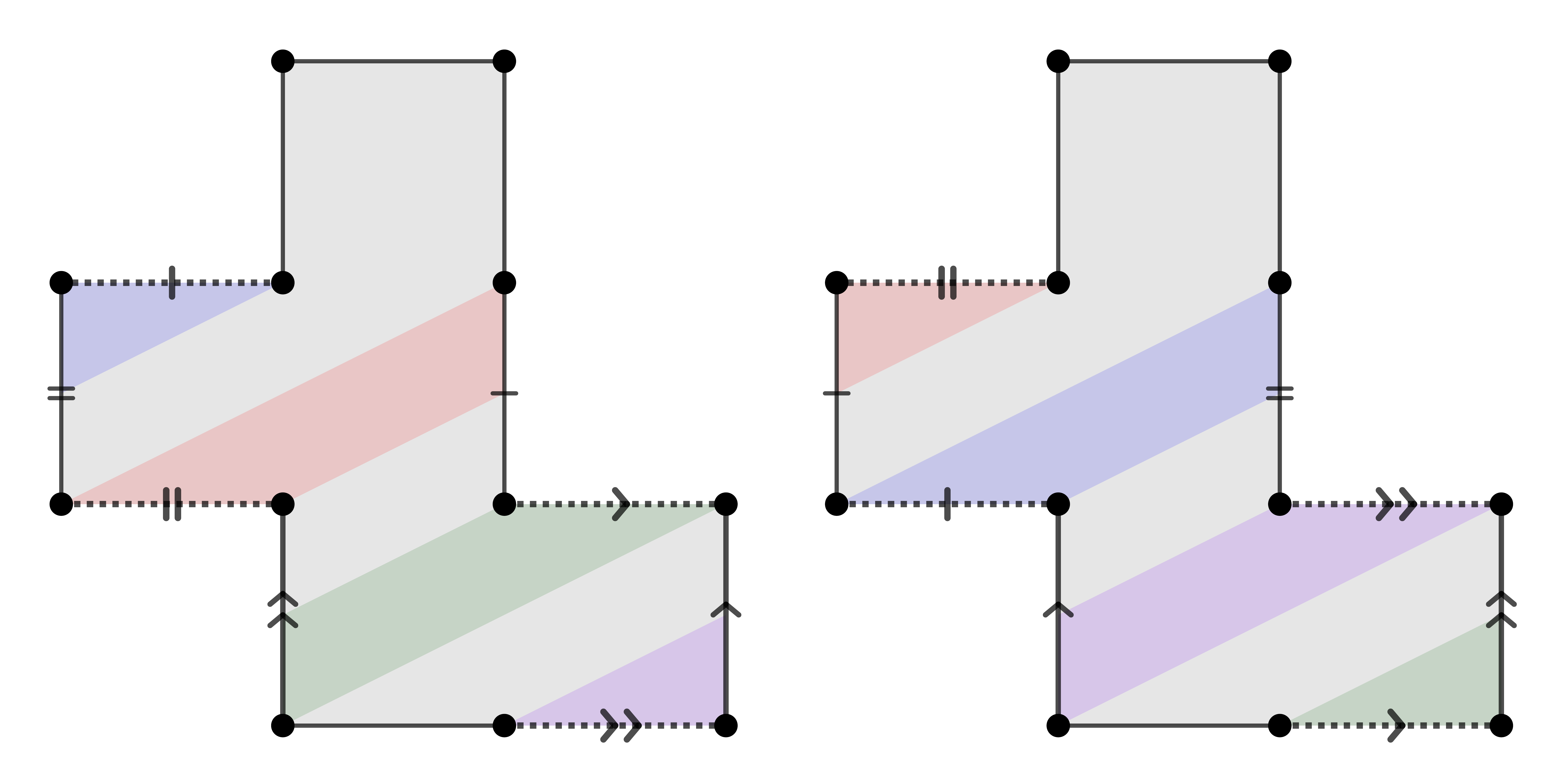}
\caption{The proof of Lemma \ref{L:NoRank2JandTT}.}
\label{F:Deg2Prym4Cyls}
\end{figure}

By considering the vertical cylinders that go through the dotted lines, it is easy to get the result for all possible gluings except possibly the one indicated in the Figure \ref{F:Deg2Prym4Cyls}. However, with that gluing, the figure illustrates a collection of four cylinders that prove the result.
\end{proof}

\begin{proof}[Proof of Proposition \ref{P:TxTDiamond}:]
Suppose not. Then $\MOneTwo$ is a $T \times T$ locus in a stratum of tori. By Lemma \ref{L:Types}, $\MTwo$ cannot be an Abelian double or a component of a stratum of Abelian differentials. By Corollary \ref{C:NoRank2JandTT}, $\MTwo$ cannot be a quadratic double. By assumption, $\MTwo$ is therefore contained in a stratum of tori. Since $\MTwo$ is not a stratum of Abelian or quadratic double,  Proposition \ref{P:0n} gives that $\MTwo$ must be a $T \times T$ locus. But this is impossible by Corollary \ref{C:NoRankTwoTxT}, so we have a contradiction. 
\end{proof}

%
%
%
%

\section{Proof of Theorem \ref{T:geminal} }\label{S:MainProof}


We will actually prove the following stronger version of Theorem \ref{T:geminal}, which is easier to prove by induction because the inductive hypothesis provides more information. Recall that optimal maps are defined in Definition \ref{D:GoodAndOptimal}, and if they exist they are always unique, and that $T\times T$ loci are defined in Definition \ref{D:TxT}. Subequivalence classes and generic cylinders are defined in Definition \ref{D:SE}, and for geminal subvarieties consist of either a single free cylinder or a pair of twins.

\begin{thm}\label{T:geminal2} 
Suppose that $\cM$ is a geminal invariant subvariety. 
\begin{enumerate}
\item\label{I:geminal2:WhatIsM} $\cM$ is one of the following:
\begin{enumerate}
\item\label{I:geminal2:ComponentOrDouble} A connected component of a stratum of Abelian differentials or an Abelian or quadratic double.
\item\label{I:geminal2:CoverOfQuadDouble} A full locus of covers of a quadratic double of a genus zero stratum. 
\item\label{I:geminal2:CoverOfTxT} A full locus of covers of a $T \times T$ locus. 
\end{enumerate}
\item\label{I:geminal2:PiOpt} Every surface  in $\cM$ has an $\cM$-optimal map $\pi_{opt}$, and:
\begin{enumerate}
\item\label{I:geminal2:ComponentOrDouble:PiOpt} In Case \eqref{I:geminal2:ComponentOrDouble}, $\pi_{opt}$ is the identity. 
\item\label{I:geminal2:CoverOfQuadDouble:PiOpt} In  Case \eqref{I:geminal2:CoverOfQuadDouble}, $\pi_{opt}$ is the covering map to surfaces in the quadratic double. When the rank is one, this covering map is the minimal degree map to a torus. 
\item\label{I:geminal2:CoverOfTxT:PiOpt} In Case \eqref{I:geminal2:CoverOfTxT}, $\pi_{opt}$ is the covering map to surfaces in the $T\times T$ locus, and is moreover the minimal degree map to a torus.
\end{enumerate}
\item\label{I:geminal2:TwoCylinders} If the degree of the optimal map is greater than one, then each subequivalence class contains exactly two cylinders. 
\item\label{I:geminal2:PiOptDegeneratesToPiOpt} If $\bfC$ is a subequivalence class of $\cM$-generic cylinders on a surface $(X, \omega)$ in  $\cM$, then $\Col_{\bfC}(\pi_{opt})$ is the $\cM_{\bfC}$-optimal map for $\Col_{\bfC}(X, \omega)$.
\end{enumerate}
\end{thm}

\begin{rem}
Recall our convention that for Abelian doubles the preimage of every marked point is marked. Similarly, for quadratic doubles every point in the preimage of a free point is marked and the preimage of a pole may or may not be marked. Recall finally that ``translation cover" has been defined in Definition \ref{D:Covering} and that it requires the preimage of any marked point on the codomain to include either a marked point or a singular point in the domain. 
\end{rem}

\begin{proof}[Proof of Theorem \ref{T:geminal2}:]
Proceed by induction on the dimension of $\cM$. Notice that the result holds when the rank of $\cM$ is one by Proposition \ref{P:geminalrk1}, keeping in mind for \eqref{I:geminal2:CoverOfQuadDouble} that if a  double of a stratum $\cQ$ of quadratic differentials consists of tori, then $\cQ$ must be genus 0. This establishes the base case and allows us to assume that the rank of $\cM$ is at least two. 
Thus, for the remainder of the paper, we make the following standing assumptions. 

\begin{ass}
$\cM$ is a geminal invariant subvariety of rank at least 2, and all  geminal subvarieties of smaller dimension satisfy Theorem \ref{T:geminal2}. 
\end{ass}

\subsection{Finding a diamond of connected surfaces} In this section we establish the following starting point for our analysis. 

\begin{lem}\label{L:ConnectedDiamonds}
Suppose that $\bfD_1$ is an $\cM$-generic  subequivalence class of cylinders on a surface $(X, \omega) \in \cM$, and that the $\cM_{\bfD_1}$-optimal map on $\Col_{\bfD_1}(X, \omega)$ is the identity. Then, possibly after changing $(X, \omega)$ to a different surface, there is a generic diamond 
$$((X, \omega), \cM, \bfC_1, \bfC_2)$$
such that the generic surface in $\cM_{\bfC_i}$ and $\MOneTwo$ is connected and has the identity as its optimal map. 
\end{lem}

Recall that, by Lemma \ref{L:DisconnectedDegenerationMaps}, if $\Col_{\bfD_1}(X, \omega)$ is disconnected then the optimal map on $\Col_{\bfD_1}(X, \omega)$ is the identity. 

\begin{proof}

Since $\cM$ has rank at least two, we can replace $(X,\omega)$ with a deformation satisfying the 
assumptions of Lemma \ref{L:GenericDiamond} and obtain subequivalence class $\bfD_2$ that is not parallel to $\bfD_1$ and such that $$((X, \omega), \cM, \bfD_1, \bfD_2)$$ is a generic skew diamond. After shearing the $\bfD_i$, we may assume it is a diamond. 


It is easy to see that if either $\Col_{\bfD_i}(X,\omega)$ is disconnected, then so is $\Col_{\bfD_1, \bfD_2}(X, \omega)$. The case where $\Col_{\bfD_1, \bfD_2}(X, \omega)$ is connected is concluded as follows. 

\begin{sublem}\label{SL:OptimalOK}
If $((X, \omega), \cM, \bfD_1, \bfD_2)$ is a generic diamond where the identity is the $\cM_{\bfD_1}$-optimal map for $\Col_{\bfD_1}(X, \omega)$, then the identity is the optimal map for the generic surface in $\cM_{\bfD_2}$ and $\cM_{\bfD_1, \bfD_2}$.
\end{sublem}

\begin{proof}

We begin by showing that the $\cM_{\bfD_1, \bfD_2}$-optimal map on $\Col_{\bfD_1, \bfD_2}(X, \omega)$ is the identity. Indeed, if $\Col_{\bfD_1}(X, \omega)$ is connected, this follows from part \eqref{I:geminal2:PiOptDegeneratesToPiOpt} of the induction hypothesis, and if it is disconnected it follows as in Lemma \ref{L:DisconnectedDegenerationMaps}. 

Lemma \ref{L:DegeneratingOptimalMap} now gives that the $\cM_{\bfD_2}$-optimal map on $\Col_{\bfD_2}(X,\omega)$ is the identity. 
\end{proof}

So suppose without loss of generality that $\Col_{\bfD_1, \bfD_2}(X, \omega)$ is disconnected. Note that Sublemma \ref{SL:OptimalOK} shows that the identity is the $\cM_{\bfD_i}$-optimal map for both $\Col_{\bfD_i}(X, \omega)$, so the roles of $\bfD_1$ and $\bfD_2$ may be interchanged without loss of generality.

We begin by showing that $\Col_{\bfD_i}(X, \omega)$ is connected for at least one $i \in \{1, 2\}$. To see this, suppose that $\Col_{\bfD_2}(X, \omega)$ is disconnected. By Lemma \ref{L:Disconnect}, $\cM_{\bfD_2}$ is a subset of $\cH_1 \times \cH_2$ where $\cH_1$ and $\cH_2$ are components of strata of Abelian differentials. Moreover, the projection from $\cM_{\bfD_2}$ to $\cH_j$ is a local diffeomorphism for $j \in \{1, 2\}$. It follows that $\Col_{\bfD_2}(\bfD_1)$ and hence $\bfD_1$ is a pair of simple cylinders, since  cylinders that are generic for strata are simple. Therefore, $\Col_{\bfD_1}(X, \omega)$ is connected, since (in total generality) collapsing any collection of disjoint simple cylinders cannot disconnect a surface. Thus, after perhaps re-indexing we may assume without loss of generality that $\Col_{\bfD_1}(X, \omega)$ is connected.

Because the $\cM_{\bfD_1}$-optimal map on $\Col_{\bfD_1}(X, \omega)$ is the identity, it follows from parts \eqref{I:geminal2:WhatIsM} and \eqref{I:geminal2:PiOpt} of the inductive hypothesis that  $\cM_{\bfD_1}$ is a stratum, double, or $T\times T$ locus. Since degenerations of generic subequivalence classes in strata and $T\times T$ loci are always connected, and since $\Col_{\bfD_1, \bfD_2}(X, \omega)$ is disconnected, we get that $\cM_{\bfD_1}$ is an Abelian or quadratic double. If $\Col_{\bfD_2}(X, \omega)$ is connected, we get the same statement for $\cM_{\bfD_2}$. 

It will be useful to recall the following definitions, which were used in \cite{ApisaWrightDiamonds}. A \emph{complex cylinder} is a cylinder for which each boundary consists of two distinct saddle connections of equal length. A \emph{complex envelope} is a cylinder in a quadratic differential where one boundary consists of a saddle connection joining two distinct poles and the other consists of two distinct saddle connections of equal length. 

\begin{sublem}\label{SL:ConnectedBase}
Suppose that $\left( (X, \omega), \cM, \bfC_1, \bfC_2 \right)$ is a generic diamond and that $\MOne$ is an Abelian or quadratic double. It is possible to perform a partial Dehn twist in $\bfC_2$ so that $\ColOneTwoX$ is connected, except when $\MOne$ is a quadratic double and  $\ColOne(\bfC_2)$ consists of a pair of complex cylinders.  
\end{sublem}

\begin{proof}
Suppose that $\ColOneTwoX$ is disconnected.

In an Abelian double, every subequivalence class consists of two simple cylinders or one complex cylinder (see for instance  \cite[Lemma 3.18]{ApisaWrightDiamonds}). Since collapsing simple cylinders cannot disconnect the surface, in the case where $\MOne$ is an Abelian double, $\bfC_2$ must be a single complex cylinder. By applying a half Dehn twist to $\bfC_2$ we can arrange for $\ColOneTwoX$ to be connected. See Figure \ref{F:DisconnectingAnAbelianDouble}.
\begin{figure}[h]
\includegraphics[width=0.7\linewidth]{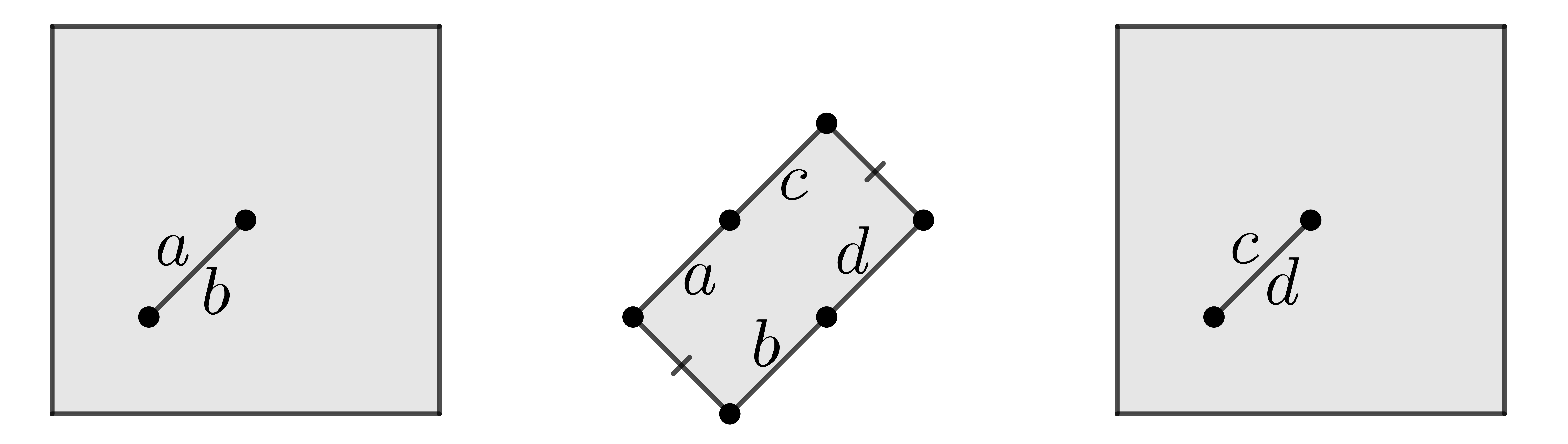}
\caption{A surface in a double of $\cH(2)$. Degenerating the complex cylinder bounded by $a,b,c$ and $d$ disconnects the surface, but not if one first performs a half Dehn twist.}
\label{F:DisconnectingAnAbelianDouble}
\end{figure}

We may therefore assume that $\MOne$ is a quadratic double.

We now summarize some claims that follow from Masur-Zorich \cite{MZ}, as in  \cite[Lemma 4.11 (2), Remark 6.2]{ApisaWrightDiamonds}. Consider a subequivalence class of cylinders $\bfC$ on a surface $(Y, \eta)$ in a quadratic double with holonomy involution $J$, and assume that $\bfC$ is generic in the sense that, for every cylinder in $\bfC$, its boundary saddle connections are generically parallel to the cylinder in the quadratic double. It is easy to see that $\Col_{\bfC}(Y, \eta)$ is disconnected if and only if $\Col_{\bfC/J}\left( (Y, \eta)/J \right)$ has trivial linear holonomy. Furthermore, $\Col_{\bfC/J}\left( (Y, \eta)/J \right)$ has trivial linear holonomy only if $\bfC/J$ is a complex envelope or a complex cylinder and, in the former case, it is possible to perform a partial Dehn twist in $\bfC/J$ so that $\Col_{\bfC/J}\left( (Y, \eta)/J \right)$ has nontrivial linear holonomy. 

Applying the previous claims to $\ColOneX$ concludes the proof.
\end{proof}

Except when $\cM_{\bfD_1}$ is a quadratic double and $\Col_{\bfD_1}(\bfD_2)$ is a pair of complex cylinders, we can apply Sublemma \ref{SL:ConnectedBase}  to change $\bfD_2$ by a partial Dehn twist, and obtain a diamond of connected surfaces. Since we know the identity is $\cM_{\bfD_1}$-optimal for $\Col_{\bfD_1}(X,\omega)$, Sublemma \ref{SL:OptimalOK} gives the desired statements about optimal maps. 

So assume $\cM_{\bfD_1}$ is a quadratic double and $\Col_{\bfD_1}(\bfD_2)$ is a pair of complex cylinders. Let $J_1$ denote the holonomy involution on $\Col_{\bfD_1}(X, \omega)$.  
We will show that there is an equivalence class $\bfD_3$ of cylinders such that $\left( (X, \omega), \cM, \bfD_1, \bfD_3 \right)$ forms a generic skew diamond with $\Col_{\bfD_1}(X, \omega)$, $\Col_{\bfD_3}(X, \omega)$, and $\Col_{\bfD_1, \bfD_3}(X, \omega)$ connected. One can then shear to get a generic diamond, and Sublemma \ref{SL:OptimalOK} will then give the desired statements about optimal maps, since the identity is $\cM_{\bfD_1}$-optimal for $\Col_{\bfD_1}(X, \omega)$.

By Masur-Zorich \cite{MZ}, as described in \cite[Theorem 4.8 (2)]{ApisaWrightDiamonds}, $\Col_{\bfD_1}(X, \omega)/J_1 - \Col_{\bfD_1}(\bfD_2)/J_1$ consists of two disjoint connected translation surfaces with boundary, where the boundary of each component consists of two saddle connections (see Figure \ref{F:ComplexCylinderDisconnects}). Moreover, two saddle connections contained on different components are generically parallel if and only if they are generically parallel to the boundary saddle connections of $\Col_{\bfD_1}(\bfD_2)/J_1$. Since $\Col_{\bfD_1}(\bfD_1)$ is not parallel to $\Col_{\bfD_1}(\bfD_2)$, it follows that $\Col_{\bfD_1}(\bfD_1)/J_1$ is contained on one component of $\Col_{\bfD_1}(X, \omega)/J_1 - \Col_{\bfD_1}(\bfD_2)/J_1$. 
\begin{figure}[h]
\includegraphics[width=0.7\linewidth]{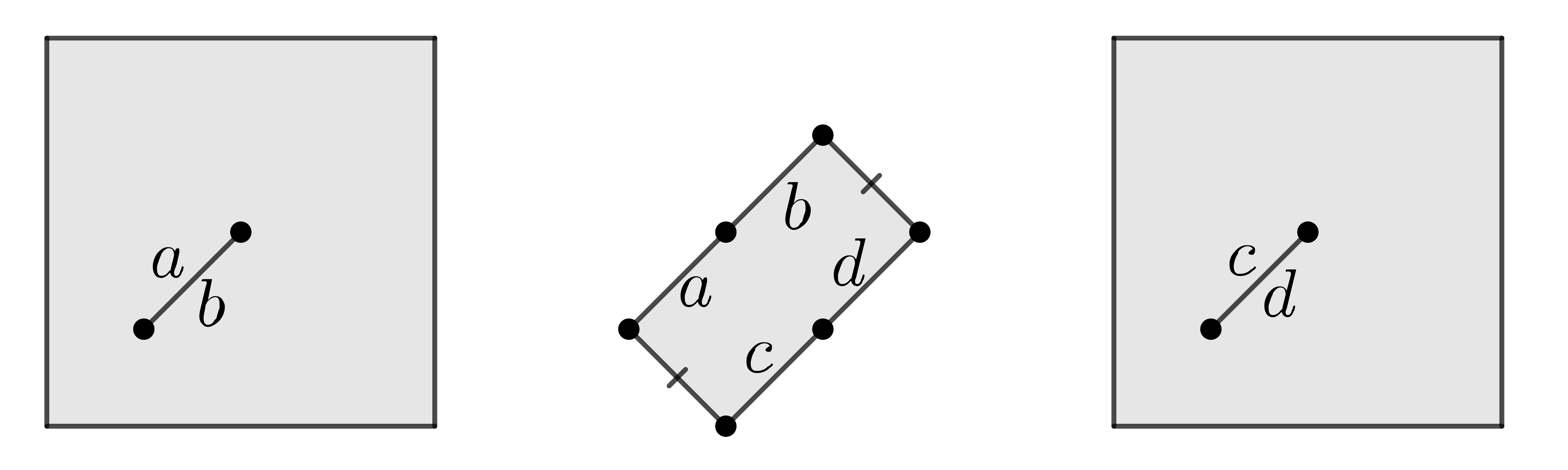}
\caption{An example of a complex cylinder on a quadratic differential. Note that some of the gluings require rotation by $\pi$.}
\label{F:ComplexCylinderDisconnects}
\end{figure}

All cylinders on a generic translation surface are simple. Each component of $\Col_{\bfD_1}(X, \omega)/J_1 - \Col_{\bfD_1}(\bfD_2)/J_1$ is a translation surface with two boundary saddle connections, which can be glued together to get a closed surface with a distinguished saddle connection. By considering a cylinder disjoint from that saddle connection, we conclude that we can find a simple cylinder on the component of $\Col_{\bfD_1}(X, \omega)/J_1 - \Col_{\bfD_1}(\bfD_2)/J_1$ not containing $\Col_{\bfD_1}(\bfD_1)/J_1$. Letting $\bfD_3$ denote the preimage of this simple cylinder we see that $\bfD_3$ is a pair of simple cylinders. Since degenerating simple cylinders cannot disconnect a surface, we see that $\left( (X, \omega), \cM, \bfD_1, \bfD_3 \right)$ forms the desired generic skew diamond. 
\end{proof}

\subsection{When $\pi_{opt}$ is the identity for some degeneration.}\label{SS:SometimesId}
We first prove Theorem \ref{T:geminal2} when there is a surface $(X, \omega)\in \cM$ with an $\cM$-generic subequivalence class of cylinders $\bfD_1$ such that the identity is the $\cM_{\bfD_1}$-optimal map on $\Col_{\bfD_1}(X,\omega)$. 

By Lemma \ref{L:ConnectedDiamonds}, there is a generic diamond $((X, \omega), \cM, \bfC_1, \bfC_2)$ such that the generic surface in $\cM_{\bfC_i}$ and $\MOneTwo$ is connected and has the identity as its optimal map.

By parts \eqref{I:geminal2:WhatIsM} and \eqref{I:geminal2:PiOpt} of the induction hypothesis, each $\cM_{\bfC_i}$ is one of the following: 
\begin{enumerate}
\item a component of a stratum of Abelian differentials, 
\item an Abelian double, a quadratic double, or 
\item a $T \times T$ locus. 
\end{enumerate} 
Since the rank of $\cM$ is at least two, Proposition \ref{P:TxTDiamond} gives that actually each $\cM_{\bfC_i}$ is one of the first two possibilities. By Proposition \ref{P:DiamondsAreCool}, $\cM$ is a stratum of Abelian differentials, an Abelian double, or a quadratic double. This proves part \eqref{I:geminal2:WhatIsM} of Theorem \ref{T:geminal2}. 

We now show that the identity is the $\cM$-optimal map on $(X, \omega)$. It is vacuously an $\cM$-good map. Moreover, there is no other $\cM$-good map $\pi'$ since otherwise, by Lemma \ref{L:DegeneratingOptimalMap}, $\ColOne(\pi')$ would be a nontrivial $\cM_{\bfC_1}$-good map on $\ColOneX$, which contradicts our assumption that the identity is $\cM_{\bfC_1}$-optimal for $\ColOneX$. Therefore, the identity is the $\cM$-optimal map on $(X,\omega)$.  This proves parts \eqref{I:geminal2:PiOpt} and \eqref{I:geminal2:TwoCylinders} of Theorem \ref{T:geminal2}. 

For any subequivalence class of cylinders $\bfC$, $\Col_{\bfC}(X, \omega)$ is either disconnected, in which case the identity is the $\cM_{\bfC}$-optimal map by Lemma \ref{L:DisconnectedDegenerationMaps}, or $\cM_{\bfC}$ is a stratum or an Abelian or quadratic double, in which case part \eqref{I:geminal2:PiOpt} of the induction hypothesis gives that the $\cM_{\bfC}$-optimal map is the identity. This proves part \eqref{I:geminal2:PiOptDegeneratesToPiOpt} of Theorem \ref{T:geminal2}.

\subsection{When $\pi_{opt}$ is never the identity for any degeneration.}\label{SS:NeverId}
It remains to prove Theorem \ref{T:geminal2} when $\Col_{\bfC}(X, \omega)$ always has an $\cM_{\bfC}$-optimal map of degree greater than one for any $\cM$-generic subequivalence class of cylinders $\bfC$, so we assume that this is the case. 

We start by observing that we no longer have to worry about disconnected degenerations. 

\begin{lem}\label{L:NoDisc}
For any generic diamond $\left( (X, \omega), \cM, \bfC_1, \bfC_2 \right),$ the surfaces $\ColOneX$, $\ColTwoX$, and $\ColOneTwoX$ are all connected. 
\end{lem}

\begin{proof}
By assumption, $\pi_{opt}$ is not the identity for $\ColOneX$ or  $\ColTwoX$. So Lemma \ref{L:DisconnectedDegenerationMaps} gives that  these surfaces are connected.

Part \eqref{I:geminal2:PiOptDegeneratesToPiOpt} of the induction hypothesis gives that $\pi_{opt}$ is not the identity for $\ColOneTwoX$, so again Lemma \ref{L:DisconnectedDegenerationMaps}  gives that this surface is connected. 
\end{proof}

Say that $(X, \omega) \in \cM$ is \emph{strongly generic} if it has dense orbit in $\cM$ and satisfies the property that all parallel saddle connections are $\cM$-parallel; such surfaces are dense in $\cM$. 

Recall that the map $\pi_{X_{min}}$, given by Theorem \ref{T:MinimalCover}, is the highest degree translation cover with domain $(X, \omega)$ and that all other translation covers with domain $(X, \omega)$ are factors of it. Since $(X,\omega)$ has dense orbit in $\cM$  this map is $\cM$-generic.

In the following lemma, we allow $\bfC_1$ to be endowed with a choice of direction in which there is a saddle connection in $\overline\bfC_1$, as in Section \ref{S:Diamond}; that is, we allow ``non-perpendicular" cylinder collapses. 

\begin{lem}\label{L:PiMinDegeneratesToPiOpt}
Suppose that $(X, \omega) \in \cM$ is strongly generic. For any subequivalence class of generic cylinders $\bfC_1$ on $(X,\omega)$, $\Col_{\bfC_1}(\pi_{X_{min}})$ is the $\cM_{\bfC_1}$-optimal map for $\Col_{\bfC_1}(X, \omega)$. 
\end{lem}

\begin{proof}
Pick $\bfC_2$ to be a disjoint and non-parallel subequivalence class on $(X,\omega)$ such that $((X, \omega), \cM,  \bfC_1, \bfC_2)$ forms a generic skew diamond; this exists by Lemma \ref{L:GenericDiamond}.

 Let $\pi_i$ denote the $\cM_{\bfC_i}$-optimal map on $\Col_{\bfC_i}(X, \omega)$. 

Part \eqref{I:geminal2:PiOptDegeneratesToPiOpt} of the  induction hypothesis implies that $$\Col_{\ColOne(\bfC_2)}(\pi_1) = \Col_{\ColTwo(\bfC_1)}(\pi_2).$$  Together with  Remark \ref{R:OptimalDiamond}, this allows us to apply the Diamond Lemma (Lemma \ref{L:diamond}), and conclude that there is a translation cover $$\pi: (X, \omega) \ra (X', \omega')$$ such that $\Col_{\bfC_i}(\pi) = \pi_i$. 

It now suffices to show that $\pi=\pi_{X_{min}}$. By Theorem \ref{T:MinimalCover}, it suffices to show that $(X', \omega')$ does not admit any non-trivial translation coverings. 

We will now consider the diamond $$\left( (X', \omega'), \cM', \bfC_1', \bfC_2' \right),$$ where $\cM'$ is the orbit closure of $(X', \omega')$ and where $\bfC_i' := \pi\left( \bfC_i \right)$. 

  Since $\pi_i$ is optimal and $\cM_{\bfC_i'}'$ arises from $\cM_{\bfC_i}$ by taking quotients by $\pi_i$, we know that $\cM_{\bfC_i'}'$ is geminal and that its optimal map is the identity. However, we know even more than this, since we know $\pi_i$ was not the identity. In particular, applying parts \eqref{I:geminal2:WhatIsM} and \eqref{I:geminal2:PiOpt} of the induction hypothesis to $\cM_{\bfC_i}$ (rather than $\cM_{\bfC_i'}'$), we see that  each $\cM_{\bfC_i'}'$ is either a quadratic double of a genus zero stratum or a $T \times T$ locus. 

At present we do not know if $\pi$ is good, so we do not know that $\cM'$ is geminal. In particular, we cannot immediately apply Proposition \ref{P:TxTDiamond} to $\cM'$. We will instead proceed in three cases.

\bold{Case 1: Both $\cM_{\bfC_i'}'$ are $T \times T$ loci.} This cannot occur by Corollary \ref{C:NoRankTwoTxT}, applied to our original diamond $((X, \omega), \cM,  \bfC_1, \bfC_2)$.  

\bold{Case 2: Both $\cM_{\bfC_i'}'$ are quadratic doubles of genus zero strata.} In previous work, the authors  classified diamonds where both sides are quadratic doubles \cite[Theorem 7.1]{ApisaWrightDiamonds}. Because $\ColOneTwoX$ is connected, this result implies that $\cM'$ is a quadratic double of a component $\cQ'$ of a stratum of quadratic differentials. Let $J'$ denote the holonomy involution on $(X', \omega')$. 

By Corollary \ref{C:MinimalCover}, to prove the lemma it suffices to show that $\For(\cQ')$ is not a hyperelliptic component. 

Since $\Col_{\bfC_2'}(\bfC_1')/\Col_{\bfC_2'}(J')$ is a generic cylinder in a genus zero stratum, it is either a simple cylinder or simple envelope by Masur-Zorich \cite{MZ}, as discussed in Apisa-Wright \cite[Theorem 4.8]{ApisaWrightDiamonds}. Therefore, $\bfC_1'/J'$ is a simple cylinder or simple envelope. 

The surface $(X', \omega')/J'$  is formed by gluing $\bfC_1'/J'$ into the genus zero surface $\Col_{\bfC_1'}(X', \omega')/\Col_{\bfC_1'}(J')$. Gluing in a simple cylinder (resp. a simple envelope) to a sphere produces a surface of genus at most one (resp. zero). So $\cQ'$ is genus 0 or 1. 

Since $\cM$ has rank at least two, so does $\cQ'$. 

Lanneau classified hyperelliptic components \cite{LanneauHyp} of strata of quadratic differentials. See also \cite[Section 9.1]{ApisaWrightDiamonds} for a summary.  It follows that the only hyperelliptic component of strata of genus zero or one that  has rank at least 2 is $\cQ(1^2, -1^2, 0^n)$. The formula for the rank of a stratum is recalled in \cite[Lemma 4.4]{ApisaWrightDiamonds}.

%

So it suffices to show that  $\For(\cQ')\neq \cQ(1^2, -1^2)$.  Suppose otherwise.
Then $\For(\cQ'_{\bfC_1'/J'})$ is a hyperelliptic genus zero stratum. The only such stratum is $\cQ(-1^4)$, so we can conclude using the following. 

\begin{sublem}
Gluing in a simple envelope or simple cylinder to a surface in $\cQ(-1^4, 0^{n'})$ cannot result in a surface in  $\cQ(1^2, -1^2, 0^{n})$.
\end{sublem} 

\begin{proof}
Gluing in a simple envelope would result in another genus 0 surface. 

If a simple cylinder is glued into a surface in $\cQ(-1^4, 0^{n'})$ and the result has two poles, then the simple cylinder must be glued into a saddle connection joining two poles. But in that case the result is in $\cQ(2,-1^2, 0^n)$. 
\end{proof}

\bold{Case 3: After perhaps re-indexing, $\cM_{\bfC_1'}'$ is a $T \times T$ locus and $\cM_{\bfC_2'}'$ is a quadratic double of a genus zero stratum.} Although we don't know $\cM'$ is geminal, we can apply Lemma \ref{L:NoRank2JandTT} to find that $\cM'$ has an equivalence class with at least 3 cylinders of equal height, and that the surfaces in $\cM'$ don't have marked points. 

Because of the lack of marked points, each cylinder on a surface in $\cM'$ lifts to a collection of cylinders of the same height on a surface in $\cM$, so this contradicts the fact that $\cM$ is geminal and concludes the proof of Lemma \ref{L:PiMinDegeneratesToPiOpt}. 
\end{proof}

\begin{cor}\label{C:MinimalIsOptimal}
If $(X, \omega) \in \cM$ is strongly generic then $\pi_{X_{min}}$ is an optimal map. Moreover, every subequivalence class contains exactly two cylinders. 
\end{cor}
\begin{proof}
To show $\pi_{X_{min}}$ is good, we must show that each cylinder $B$ is the full preimage of its image. Let $\bfC$ be any subequivalence class disjoint from and not parallel to $B$; this can be produced using Lemma \ref{L:GenericDiamond} or by more elementary arguments. Endow $\bfC$ with a choice of direction in which $\overline{\bfC}$ contains a saddle connection; we will use this direction to collapse $\bfC$.

By Lemma \ref{L:PiMinDegeneratesToPiOpt}, $\Col_{\bfC}(\pi_{X_{min}})$ is the $\cM_{\bfC}$-optimal map for $\Col_{\bfC}(X,\omega)$. In particular, $\Col_\bfC(B)$ is the full preimage of its image under $\Col_{\bfC}(\pi_{X_{min}})$. So we get the same statement for $B$. Since $B$ was arbitrary, we conclude that $\pi_{X_{min}}$ is good.

Additionally, from part \eqref{I:geminal2:TwoCylinders} of the induction hypothesis, this argument gives that the subequivalence class containing $B$ has exactly two cylinders. 

Corollary \ref{C:MinimalCover} gives that \emph{every} translation cover factors through $\pi_{X_{min}}$, so in particular all good covers factor through $\pi_{X_{min}}$, and we conclude that $\pi_{X_{min}}$ is optimal. 
\end{proof}

Since $\pi_{X_{min}}$ is $\cM$-generic whenever $(X, \omega)$ is strongly generic, this implies that $\cM$ is a locus of covers. Given a translation surface in $\cM$ we will let $\pi_{opt}$ denote the cover corresponding to the fact that $\cM$ is a locus of covers, which by Corollary \ref{C:MinimalIsOptimal} is $\cM$-optimal.  (If $(X,\omega)\in \cM$ is strongly generic, $\pi_{opt}=\pi_{X_{min}}$  by Corollary \ref{C:MinimalIsOptimal}, but for non-generic surfaces this may not be true.) We note that Theorem \ref{T:geminal2} \eqref{I:geminal2:TwoCylinders} and \eqref{I:geminal2:PiOptDegeneratesToPiOpt} both follow from Corollary \ref{C:MinimalIsOptimal} and Lemma \ref{L:PiMinDegeneratesToPiOpt} respectively.  (Lemma \ref{L:PiMinDegeneratesToPiOpt} has the assumption of ``strongly generic" and Theorem \ref{T:geminal2} does not, but that assumption can be obtained by perturbing, even in a way compatible with a degeneration.)  

Choose a generic diamond $((X, \omega), \cM, \bfC_1, \bfC_2)$, which must exist by Lemma \ref{L:GenericDiamond}, and recall from Lemma \ref{L:NoDisc} that $\ColOneX, \ColTwoX$ and $\ColOneTwoX$ are connected. 

Let $\pi$ denote $\pi_{opt}$. We will again consider the generic diamond $$\left( (X', \omega'), \cM', \bfC_1', \bfC_2' \right),$$ where $\cM'$ is the orbit closure of $(X', \omega') := \pi\left(X, \omega \right)$ and where $\bfC_i' := \pi\left( \bfC_i \right)$.



Since $\cM$ is geminal and $\pi$ is $\cM$-good (actually $\cM$-optimal by Corollary \ref{C:MinimalIsOptimal}), $\cM'$ is geminal.

\begin{sublem}\label{SL:QuadDouble}
 $\cM'$ is a quadratic double of a component $\cQ'$ of a stratum of quadratic differentials.
\end{sublem}
\begin{proof}
As before, by  parts \eqref{I:geminal2:WhatIsM} and \eqref{I:geminal2:PiOpt} of the induction hypothesis, each $\cM_{\bfC_i'}'$ is either a quadratic double of a genus zero stratum or a $T \times T$ locus. Since the rank of $\cM$ is at least two, Proposition \ref{P:TxTDiamond} gives that actually each $\cM_{\bfC_i'}'$ is a quadratic double of a genus zero stratum. By Proposition \ref{P:DiamondsAreCool}, since $\ColOneTwoX$ and hence $\Col_{\bfC_1', \bfC_2'}(X', \omega')$ is connected, we get the result. 
\end{proof}

We need to show that $\cQ'$ is genus zero stratum.  

\begin{sublem}\label{SL:Boundary}
Let $\cQ$ be a connected component of a stratum of quadratic differentials of genus greater than 0. Then, unless $\cQ = \cQ(-1^2,2),$
there is a connected codimension one cylinder degeneration that is genus greater than 0 and still has non-trivial holonomy. 
\end{sublem}
\begin{proof}
First observe that if $\cQ$ contains any marked points, then there is a codimension one cylinder degeneration that moves the marked point into a singularity of the flat metric. This allows us to assume that $\cQ$ has no marked points. 

Suppose first that $\cQ$ has genus at least 2. Pick a generic cylinder $C$ on a generic surface. If $C$ is simple or a simple envelope, degenerating it gives the result. Otherwise, Masur-Zorich \cite{MZ}, as discussed in Apisa-Wright \cite[Theorem 4.8]{ApisaWrightDiamonds}, shows that $C$ has two saddle connections on one side, and cutting them gives a component with trivial holonomy. We can find a simple cylinder on that component and degenerate it to get the result. 

So assume that $\cQ$ has genus 1. Recall that all genus 1 strata are connected. Suppose the stratum is $\cQ(-1^n, \kappa_1, \ldots, \kappa_s)$. If $n\geq 2$, one can show that there is a surface with a saddle connection joining a pole to a pole, and hence an envelope. (Thinking analytically, one can fix the elliptic curve, and collide two poles with each other to form a double pole. Before the collision occurs, one can find a  saddle connection joining the two poles.) If the envelope is simple, degenerating it proves the result. If it is a complex envelope, then cutting it out gives a surface of genus at most 1 with trivial holonomy and two boundary saddle connections. If this surface is a cylinder, we get $\cQ = \cQ(-1^2,2)$, and if a torus, we get $\cQ = \cQ(-1^2,1^2)$.  One can check directly that the former is a codimension one cylinder degeneration of the latter.
\end{proof}

\begin{sublem}\label{SL:GenusZero}
The stratum $\cQ'$ of which $\cM'$ is a quadratic double consists of genus zero surfaces.
\end{sublem}
\begin{proof}
Suppose otherwise. Since $\cM$ is a locus of covers of a stratum of surfaces of genus greater than zero, by successively applying Sublemma \ref{SL:Boundary} we can find a geminal invariant subvariety $\cN$ in the boundary of $\cM$ that is a full locus of a double of $\cQ(2, -1^2)$. By  Lemma \ref{L:DisconnectingUnderOptimalMap}, $\cN$ consist of connected surfaces.

Note that $\cN$ is geminal, and since $\cQ(2, -1^2)$ has rank one, $\cN$ has rank one. Since doubles of $\cQ(2, -1^2)$ have non-homologous twin cylinders, we get that $\cN$ is not $h$-geminal. 

By parts \eqref{I:geminalrk1:ComponentOrDouble} and \eqref{I:geminalrk1:ComponentOrDouble:PiOpt} of Proposition \ref{P:geminalrk1}, the $\cN$-optimal map for surfaces in $\cN$ is the identity. However, there are surfaces in $\cN$ admitting a $\cN$-good map of degree greater than one coming from the degeneration of the $\cM$-optimal map on surfaces in $\cM$. Therefore, we have a contradiction. 
\end{proof}

 As we now summarize, the proof of Theorem \ref{T:geminal2} is complete. 
 \begin{itemize}
 \item The case when $\cM$ has rank 1 was established by Proposition \ref{P:geminalrk1}, providing a base case for an inductive argument. Hence we assumed that $\cM$ had rank at least 2 and that Theorem \ref{T:geminal2} was true for all degenerations of $\cM$. 
\item  If an appropriate degeneration had $\pi_{opt}$ of degree 1, then the proof was given in Subsection \ref{SS:SometimesId}. 
 \item The remaining case was handled in Subsection \ref{SS:NeverId}, where Sublemmas \ref{SL:QuadDouble} and \ref{SL:GenusZero} showed that $\cM$ is described by Case \eqref{I:geminal2:CoverOfQuadDouble} of Theorem \ref{T:geminal2}, Corollary \ref{C:MinimalIsOptimal} and Sublemma \ref{SL:QuadDouble} established part \eqref{I:geminal2:CoverOfQuadDouble:PiOpt} of Theorem \ref{T:geminal2}, and Lemma \ref{L:PiMinDegeneratesToPiOpt} and Corollary \ref{C:MinimalIsOptimal} established parts \eqref{I:geminal2:TwoCylinders} and \eqref{I:geminal2:PiOptDegeneratesToPiOpt} of Theorem \ref{T:geminal2}.\qedhere
 \end{itemize}
\end{proof}

\subsection{An open problem}\label{SS:Open}
We end the paper with a conjecture on geminal orbit closures. 

\begin{conj}\label{Conjecture1}
The optimal covering map for a geminal orbit closure is a cyclic cover.
\end{conj}

We anticipate that a corollary of Conjecture \ref{Conjecture1} would be a classification of geminal orbit closures and that the crux would be establishing the conjecture in the rank one rel zero case. The problem of classifying the rank one rel zero geminal subvarieties has a simple group theoretic formulation, which we will now sketch.


By Proposition \ref{P:geminalrk1} parts \eqref{I:geminalrk1:ComponentOrDouble} and \eqref{I:geminalrk1:ComponentOrDouble:PiOpt} it suffices to classify the $h$-geminal rank one rel zero subvarieties $\cM$ whose optimal maps have degree greater than one. Since in rank 1 rel 0, $T\times T$ loci are quadratic doubles of $\cQ(-1^4)$, we know that the optimal map has codomain  contained in such a double. Denote the optimal map 
$$\pi_{opt}: (X,\omega) \to (X', \omega').$$ 

For such subvarieties, every cylinder has a twin by Proposition \ref{P:geminalrk1} (\ref{I:geminalrk1:TwoCylinders}). Thus, the quadratic double must have at least three preimages of poles marked. 

Let $d$ denote the degree of the optimal map. Composing the optimal map with the quotient by two torsion, we get a degree $4d$ map from $(X,\omega)$ to a surface $(X'',\omega'')$ in $\cH(0)$. We can obtain from this an unbranched cover of a torus with one punctured point, so it will be helpful to let $E$ be the once punctured torus obtained from $(X'',\omega'')$ by deleting the marked point, and to let $E'$ denote the four times punctured torus obtained from  $(X', \omega')$ by deleting the preimages of poles. 

The unbranched cover of $E$ obtained from the degree $4d$ map $$(X,\omega)\to (X'', \omega'')$$ is specified by a representation $\rho: F_2 \rightarrow \mathrm{Sym}(4d)$. Here and in the remainder of this subsection we identify $F_2$ with the fundamental group of the once punctured flat torus $E$. The following is a standard result whose proof we only sketch here. 

\begin{lem}\label{L:GeneratorsCyl}
Every element of a size two generating set of $F_2 \cong \pi_1(E)$ is freely homotopic to a cylinder core curve.
\end{lem}
\begin{proof}
Let $S$ be the collection of elements of $F_2$ that are part of a size two generating set. The result follows since $\mathrm{Aut}(F_2)$ acts transitively on $S$ and since $\mathrm{Out}(F_2)$ is isomorphic to $\mathrm{SL}(2, \mathbb{Z})$, which acts on the torus by affine maps. 
%
%
\end{proof}

The assumptions that $\cM$ is geminal and that the map to $(X',\omega')$ is good (in fact optimal) is equivalent to saying that the core curves of cylinders in $F_2$ map to $(2d,2d)$-cycles under $\rho$. Since every element of a size two generating set of $F_2$ is conjugate to the core curve of a cylinder and since cycle type is the (unique) conjugacy class invariant for elements of the symmetric group, we see that every element of a size two generating set of $F_2$ maps to a $(2d,2d)$-cycle under $\rho$. Any homomorphism $\rho: F_2 \ra \mathrm{Sym}(4d)$ with this property will be called a \emph{geminal representation}.

Let $\mathrm{Fix}(1)$ denote the permutations in $\mathrm{Sym}(4d)$ that fix $1$ (here we think about $\mathrm{Sym}(4d)$ as being permutations of the set $\{1, \hdots, 4d\}$). Let $K$ be the kernel of the map $F_2 \ra \bZ/2\bZ \times \bZ/2\bZ$ given by abelianizing and then taking a mod $2$ quotient. Then the covering maps $$(X, \omega) \ra (X', \omega') \ra (X'', \omega'')$$ correspond to the inclusion of subgroups $$\rho^{-1}\left( \mathrm{Fix}(1) \right) \subseteq K \subseteq F_2.$$ Therefore, we hope that Conjecture \ref{Conjecture1} could be established if the following problem could be resolved.

\begin{prob}
Suppose that $\rho: F_2 \ra \mathrm{Sym}(4d)$ is geminal. Is $\rho^{-1}\left( \mathrm{Fix}(1) \right)$ a normal subgroup of $K$ and, if so, is $K/\rho^{-1}\left( \mathrm{Fix}(1) \right)$ cyclic? 
\end{prob}

\bibliography{mybib}{}
\bibliographystyle{amsalpha}
\end{document}